\documentclass[a4,12pt,epsfig]{article}
\usepackage{graphicx}
\usepackage{subfig}
\usepackage{amsmath}
\usepackage{amssymb}
\usepackage{amsfonts}
\usepackage{euscript}
\begin{document}
\bibliographystyle{plain}
\newcommand{\bea}{\begin{eqnarray}}
\newcommand{\eea}{\end{eqnarray}}
\newcommand{\bfmN}{{\mbox{\boldmath{$N$}}}}
\newcommand{\bfmx}{{\mbox{\boldmath{$x$}}}}
\newcommand{\bfmv}{{\mbox{\boldmath{$v$}}}}
\newcommand{\se}{\setcounter{equation}{0}}
\newtheorem{corollary}{Corollary}[section]
\newtheorem{example}{Example}[section]
\newtheorem{definition}{Definition}[section]
\newtheorem{theorem}{Theorem}[section]
\newtheorem{proposition}{Proposition}[section]
\newtheorem{lemma}{Lemma}[section]
\newtheorem{remark}{Remark}[section]
\newtheorem{result}{Result}[section]
\newcommand{\vtwo}{\vskip 4ex}
\newcommand{\vthree}{\vskip 6ex}
\newcommand{\vfour}{\vspace*{8ex}}
\newcommand{\hone}{\mbox{\hspace{1em}}}
\newcommand{\hon}{\mbox{\hspace{1em}}}
\newcommand{\htwo}{\mbox{\hspace{2em}}}
\newcommand{\hthree}{\mbox{\hspace{3em}}}
\newcommand{\hfour}{\mbox{\hspace{4em}}}
\newcommand{\von}{\vskip 1ex}
\newcommand{\vone}{\vskip 2ex}
\newcommand{\n}{\mathfrak{n} }
\newcommand{\m}{\mathfrak{m} }
\newcommand{\q}{\mathfrak{q} }
\newcommand{\aF}{\mathfrak{a} }

\newcommand{\kl}{\mathcal{K}}
\newcommand{\p}{\mathcal{P}}
\newcommand{\Lt}{\mathcal{L}}
\newcommand{\bv}{{\mbox{\boldmath{$v$}}}}
\newcommand{\bc}{{\mbox{\boldmath{$c$}}}}
\newcommand{\bx}{{\mbox{\boldmath{$x$}}}}
\newcommand{\br}{{\mbox{\boldmath{$r$}}}}
\newcommand{\bs}{{\mbox{\boldmath{$s$}}}}
\newcommand{\bb}{{\mbox{\boldmath{$b$}}}}
\newcommand{\ba}{{\mbox{\boldmath{$a$}}}}
\newcommand{\bn}{{\mbox{\boldmath{$n$}}}}
\newcommand{\bp}{{\mbox{\boldmath{$p$}}}}
\newcommand{\by}{{\mbox{\boldmath{$y$}}}}
\newcommand{\bz}{{\mbox{\boldmath{$z$}}}}
\newcommand{\be}{{\mbox{\boldmath{$e$}}}}
\newcommand{\proof}{\noindent {\sc Proof :} \par }
\newcommand{\bP}{{\mbox{\boldmath{$P$}}}}

\newcommand{\M}{\mathcal{M}}
\newcommand{\R}{\mathbb{R}}
\newcommand{\Q}{\mathbb{Q}}
\newcommand{\Z}{\mathbb{Z}}
\newcommand{\N}{\mathbb{N}}
\newcommand{\C}{\mathbb{C}}
\newcommand{\xar}{\longrightarrow}
\newcommand{\ov}{\overline}
 \newcommand{\rt}{\rightarrow}
 \newcommand{\om}{\omega}
 \newcommand{\wh}{\widehat }
 \newcommand{\wt}{\widetilde }
 \newcommand{\g}{\Gamma}
 \newcommand{\lm}{\lambda}

\newcommand{\eN}{\EuScript{N}}
\newcommand{\ncom}{\newcommand}
\newcommand{\norm}{\|\;\;\|}
\newcommand{\inp}[2]{\langle{#1},\,{#2} \rangle}
\newcommand{\nrm}[1]{\parallel {#1} \parallel}
\newcommand{\nrms}[1]{\parallel {#1} \parallel^2}
\title{On the convergence of quasilinear viscous approximations using Compensated Compactness and Kinetic Formulation}
\author{ Ramesh Mondal\\ Email: rmondal86@gmail.com\\ Department of Mathematics, University of Kalyani, Kalyani, India-741235.}
\maketitle{}
\begin{abstract}
We use the method of Compensated Compactness and Kinteic Formulation to show that the almost everywhere limit of quasilinear viscous approximations is the unique entropy
solution (in the sense of {\it F. Otto}) of the corresponding scalar conservation laws on a bounded domain in $\mathbb{R}^{d}$, where the viscous term is
of the form $\varepsilon\,div\left(B(u^{\varepsilon})\nabla u^{\varepsilon}\right)$.
 \end{abstract}
 \section{Introduction}
 Let $\Omega$ be a bounded domain in $\mathbb{R}^{d}$ with smooth boundary $\partial \Omega$. For $T >0$, denote $\Omega_{T}:= \Omega\times(0,T)$. 
 We write the initial boundary value problem $\left(\mbox{IBVP}\right)$ for scalar conservation laws given by
 \begin{subequations}\label{ivp.cl}
\begin{eqnarray}
  u_t + \nabla\cdot f(u) =0& \mbox{in }\Omega_T,\label{ivp.cl.a}\\
u(x,t)= 0&\mbox{on}\,\,\partial \Omega\times(0,T),\label{ivp.cl.b}\\
  u(x,0) = u_0(x)& x\in \Omega.\label{ivp.cl.c}
  \end{eqnarray}
\end{subequations}
where $f=(f_{1},f_{2},\cdots,f_{d})$ is the flux function and $u_{0}$ is the initial condition.\\
Consider the IBVP for the generalized viscosity problem 
\begin{subequations}\label{regularized.IBVP}
\begin{eqnarray}
 u^\varepsilon_{t} + \nabla \cdot f(u^{\varepsilon}) = \varepsilon\,\nabla\cdot\left(B(u^\varepsilon)\,\nabla u^\varepsilon\right)
 &\mbox{in }\Omega_{T},\label{regularized.IBVP.a} \\
    u^\varepsilon(x,t)= 0&\,\,\,\,\mbox{on}\,\, \partial \Omega\times(0,T),\label{regularized.IBVP.b}\\
u^{\varepsilon}(x,0) = u_{0\varepsilon}(x)& x\in \Omega,\label{regularized.IBVP.c}
\end{eqnarray}
\end{subequations}
indexed by $\varepsilon>0$. The problem of the form \eqref{regularized.IBVP.a} is posed in \cite{MR2150387}. The aim of this article is to prove that the {\it a.e.} limit of sequence of solutions $\left(u^{\varepsilon}\right)$ to \eqref{regularized.IBVP}(called quasilinear viscous approximations) is the unique entropy solution for IBVP \eqref{ivp.cl} in the sense of Otto \cite{MR1387428}. \\
The equation of the form \eqref{regularized.IBVP.a} appears in many physical systems. For instances, it appears in viscous shallow water problem \cite{GQCHEN} and in the equations of gas dynamics for viscous, heat conducting fluid in eulerian coordinates \cite{JSmoller}. Apart from above mentioned physical systems, the equation of the form \eqref{regularized.IBVP.a} is very important in the study of numerical analysis. Kurganov and  Liu \cite{AKurganov} proposed a finite volume method for solving general multidimentional system of conservation laws. In which they introduce an adaptive way of adding viscosity in the shock region. As a result the scheme captures numerically stable solution of hyperbolic conservation laws. The coefficient of the added numerical viscosity appears as a function of the solution and hence the convergence of the scheme to the entropy solution is similar to the problem considered in this article. Mishra and Spinolo \cite{SMishra} have designed schemes for systems in one space dimension that incorporate the explicit information about the underlying viscous operators. Consequently, these schemes approximate the viscosity solution of conservation laws.\\
\vspace{0.1cm}\\
Let us write the hypothesis on $f,\,B,\,$ and $u_{0}$. \\
\noindent{\bf Hypothesis D}
\begin{enumerate}
	\item[(a)] Let $f\in \left(C^{2}(\mathbb{R})\right)^d$, $f^\prime\in \left(L^\infty(\mathbb{R})\right)^d$, and denote 
	$$\|f^\prime\|_{\left(L^\infty(\mathbb{R})\right)^d}:=\displaystyle\max_{1\leq j\leq d}\left(\sup_{y\in\mathbb{R}}|f^\prime(y)|\right).$$
	\item[(b)] Let $B\in C^1(\mathbb{R})\cap L^\infty(\mathbb{R})$, and there exists an $r>0$ such that $B\geq r$.
	\item[(c)] Let $u_{0}$ be in $ L^{\infty}(\Omega)$. There exists a sequence $\left(u_{0\varepsilon}\right)\in\mathcal{D}(\Omega)$ and a constant $A>0$ such that $\|u_{0\varepsilon}\|\leq A$. Denote $I:=[-A, A]$. 
	\item[(d)] Let $1<p\leq 2$ and $p^{\prime}>1$ be such that $\frac{1}{p}+\frac{1}{p^{\prime}}=1$. Let $\psi\in L^{p^{\prime}}\left(\mathbb{R}\right)$ having essential compact support.
	Let $f$ satisfy the following 
	\begin{equation*}
	\begin{split}
	\mbox{meas}\Big\{c\in\mbox{supp}\,\psi,\,\,\tau+\,\left(f_{1}^{\prime}(c),f_{2}^{\prime}(c),\cdots,f_{d}^{\prime}(c)\right)\cdot\xi=0\Big\}=0\,\, \\
	\mbox{for all}\,\left(\tau,\xi\right)\in\mathbb{R}\times\mathbb{R}^{d}\,\, \mbox{with}\,\,\tau^{2} +\left|\xi\right|^{2}=1.
	\end{split}
	\end{equation*}
\end{enumerate}
We now give another hypothesis on $f,\,B,\,u_{0}$.\\
\noindent{\bf Hypothesis E:}
\begin{enumerate}
	\item[(a)] Let $f\in C^2(\R)$, $f^\prime\in L^\infty(\R)$, and denote 
	$$\|f^\prime\|_{L^\infty(\R)}:=\sup_{y\in\R}|f^\prime(y)|.$$
	\item[(b)] Let $B\in C^1(\R)\cap L^\infty(\R)$, and there exists an $r>0$ such that $B\geq r$.
	\item[(c)] Let $u_{0}$ be in $ L^{\infty}(\Omega)$. There exists a sequence $\left(u_{0\varepsilon}\right)\in\mathcal{D}(\Omega)$ and a constant $A>0$ such that $\|u_{0\varepsilon}\|\leq A$. Denote $I:=[-A, A]$. 
	\item[(d)] Let $d=1$ and the set $\left\{z\in\mathbb{R}\,\,:\,\,f^{\prime\prime}(z)\neq 0\right\}$ is dense in $\mathbb{R}$,\,{\it i.e.}, the derivative of the flux function $f$ satisfies the following property: Restriction of 
	$f^{\prime}:\mathbb{R}\to\mathbb{R}$ to any non-degenerate interval is a nonconstant function.
\end{enumerate}
\vspace{0.1cm}
We now give the final hypothesis on $f,\,B,\,u_{0}$.\\
\textbf{Hypothesis F:}
\begin{enumerate}
	\item[(a)] Let $f\in \left(C^2(\R)\right)^d$, $f^\prime\in \left(L^\infty(\R)\right)^d$, and denote 
	$$\|f^\prime\|_{\left(L^\infty(\R)\right)^d}:=\max_{1\leq j\leq d}\,\sup_{y\in\R}|f^\prime_j(y)|.$$
	\item[(b)] Let $B\in C^1(\R)\cap L^\infty(\R)$, and there exists an $r>0$ such that $B\geq r$.
	\item[(c)] Let $d=2$ and the functions $f_{1}^{\prime},\,f^{\prime}_{2}:\mathbb{R}\to\mathbb{R}$ are linearly independent, {\it i.e.,} for all 
	$\xi\in\mathbb{R}^{2}$ with $|\xi|=1$, the following quantity 
	$$S(\xi,\cdot):=\xi_{1}f_{1}^{\prime} +\xi_{2}f_{2}^{\prime},$$
	is not zero indentically on $\mathbb{R}$.
	\item[(d)] $u_0\in BV_0(\Omega)\cap L^\infty(\Omega)$, where $BV_0(\Omega)$ denotes the space of functions of bounded variation, having zero trace on $\partial\Omega$ (see \cite{PhucT2017} for details), and $u_{0\varepsilon}$ as constructed in the next section having the properties stated in Lemma~\ref{regularized.lem1}. and we denote 
	$I:=[-A, A]$.
\end{enumerate}
For hyperbolic problems posed on bounded domain, it is well known that any smooth solution is constant on any maximal segment of characteristic. We assume that this segment intersects $\Omega\times\left\{0\right\}$ and $\partial\Omega\times (0,T)$. The informations which propagate through this segment from initial datum to the boundary $\partial\Omega\times (0,T)$ may not match with the prescribed boundary condition. In other words IBVP for hyperbolic problems are ill-posed in general \cite{MR542510}. Therefore it is incorrect to assume boundary data pointwise. Thus we need to give a meaning to the way in which the boundary conditions are realized. In BV-setting, Bardos-LeRoux-Nedelec prove the existence, uniqueness of entropy solutions by introducing a formulation of entropy solutions ( BLN entropy condition) including boundary term. BLN entropy condition gives a way to interpret the boundary conditions whenever the entropy solution is of bounded variations. Since the BLN condition uses the trace of the solution on the boundary $\partial\Omega\times (0,T)$, it is not relevant for all those solutions which are not of bounded variations. Otto \cite{MR1387428} generalizes the BLN entropy condition to the $L^{\infty}-$ setting by giving a new concept of entropy solutions. In this approach Otto introduces ``boundary entropy-entropy flux" pair and ask the boundary condition to hold in an integral form. In the $L^{\infty}-$ setting Otto also proves the existence, uniqueness of entropy solution satisfying the concept of entropy introduced by him. Therefore Otto's entropy condition gives another way to interpret the boundary conditions. For extensions of the concepts of entropy solution to such contexts, we refer the reader to Martin \cite{Martin}, Carrillo \cite{Carrillo}, Vallet \cite{Vallet} in the $L^{\infty}$-setting, and to Porretta and Vovelle \cite{Porretta}, Ammar et al. \cite{Ammar} in $L^{1}$-setting. For system of hyperbolic equations, Dubois and LeFloch \cite{Dubois_Lefloch} generalizes the boundary-entropy inequality of Bardos-LeRoux-Nedelec by introducing a boundary entropy inequality in the domain $\left\{(x,t)\,:\,x>0,\,t>0\right\}$ for general viscosity term as considered in this article.
\vspace{0.1cm}\\
In this context, we have the following three main results. The main differences of these results are on the assumptions on the initial data $u_{0}$ and the space dimension $d$ in Hypothesis D, Hypothesis E, Hypothesis F. 
\begin{theorem}\label{paper2.compensatedcompactness.theorem1}
 {\rm Let $f,\,B,\,u_{0}$ satisfy Hypothesis D. Then the {\it a.e.} limit of the quasilinear viscous approximations 
 $\left(u^{\varepsilon}\right)$ is the unique entropy solution of IBVP \eqref{ivp.cl} in the sense of {\it Otto}\cite{MR1387428}.}
\end{theorem}
\begin{theorem}\label{paper2.compensatedcompactness.theorem2}
 {\rm Let $f,\,B,\,u_{0}$ satisfy Hypothesis E. Then the {\it a.e.} limit of the quasilinear viscous approximations 
 $\left(u^{\varepsilon}\right)$ is the unique entropy solution of IBVP \eqref{ivp.cl} in the sense of {\it Otto}\cite{MR1387428}.}
\end{theorem}
\begin{theorem}\label{paper2.compensatedcompactness.theorem3}
{\rm	Let $f,\,B,\,u_{0}$ satisfy Hypothesis F. Then the {\it a.e.} limit of the quasilinear viscous approximations 
	$\left(u^{\varepsilon}\right)$ is the unique entropy solution of IBVP \eqref{ivp.cl} in the sense of {\it Otto}\cite{MR1387428}.}
\end{theorem}
In order to extract the {\it a.e.} convergent subsequence of the quasilinear viscous approximations, we use the method of Compensated Compactness for space dimension $d=1$ and $d=2$ and the Kinetic formulation for the multidimensional case. But compactness arguments based on a priori BV bound is the most effective tool for studing the conservation laws. In general, it is difficult to get a priori BV estimates for sequence of approximate solutions $\left(u^{\varepsilon}\right)$. We have establied the BV-estimate of quasiliear viscous approximations $\left(u^{\varepsilon}\right)$ in \cite{Ramesh} for BV initial data. In this work, we use compensated compactness and kinetic formulation as we mainly focus on $L^{\infty}\left(\Omega\right)$ initial data.\\
We now briefly review the literatures on the method of Compensated Compactness in the context of the the problem \eqref{regularized.IBVP.a}. 
 Let $\Omega=\mathbb{R},\, T=\infty$. For $B(.)\equiv 1$, Tartar  \cite{Dacorogna},\cite{MR584398} established the almost everywhere convergence of a subsequence of the approximate solutions $\left(u^{\varepsilon}\right)$ using theory of compensated compactness in one space dimension. In \cite{Markati2}, Markati considers the viscosity term of the form $\varepsilon\left(B(u^{\varepsilon})\,u_{x}\right)_{x}$ and prove that the {\it a.e.} limit of the viscosity approximations is the unique Kruzhkov's entropy solution to the corresponding scalar conservation laws. In \cite{Markati1}, the Markati and Natalini study viscosity problem with degenarate gradient dependent viscosity of the form $\varepsilon\left(B(u_{x})\right)_{x}$. They also prove the convergence of the viscosity approximations to the unique Kruzhkov's entropy solution of  scalar conservation laws using Compensated Compactness. We use the similar analysis from Tartar \cite{Dacorogna},\cite{MR584398} to conclude the existence of {\it a.e.} convergent subsequence in the bounded domain $\Omega$ of $\mathbb{R}$. The hard to get BV a priori estimates difficulty is replaced by $L^{2}-$type compact entropy production by Tartar in the context of conservation laws.  An existence result is proved in \cite[p.66]{Dacorogna} for $u_{0}\in W^{1,\infty}(\Omega)$. In one space dimension, we use initial data $u_{0}$ in $L^{\infty}(\Omega)$ for the proof of existence of almost everywhere convergent subsequence. We use a similar analysis from \cite{Tadmor} to conclude the extraction of {\it a.e.} convergent subsequence for space dimension $d=2$. These extraction of almost everywhere convergent subsequence is independent of spatial a priori BV bounds of $\left(u^{\varepsilon}\right)$. Since we work with bounded domain $\Omega$ of $\mathbb{R}^{d}$, we prove that the {\it a.e.} limit is the unique entropy solution of corresponding scalar conservation laws in the sense of Otto \cite{MR1387428}. \\ 
We briefly review the literatures on Kinetic formulations and Kinetic techniques of scalar conservation laws. Let $\Omega=\mathbb{R}^{d}$ and $T=\infty$. In \cite{B.Parthame}, Perthame and Tadmor consider a Boltzmann-like kinteic equation and show that the local density of a kinetic particle admits a continuum limit as it converges strongly with $\varepsilon\to 0$ to the unique Kruzhkov entropy solution of the corresponding scalar conservation laws using BV- compactness arguement in several space dimension and using compensated compactness arguement in one space dimension. In \cite{Lions}, Tadmor {\it et al.} introduced a Kinetic formulation for multidimensional scalar conservation laws and proved that sequence of solutions of the cauchy problem to scalar conservations laws and sequence of approximations to scalar conservations laws are $L^{1}\left(\mathbb{R}^{d}\times (0,\infty)\right)-$ compact using Velocity averaging lemmata. Let $\Omega$ bounded domain. In \cite{CImbert}, Imbert and Vovelle prove a comparison principle between any weak entropy subsolution and any weak entropy supersolution of the initial-boundary value problem using Kinetic formulation and Kinetic techiniques . As an application of comaprism principle, they study that the Bhatnagar–Gross–Krook-like kinetic model that approximates the scalar conservation law. For $B(.)\equiv 1$, In \cite{JDroniou},  Droniou {\it et al.} prove that the rate of convergence of the viscous approximation to the weak entropy solution is of order $\varepsilon^{\frac{1}{3}}$ , where $\varepsilon$ is the size of the artificial viscosity. We refer the reader to \cite{ZWang} for a rate of convergence result using kinetic formulation and kinetic technique.\\
\vspace{0.1cm}\\
In this article, we establish the existence of an almost everywhere convergent subsequence of the approximate solutions $\left(u^{\varepsilon}\right)$ using Kinetic Formulation and kinetic technique introduced in \cite{Lions} in any space dimension. For  $\Omega=\mathbb{R}^{d},\,T=\infty\,\,\mbox{and}\,B(.)\equiv 1$, Tadmor {\it et al.} \cite{Tadmor} apply Velocity Averaging lemma to the kinetic formulation of \eqref{regularized.IBVP.a} as a multidimensional spatial kinetic formulation. As a consequence, they use that the time derivative of the pseudo maxwellian is bounded in the space of measure $\mathcal{M}\left(\mathbb{R}^{d}\times\mathbb{R}\times\left(0,\infty\right)\right)$,\,{\it i.e.}, $\left(\frac{\partial\chi^{\varepsilon}}{\partial t}\right)$ is bounded in $\mathcal{M}\left(\mathbb{R}^{d}\times\mathbb{R}\times\left(0,\infty\right)\right)$. This is possible if the initial data are BV functions as shown in \cite{Tadmor}. Therefore the extraction of almost everywhere convergent subsequence of $\left(u^{\varepsilon}\right)$ is proved in \cite{Tadmor} for BV initial data. But we use Velocity Averaging lemma directly from \cite{Lions} to kinetic formulation of \eqref{regularized.IBVP.a} as space-time separate version, we are able to establish the existence of {\it a.e.} convergent subsequence of $\left(u^{\varepsilon}\right)$ for $L^{\infty}\left(\Omega\right)$ initial data and $B$ as mentioned in Hypothesis D. This proof relies on Velocity Averaging lemma. These results are proved in Section \ref{kineticformulation.generalizedviscosityproblem}. We use Velocity Averaging lemma on the domain $\mathbb{R}^{d}\times\mathbb{R}\times (0,\infty)$. But the sequence of measures 
$$m^{\varepsilon}(x,c,t):= \varepsilon\,\eta^{\prime}\left(u^{\varepsilon};c\right)\displaystyle\sum_{j=1}^{d}\frac{\partial}{\partial x_{j}}\left(B(u^{\varepsilon})\,\frac{\partial u^{\varepsilon}}{\partial x_{j}}\right)$$
is defined on $\Omega\times\mathbb{R}\times (0,T)$. Therefore firstly, we prove that for every $M>0$,  $\left(m^{\varepsilon}\right)$ is bounded in $\mathcal{M}\left(\Omega\times\left(-M,M\right)\times(0,T)\right)$. Secondly, applying Reisz-reprsentation theorem of dual spaces, we extend the sequence of measures $\left(m^{\varepsilon}\right)$ to the space $\mathbb{R}^{d}\times\mathbb{R}\times (0,\infty)$ which preserves the total variations. As a result, we can use a standard result from \cite{Lions} to get that the derivative of this extended sequence of measures can be expressed as the composition of inverses of Riesz potentials. This is one of the requirements to apply Velocity Averaging lemma. Therefore it is a combination of arguements of the domains $\Omega\times\mathbb{R}\times (0,T)$ and $\mathbb{R}^{d}\times\mathbb{R}\times (0,\infty)$. These arguements are connected through Riesz represtation theorem. \\  In order to apply Velocity Averaging lemma, we need to prove that the sequence of measures $\left(m^{\varepsilon}\right)$ is bounded in $\mathcal{M}\left(\mathbb{R}^{d}\times\mathbb{R}\times (0,\infty)\right)$. We observe from \cite{Lions} that the sequence of measures $\left(m^{\varepsilon}\right)$ can be interpreted as the following Kruzhkov entropy dissipation
$$m^{\varepsilon}(x,c,t)=\frac{\partial}{\partial t}\left(|u^{\varepsilon}-c|\right) + \displaystyle\sum_{j=1}^{d}\,\frac{\partial}{\partial x_{j}}\left(sg(u^{\varepsilon}-c)\,\left(f_{j}(u^{\varepsilon})-f_{j}(c)\right)\right)$$
in the sense of distribution $\mathcal{D}^{\prime}\left(\Omega\times (0,T)\right)$ (See Section \ref{kineticformulation.generalizedviscosityproblem}). We notice in the entropy dissipation form that the second order nonlinear term on the RHS of \eqref{regularized.IBVP.a} can be expressed as a sum of first order derivatives of ``Kruzhkov entropy-entropy fluxes" in the sense of distributions. Then applying the boundedness of the sequence $\left(\displaystyle\sum_{j=1}^{d}\left(\sqrt{\varepsilon}\,\left\|\frac{\partial u^{\varepsilon}}{\partial x_{j}}\right\|_{L^{2}\left(\Omega_{T}\right)}\right)\right)$ and Step 2 of Lemma \ref{MeasureBounded.1} from Subsection \ref{Boundedness.sequenceofmeasures}, we can show that $\left(m^{\varepsilon}\right)$ are bounded sequence of measures. Instead of using entropy dissipation form, we work directly with the sequence of measures $m^{\varepsilon}(x,c,t):= \varepsilon\,\eta^{\prime}\left(u^{\varepsilon};c\right)\displaystyle\sum_{j=1}^{d}\frac{\partial}{\partial x_{j}}\left(B(u^{\varepsilon})\,\frac{\partial u^{\varepsilon}}{\partial x_{j}}\right)$. We give a different proof of boundedness of sequence of measures $\left(m^{\varepsilon}\right)$ by using $\left(sg(x)\,tanh\,\left(n\,|x|\right)\right)$ as the sequence of approximations for the sign function $sg(x)$. This is motivated by the BV-estimate result from our previous work \cite{Ramesh} as in the case of BV-estimate, we use $\left(tanh\,nx\right)$ as the approximations of $\left(\mbox{sg}\left(x\right)\right)$. Then we express the second order nonlinear term present in $\left(m^{\varepsilon}\right)$ as a sum of first order derivatives in the sense of distributions by clever computations and using the property
$$\displaystyle\lim_{n\to\infty}\,tanh^{\prime}\,\left(n\,|u^{\varepsilon}-c|\right)=\displaystyle\lim_{n\to\infty}\,tanh^{\prime}\,\left(n\,\left|\frac{\partial u^{\varepsilon}}{\partial x_{j}}\right|\right).$$  
This is the idea of the entropy dissipation property. This result is proved in Lemma \ref{MeasureBounded.1}. \\ 
\vspace{0.1cm}\\
In Section \ref{Section.F.Otto.entropysolution}, we derive that the {\it a.e.} limit of a subsequence of approximate solutions $\left(u^{\varepsilon}\right)$ is the unique entropy solution in the sense of Otto \cite{MR1387428}. For $B(.)\equiv 1$, in $L^{\infty}-$ setting, the existence of entropy solution is proved using vanishing viscosity method. In this proof, an inequality \cite[p.113]{Necas} satisfied by ``boundary entropy-entropy flux" pairs is proved using an appropriate multiplier (see \cite[p.129]{Necas} for the multiplier). This inequality is equivalent to the definition of Otto's entropy solution \cite[p.113]{Necas}. As a consequence, existence of entropy solution is proved in the sense of Otto. But we prove that the {\it a.e.} limit of sequence of quasiliear viscous approximations is the Otto's entropy solution for $B$ as mentioned in Hypothesis D in $L^{\infty}-$ setting by proving Lemma \ref{Boundaryinequality.F.Ottolem41}. In proving Lemma \ref{Boundaryinequality.F.Ottolem41}, we have used the properties of ``boundary entropy-entropy flux" pair and introduced clever computations. To prove Lemma \ref{Boundaryinequality.F.Ottolem41}, it is equivalent to prove an inequality which is equivalent to definition of Otto's entropy solution of \cite[p.113]{Necas} satisfied by  ``boundary entropy-entropy flux" pairs mentioned in Step 3 in the proof of Theorem \ref{paper2.compensatedcompactness.theorem1}. The only difference in our inequality with the inequality of \cite[p.113]{Necas} is that the Lipschitz constant $M$ is replaced by another Lipschitz constant $M\,d$ in our case.
 Therefore, we finish establishing that the {\it a.e.} limit of quasilinear viscous approximations is the unique entropy solution in the sense of  Otto whenever the initial data in $L^{\infty}\left(\Omega\right)$ and for $B$ as mentioned in Hypothesis D.\\ 
\vspace{0.1cm}\\
The plan for the paper are the following. In Section 2, we state results about existence, uniqueness of solutions of \eqref{regularized.IBVP}, estimates results of solutions and its dervatives. In Section 3, we extract {\it a.e.} convergent subsequence of $\left(u^{\varepsilon}\right)$ using Compensated Compactness for space dimension $d=1$ and $d=2$. In Section 4, we establish the existence of an {\it a.e.} convergent subsequence of $\left(u^{\varepsilon}\right)$ using Kinetic Formulation in any space dimension. In Section 5, we deriave that the {\it a.e.} limit of $\left(u^{\varepsilon}\right)$ is the unique entropy solution in the sense of Otto \cite{MR1387428}.
\section{Existence, uniqueness, maximum principle and derivative estimates}
\subsection{Existence and Uniqueness of Solutions}
We begin this section by constructing sequences of approximations $\left(u_{0\varepsilon}\right)$ for the initial data $u_{0}$. The sequence $\left(u_{0\varepsilon}\right)$ mentioned in \textbf{Hypothesis D} are constructed in the following result in view of the discussions from \cite[p.31-p.35]{MR2309679}. We prove the result for any space dimension $d\in\mathbb{N}$. 	
\begin{lemma}\label{hypothesisDinitialdatalem}
	Let $u_{0}\in L^{\infty}(\Omega)$ and $1\leq p<\infty$. Then there exists a sequence $\left(u_{0\varepsilon}\right)$ in $\mathcal{D}(\Omega)$ and $A> 0$ such that $\|u_{0\varepsilon}\|\leq A$ and $u_{0\varepsilon}\to u_{0}$ in $L^{p}(\Omega)$ as $\varepsilon\to 0$. 	
\end{lemma}	
\begin{proof}
	We prove Lemma \ref{hypothesisDinitialdatalem} in two steps. In Step 1, we show the construction of $\left(u_{0\varepsilon}\right)$	and prove its convergence to $u_{0}$ in $L^{p}(\Omega)$. In Step 2, we show the existence of a constant $A> 0$ such that $\|u_{0\varepsilon}\|\leq A$.
	\vspace{0.1cm}\\
	\textbf{Step 1:} Denote
	$$P:=\left\{y\in\mathbb{R}^{d}:\,\,\left|y_{i}\right|< 1,\,i=1,2,\cdots,d\right\},\, P^{+}:=\left\{y\in P:\,\,y_{d}> 0\right\},\,\,
	\Gamma:=\left\{y\in P:\,\,y_{d}=0\right\}.$$
	Since $\partial\Omega$ is smooth, for a given point $x_{0}\in\partial\Omega$, there exists a neighbourhood 
	$U_{0}$ of $x_{0}$ and a smooth invertible mapping $\Psi_{0}: P\to U_{0}$ such that 
	$$\Psi_{0}\left(P^{+}\right)= U_{0}\cap\Omega,\,\, \Psi_{0}\left(\Gamma\right)= U_{0}\cap\partial\Omega.$$
	Let $\eta_{0}\in\mathcal{D}(U_{0})$. We consider $\eta_{0} u_{0}$ instead of $u_{0}$ in the neighbourhood $U_{0}$ of $x_{0}$ . 
	Define $v:P^{+}\to\mathbb{R}$ by
	$$v(y):=\left(\eta u_{0}\right)\left(\Psi(y)\right).$$ 
	Since $\eta_{0}\in\mathcal{D}(U_{0})$, therefore $v\equiv 0$ in a neighbourhood of the upper boundary and the lateral boundary of $P^{+}$. 
	Let $\tilde{v}$ be the extension 
	of $v$ to $P$ by setting $\tilde{v}=0$ in $P\setminus P^{+}$. Let $\tilde{\rho}_{\varepsilon}:\mathbb{R}\to\mathbb{R}$
	be the standard sequence of mollifiers.  Define $J_{\varepsilon}\tilde{v}(y): P^{+}\to\mathbb{R}$ by
	\begin{eqnarray}\label{initialdata.h1.eqn3AA}
	J_{\varepsilon}\tilde{v}(y) :=\int_{P}\tilde{\rho}_{\varepsilon}(y_{1}-z_{1})\,\tilde{\rho}_{\varepsilon}(y_{2}-z_{2})\,\cdots
	\tilde{\rho}_{\varepsilon}(y_{d-1}-z_{d-1})\tilde{\rho}_{\varepsilon}(y_{d}-z_{d}-2\varepsilon)\,\tilde{v}(z)\,dz.\nonumber\\
	{}
	\end{eqnarray}
	For clarity, we only show that 
	$\left(J_{\varepsilon}\tilde{v}\right)$ has compact support in $P^{+}$. Since $\tilde{v}$ is zero in a neighbourhood of the upper 
	boundary and the lateral boundary of $P^{+}$, we only show that $J_{\varepsilon}\tilde{v}$ is zero in $\left\{y\in P^{+};\,\, 0<y_{d}<\varepsilon\right\}$. We know that 
	$\rho_{\varepsilon}(y_{d}-z_{d}-2\varepsilon)=0$ whenever $\left|y_{d}-z_{d}-2\varepsilon\right|\geq\varepsilon$. Let us 
	compute
	\begin{eqnarray}\label{initialdata.h1.eqn3}
	\left|y_{d}-z_{d}-2\varepsilon\right|&=& \left|z_{d}+\varepsilon +\left(\varepsilon-y_{d}\right)\right|,\nonumber\\
	&\geq& z_{d} +\varepsilon >\varepsilon.
	\end{eqnarray}
	Therefore $\left(J_{\varepsilon}\tilde{v}\right)$ has compact support in $P^{+}$. As a result, the function 
	$\left(J_{\varepsilon}\tilde{v}(\Psi^{-1}_{0}(x))\right)$ belongs to $C^{1}_{0}(\Omega\cap U_{x_{0}})$. Denote $h_{\epsilon}:\mathbb{R}^{d}\to\mathbb{R}$ by 
	$$h_{\epsilon}(p)=\tilde{\rho}_{\varepsilon}(p_{1})\tilde{\rho}_{\varepsilon}(p_{2})\cdots\tilde{\rho}_{\varepsilon}(p_{d-1})\tilde{\rho}_{\varepsilon}(p_{d}-2\varepsilon).$$
	For $y\in P$, we have $J_{\varepsilon}\tilde{v}(y)=\left(h_{\epsilon}\ast\tilde{v}\right)(y)$. Therefore for $1\leq p<\infty$, we obtain 
	\\
	$$J_{\varepsilon}\tilde{v}\to \eta u_{0}\,\,\,\mbox{in}\,\,\, L^{p}(P^{+})\,\,\mbox{as}\,\,\varepsilon\to 0.$$	
	Therefore we have 
	$$J_{\varepsilon}\tilde{v}(\Psi^{-1}_{0}(x))\to \eta u_{0}\,\,\,\mbox{in}\,\,\, L^{P}(\Omega)\,\,\mbox{as}\,\,\varepsilon\to 0.$$
	Since $\partial\Omega$ is compact, there exists $x_{1}, x_{2},\cdots,x_{N}\in\partial\Omega$ and $U_{1},U_{2},\cdots,U_{N}$
	such that $\partial\Omega\subset\displaystyle\cup_{i=1}^{N}U_{i}$. Choose $U_{N+1}\subset\subset\Omega$ such that 
	$$\overline{\Omega}\subset\displaystyle\cup_{i=1}^{N+1}U_{i}.$$
	For $i=1,2,\cdots,N,N+1$, let $\eta_{i}$ be a partition of unity associated to $U_{i}$. For $i=1,2,\cdots,N$, let 
	$\left(u^{\varepsilon}_{0i}\right)$ be the sequences corresponding to $U_{i},\,\eta_{i}$ obtained as above manner, {\it i.e.},
	$u_{0i}^{\varepsilon}=J_{\varepsilon}\tilde{v_{i}}(\Psi^{-1}_{i}(x))$. Let $\rho_{\varepsilon}:\mathbb{R}^{d}\to\mathbb{R}$ be 
	sequence of mollifiers. Then $u_{0N+1}^{\varepsilon}:=\left(\eta_{N+1}u_{0}\right)\ast\rho_{\varepsilon}\to \eta_{N+1}u_{0}$ on 
	$\overline{U_{N+1}}$ in $L^{p}(\Omega)$ as $\varepsilon\to 0$.\\ Denote
	$$u_{0\varepsilon}(x):=\displaystyle\sum_{i=1}^{N+1}u_{0i}^{\varepsilon}(x).$$
	It is clear that $u_{0\varepsilon}\to u_{0}$ in $L^{p}(\Omega)$ as $\varepsilon\to 0$.\\
	\vspace{0.1cm}\\
	\textbf{Step 2: } Applying change of variable $\frac{y-z}{\varepsilon}=p$ in 
	\eqref{initialdata.h1.eqn3AA}, we get
	\begin{eqnarray}\label{initialdata.h1.eqn4}
	J_{\varepsilon}^{-}\tilde{v}(y) &:=&\int_{y_{1}-\varepsilon}^{y_{1}+\varepsilon}\int_{y_{2}-\varepsilon}^{y_{2}+\varepsilon}\cdots
	\int_{y_{d}-\varepsilon}^{y_{d}+\varepsilon}\frac{1}{\varepsilon^{d}}\tilde{\rho}(\frac{y_{1}-z_{1}}{\varepsilon})\cdots
	\tilde{\rho}(\frac{y_{d-1}-z_{d-1}}{\varepsilon})\tilde{\rho}(\frac{y_{d}-z_{d}-2\varepsilon}{\varepsilon})\,\tilde{v}(z)\,dz,\nonumber\\
	&=& \int_{-1}^{1}\int_{-1}^{1}\cdots\int_{-1}^{1}\int_{3}^{1}\tilde{\rho}(p_{1})\tilde{\rho}(p_{2})\cdots\tilde{\rho}(p_{d-1})\tilde{\rho}(p_{d}-2)\,\tilde{v}(x-\varepsilon p)\,(-1)^{d}\,dp
	\end{eqnarray}
	Taking modulus on both sides of \eqref{initialdata.h1.eqn4}, for $i\in\left\{1,2,\cdots,N\right\}$, we get
	\begin{eqnarray}\label{initialdata.h1.eqn5}
	\left\|J_{\varepsilon}\tilde{v_{i}}(\Psi^{-1}_{i})\right\|_{L^{\infty}(\Omega)}\leq 2^{d}\left(\left\|\rho\right\|_{L^{\infty}(\mathbb{R})}\right)^{d}\,
	\|\eta_{i}\,u_{0}\|_{L^{\infty}(\Omega)}.
	\end{eqnarray}
	In view of \eqref{initialdata.h1.eqn5}, for $i=1,2,\cdots,N$, there exist constants $C_{i}>0$ such that 
	\begin{eqnarray}\label{initialdata.h1.eqn6}
	\|u_{0i}^{\varepsilon}\|_{L^{\infty}(\Omega)}\leq C_{i}
	\end{eqnarray}
	Again, observe that $\left\|\eta_{0} u_{0}\ast\rho_{\epsilon}\right\|_{L^{\infty}(\Omega)}\leq
	\left\|\eta_{0} u_{0}\right\|_{L^{\infty}(\Omega)}$ on $U_{N+1}$.\\
	Taking \\$A=\max\left\{C_{1}, C_{2},\cdots,C_{d}, \left\|\eta_{0} u_{0}\right\|_{L^{\infty}(\Omega)}\right\}$, we conclude the proof of Lemma \ref{hypothesisDinitialdatalem}$\blacksquare$\\
\end{proof}	
\vspace{0.2cm}\\
We give the next result for any space dimension $d\in\mathbb{N}$. A proof of this result can be found in \cite{Ramesh}. The sequence $\left(u_{0\varepsilon}\right)$ of Lemma \ref{regularized.lem1} is used in \textbf{Hypothesis F}.
\begin{lemma}\label{regularized.lem1} {\rm The sequence $(u_{0\varepsilon})$ has the following properties:
	\begin{enumerate} 
		\item $u_{0\varepsilon}\to u_0$ in the sense of intermediate (or strict) convergence in $BV(\Omega)$, {\it i.e.,}
		\begin{equation}\nonumber
		u_{0\varepsilon}\to u_0\,\,\mbox{in}\,\, L^1(\Omega),\quad \int_{\Omega}|Du_{0\varepsilon}|\to \int_{\Omega}|Du_{0}|.
		\end{equation}
		\item There exists a constant $A> 0$ such that for all $\varepsilon > 0$,
		\begin{equation}\label{PhucTorresseqbounds}
		\|u_{0\varepsilon}\|_{L^{\infty}(\Omega)}+ \|\nabla u_{0\varepsilon}\|_{\left(L^{1}(\Omega)\right)^{d}}+\varepsilon\, \|\Delta u_{0\varepsilon}\|_{L^{1}(\Omega)}+ \varepsilon\, \|\nabla u_{0\varepsilon}\|^2_{\left(L^{2}(\Omega)\right)^{d}}\leq A.\nonumber
		\end{equation}
	\end{enumerate}
}
\end{lemma}
\vspace{0.2cm}
Applying a result from \cite{lad-etal_68a}, we conclude the following existence of a unique classical solution for IBVP \eqref{regularized.IBVP}.
\begin{theorem}\cite[p.452]{lad-etal_68a}[\textbf{Unique Classical Solution}]\label{ExistenceofClassicalLadyzenskajap452}
{\rm Let $f$, $B$, $u_0$ satisfy Hypothesis D. Then there exists a unique solution $u^\varepsilon$ of generalized viscosity problem \eqref{regularized.IBVP} in the space
	$C^{2+\beta,\frac{2+\beta}{2}}(\overline{\Omega_T})$. Further, for each $i=1,2,\cdots,d$ the second order partial derivatives $u^\varepsilon_{x_i t}$ belong to $L^2(\Omega_T)$.}  
\end{theorem}
\subsection{Estimates on Solutions and its derivatives}
We now recall the maximum principle of generalized viscosity problem \eqref{regularized.IBVP} from \cite[p.12]{Ramesh}, {\it i.e.,}
\begin{theorem}[Maximum principle]\label{chap3thm1}
{\rm Let $f:\mathbb{R}\to\mathbb{R}^{d}$ be a $C^{1}$ function and $u_{0}\in L^{\infty}(\Omega)$. Then any solution $u^{\varepsilon}$ of generalized viscosity problem \eqref{regularized.IBVP} satisfies the bound
\begin{equation}\label{eqnchap303}
||u^{\varepsilon}(\cdot,t)||_{L^{\infty}(\Omega)}\,\leq\,||u_{0}||_{L^{\infty}(\Omega)}\,a.e.\,\,t\in(0,T).
\end{equation}
}
\end{theorem}
Applying Theorem \ref{chap3thm1} to regularized viscosity problem \eqref{regularized.IBVP} and using $\|u_{0\varepsilon}\|_{L^{\infty}(\Omega)}
\leq A$, we conclude the next result.
\begin{theorem}\label{regularized.chap3thm1}
{\rm Let $f,\,B,\,$ and $u_{0}$ be as in Hypothesis D, Hypothesis E and and Hypothesis F. Then any solution $u$ of generalized viscosity problem \eqref{regularized.IBVP}  satisfies the bound
\begin{equation}\label{regularized.eqnchap303}
||u^{\varepsilon}(\cdot,t)||_{L^{\infty}(\Omega_{T})}\,\leq\,A.
\end{equation}
}
\end{theorem}
Quasilinear viscous approximations $\left(u^{\varepsilon}\right)$ obtained in
Theorem \ref{ExistenceofClassicalLadyzenskajap452} satisfies the following estimate. This is a very useful result.
\begin{theorem}\label{Compactness.lemma.1}
{\rm Let $f,\,\,B,\,\,u_{0}$ satisfy Hypothesis D, Hypothesis E and Hypothesis F. Let $u^{\varepsilon}$ be the unique solution to generalized 
viscosity problem \eqref{regularized.IBVP}. Then 
 \begin{eqnarray}\label{uniformnot.compactness.eqn1a}
 \displaystyle\sum_{j=1}^{d} \,\left(\sqrt{\varepsilon}\Big\| \frac{\partial u^{\varepsilon}}{\partial x_{j}}\Big\|_{L^{2}(\Omega_{T})}\right)^{2} \leq\frac{1}{2r}\|u_{0\varepsilon}\|^{2}_{L^{2}(\Omega)}\leq\frac{1}{2r} A^2\,\,\mbox{Vol}(\Omega).
 \end{eqnarray}
}
\end{theorem}
A proof of Theorem \ref{Compactness.lemma.1} follows from \cite{Ramesh}.
\section{Compactness of quasilinear viscous approximations}\label{compactnesscompensatedofsolutions}
In this section, we show the existence of an {\it a.e.} convergent subsequence of sequence 
of quasilinear viscous approximations $\left(u^{\varepsilon}\right)$ which are solutions to generalized viscosity problem \eqref{regularized.IBVP}.
In order to extract {\it a.e.} convergent subsequence of quasiliear viscous approximations $\left(u^{\varepsilon}\right)$, we use the method of Compensated Compactness. We use Divergence-Curl lemma for $d=1$ and for $d=2$, we use Theorem of Compensated Compactness.\\
\vspace{0.1cm}\\
The following result shows that the quasilinear viscous approximations $\left(u^{\varepsilon}\right)$ satisfies compact entropy
productions. We prove Theorem \ref{chap9thm2} with Hypothesis E in any space dimension $d\in\mathbb{N}$ and the same proof works with Hypothesis F.
\begin{theorem}\label{chap9thm2}
{\rm Assume \textbf{Hypothesis E} and let $\left(u^{\varepsilon}\right)$ be as in Theorem \ref{ExistenceofClassicalLadyzenskajap452}.
Then 
 \begin{equation}\label{chap9eqn2}
  \frac{\partial \eta(u^{\varepsilon})}{\partial t} + \displaystyle\sum_{j=1}^{d}\frac{\partial q_{j}(u^{\varepsilon})}{\partial x_{j}}\hspace{0.2cm}\subset\,\mbox{compact set in }\hspace{0.2cm} H^{-1}(\Omega_{T})
 \end{equation}
for every $C^{2}(\mathbb{R})$ entropy-entropy flux pair $(\eta, q)$ of conservation laws \eqref{ivp.cl.a}.}
\end{theorem}
The following result is used to prove \eqref{chap9eqn2}.
\begin{lemma}{\rm\cite[p.514]{MR2169977}}\label{chap9lem2}
 {\rm Let $\Omega$ be an open subset of $\mathbb{R}^{d}$ and $\left(\phi_{n}\right)$ be a bounded sequence in 
 $W^{-1,p}(\Omega)$, for some $p > 2$. Further, let $\phi_{n}= \xi_{n} + \psi_{n}$, where $\left(\xi_{n}\right)$ 
 lies in a compact set of $H^{-1}(\Omega)$, while $\left(\psi_{n}\right)$ lies in a bounded set of the space of measures
 $M(\Omega)$. Then $\left(\phi_{n}\right)$ lies in a compact set of $H^{-1}(\Omega)$.} 
\end{lemma}
\textbf{Proof of Theorem \ref{chap9thm2}:} Let $\eta:\mathbb{R}\to \mathbb{R}$ be a convex, $C^{2}(\mathbb{R})$ entropy. Then for $j=1,2,\cdots,d$, there exist $C^{2}(\mathbb{R})$ functions $q_{j}:\mathbb{R}\to\mathbb{R}$ such that
 \begin{equation}\label{compactnessequation1}
  \eta^{\prime}f_{j}^{\prime}= q_{j}^{\prime}.
 \end{equation}
 Multiplying both sides the equation \eqref{regularized.IBVP.a} by $\eta^{\prime}(u^{\varepsilon})$ and using chain rule, we get
 \begin{equation}\label{compactnessequation2}
 \frac{\partial \eta(u^{\varepsilon})}{\partial t} + \displaystyle\sum_{j=1}^{d}\frac{\partial q_{j}(u^{\varepsilon})}{\partial x_{j}} = \varepsilon\,\displaystyle\sum_{j=1}^{d}\frac{\partial }{\partial x_{j}}\left(B(u^{\varepsilon})\frac{\partial u^{\varepsilon}}{\partial x_{j}}\right)\,\eta^{\prime}(u^{\varepsilon}).
\end{equation}
From equation \eqref{compactnessequation2}, we get
\begin{equation}\label{chap9eqn4}
 \frac{\partial \eta(u^{\varepsilon})}{\partial t} + \displaystyle\sum_{j=1}^{d}\frac{\partial q_{j}(u^{\varepsilon})}{\partial x_{j}} = \varepsilon\,\displaystyle\sum_{j=1}^{d}\frac{\partial }{\partial x_{j}}\left(B(u^{\varepsilon})\frac{\partial \eta(u^{\varepsilon})}{\partial x_{j}}\right) - \varepsilon \displaystyle\sum_{j=1}^{d}B(u^{\varepsilon})\,\left(\frac{\partial u}{\partial x_{j}}\right)^{2}\,\eta^{\prime\prime}(u^{\varepsilon})
\end{equation}
By appealing to Lemma \ref{chap9lem2}, \eqref{chap9eqn2} will be proved if we prove 
\begin{equation}\label{compactnessequation3}
 \varepsilon\,\displaystyle\sum_{j=1}^{d}\frac{\partial }{\partial x_{j}}\left(B(u^{\varepsilon})\frac{\partial \eta(u^{\varepsilon})}{\partial x_{j}}\right)\to 0\,\,\mbox{in}\,\,H^{-1}(\Omega_{T})
\end{equation}
and 
\begin{equation}\label{compactnessequation4}
 - \varepsilon \displaystyle\sum_{j=1}^{d}B(u^{\varepsilon})\,\left(\frac{\partial u}{\partial x_{j}}\right)^{2}\,\eta^{\prime\prime}(u^{\varepsilon})\hspace{0.2cm}\mbox{is bounded in the space of measures}\,\,M(\Omega_{T}).
\end{equation}
Firstly, we prove \eqref{compactnessequation3}. Note that
\begin{equation}\label{chap9eqn5}
 \Big\|\varepsilon\,\displaystyle\sum_{j=1}^{d}\frac{\partial }{\partial x_{j}}\left(B(u^{\varepsilon})\frac{\partial \eta(u^{\varepsilon})}{\partial x_{j}}\right)\Big\|_{H^{-1}(\Omega_{T})} = \displaystyle\sup_{\substack{\phi\in H^{1}_{0}(\Omega_{T}),\\ \|\phi\|_{H^{1}_{0}(\Omega_{T})}\leq 1}}\left|\int_{0}^{T}\int_{\Omega}\varepsilon\,\displaystyle\sum_{j=1}^{d}\frac{\partial }{\partial x_{j}}\left(B(u^{\varepsilon})\frac{\partial \eta(u^{\varepsilon})}{\partial x_{j}}\right)\,\phi\,dx\,dt\right|. 
\end{equation}
Using integration by parts formula and  $\|\phi\|_{H^{1}_{0}(\Omega_{T})}\leq 1$ in \eqref{chap9eqn5}, we arrive at 
\begin{equation}\label{chap9eqn6}
 \left|-\int_{0}^{T}\int_{\Omega}\varepsilon\,\displaystyle\sum_{j=1}^{d}B(u^{\varepsilon})\eta^{\prime}(u^{\varepsilon})\frac{\partial u^{\varepsilon}}{\partial x_{j}}\,\frac{\partial \phi}{\partial x_{j}}\,dx\,dt\right| \leq \varepsilon \|B\|_{L^{\infty}(\mathbb{R})}\|\eta^{'}\|_{L^{\infty}(I)}\|\nabla u^{\varepsilon}\|_{\left(L^{2}(\Omega_{T})\right)^{d}}.
\end{equation}
From  Theorem \ref{Compactness.lemma.1}, we have
\begin{equation}\label{Comp.chap9eqn26}
  \displaystyle\sum_{j=1}^{d} \,\left(\sqrt{\varepsilon}\Big\| \frac{\partial u^{\varepsilon}}{\partial x_{j}}\Big\|_{L^{2}(\Omega_{T})}\right)^{2} \leq\frac{1}{2r}\|u_{0\varepsilon}\|^{2}_{L^{2}(\Omega)}\leq
  \frac{1}{2r}A^{2}\mbox{Vol}(\Omega).
\end{equation}
Using \eqref{Comp.chap9eqn26} in \eqref{chap9eqn6} and letting $\varepsilon\to 0$ in \eqref{chap9eqn6}, we have \eqref{compactnessequation3}.\\
Secondly, we want to prove \eqref{compactnessequation4}.  We have $- \varepsilon \displaystyle\sum_{j=1}^{d}B(u^{\varepsilon})\,\left(\frac{\partial u}{\partial x_{j}}\right)^{2}\,\eta^{''}(u^{\varepsilon})\in L^{1}(\Omega_{T})$. We know that $L^{1}(\Omega_{T})$ is continuously imbedded in $\left(L^{\infty}(\Omega_{T})\right)^{\ast}$. Therefore, we have 
\begin{eqnarray}\label{chap9eqn11}
 \Big\| -\varepsilon \displaystyle\sum_{j=1}^{d}B(u^{\varepsilon})\,\left(\frac{\partial u}{\partial x_{j}}\right)^{2}\,\eta^{''}(u^{\varepsilon})\Big\|_{M(\Omega_{T})} &\leq& \Big\|-\varepsilon \displaystyle\sum_{j=1}^{d}B(u^{\varepsilon})\,\left(\frac{\partial u}{\partial x_{j}}\right)^{2}\,\eta^{''}(u^{\varepsilon})\Big\|_{L^{1}(\Omega_{T})},\nonumber\\
 &\leq& \varepsilon \|B\|_{L^{\infty}(\mathbb{R})}\,\|\eta^{''}\|_{L^{\infty}(I)}\|\nabla u\|_{\left(L^{2}(\Omega_{T})\right)^{d}}.
\end{eqnarray}
Using inequality \eqref{Comp.chap9eqn26} in \eqref{chap9eqn11}, we get
\begin{equation}\label{compactnessequation9}
 \Big\| -\varepsilon \displaystyle\sum_{j=1}^{d}B(u^{\varepsilon})\,\left(\frac{\partial u}{\partial x_{j}}\right)^{2}\,\eta^{''}(u^{\varepsilon})\Big\|_{M(\Omega_{T})} \leq C^{'}\|B\|_{L^{\infty}(\mathbb{R})}\|\eta^{''}\|_{L^{\infty}(I)},
\end{equation}
where $C^{'}$ is independent of $\varepsilon$. Therefore we have obtained \eqref{compactnessequation4}.
Using \eqref{compactnessequation3}, \eqref{compactnessequation4} and in view of Lemma \ref{chap9lem2}, we have \eqref{chap9eqn2}.
\\

For space dimension $d=1$, the extraction of an {\it a.e.} convergent subsequence is obtained by proving the following result.

\begin{theorem}\label{chap9thm3}
{\rm Assume \textbf{Hypothesis E} and let $\left(u^{\varepsilon}\right)$ be sequence of solutions to generalized viscosity problem \eqref{regularized.IBVP} such that 
 \begin{equation}\label{chap9eqn13}
  \frac{\partial \eta(u^{\varepsilon})}{\partial t} + \frac{\partial q(u^{\varepsilon})}{\partial x}\hspace{0.2cm}\subset\,\mbox{compact set in }\hspace{0.2cm} H^{-1}(\Omega_{T}).
 \end{equation}
for every $C^{2}(\mathbb{R})$ entropy-entropy flux pair $(\eta, q)$ of scalar conservation laws \eqref{ivp.cl.a} in one space dimension. Then there is a subsequence of $\left(u^{\varepsilon}\right)$ such that the subsequence is denoted by $\left(u^{\varepsilon}\right)$ and  the following 
convergence in $L^{\infty}(\Omega_{T})-\mbox{weak}^{\ast}$
\begin{equation*}
 u^{\varepsilon}\rightharpoonup u,\hspace{0.3cm}f(u^{\varepsilon})\rightharpoonup f(u),\hspace{0.2cm}\mbox{as}\,\varepsilon\to 0.
\end{equation*}
holds. Further, if the set of $u$ with $f^{''}(u)\neq 0$ is dense in $\mathbb{R}$, then $\left(u^{\varepsilon}\right)$ converges almost everywhere to $u$ in $\Omega_{T}$.}
\end{theorem}
The following results are used to prove Theorem \ref{chap9thm3}. 
\begin{theorem}[\textbf{Young Measure}]{\rm \cite[p.147]{MR584398}}\label{chap9thm1}
\begin{enumerate}
\item {\rm Let $K\subset \mathbb{R}^{p}$ be bounded and $\Omega\subset \mathbb{R}^{d}$ be an open set. Let $u_{n}:\Omega\to\mathbb{R}^{p}$ be such that $u_{n}\in K$ {\it a.e.}. Then there exists a subsequence $\left(u_{m}\right)$ and a family of probability measures $\left(\nu_{x}\right)_{x\in\Omega}$ $\left(\mbox{depending measurably on x}\right)$ with $\mbox{supp}\,\nu_{x}\subset \overline{K}$ such that if $F$ is continuous function on $\mathbb{R}^{p}$ and 
\begin{equation*}
 \overline{f} = <\nu_{x}, F(\lambda)>\,\, {\it a.e.}
\end{equation*}
then 
\begin{equation*}
 F(u_{m})\rightharpoonup \overline{f}(x)\hspace{0.2cm}\mbox{in}\hspace{0.2cm}L^{\infty}(\Omega)-\mbox{weak}^{\ast}
\end{equation*}
}
\item {\rm Conversely, let $\left(\nu_{x}\right)_{x\in\Omega}$ be a family of  probability measures with support in $ \overline{K}$. Then there exists a sequence  $\left(u_{n}\right)$, where $u_{n}:\Omega\to\mathbb{R}^{p}$ and $u_{n}\in K$ {\it a.e.}, such that for all continuous functions on $\mathbb{R}^{p}$, we have
\begin{equation*}
 F(u_{n})\rightharpoonup \overline{f}(x)=<\nu_{x}, F(\lambda)>\hspace{0.2cm}\mbox{in}\hspace{0.2cm}L^{\infty}(\Omega)-\mbox{weak}^{\ast}
\end{equation*}
}
\end{enumerate}
\end{theorem}
We give definition of Young measures.
\begin{definition}
 {\rm The family of probability measure $\left(\nu_{x}\right)_{x\in \Omega}$ that we get from Theorem \ref{chap9thm1} is called Young measures associated with the sequence $\left(u_{n}\right)_{n=1}^{\infty}$.}
\end{definition}
\begin{lemma}(\textbf{Div-curl lemma}){\rm \cite[p.90]{MR2582099},\cite[p.513]{MR2169977}}\label{chap9lem1}\\
 {\em Let $\Omega\subset\mathbb{R}^{d}$ be an open set and $G_{n}$ and $H_{n}$ be two sequences of vector fields in $L^{2}(\Omega;\mathbb{R}^{d})$ converging weakly to limits $\overline{G}$ and $\overline{H}$ respectively as $n\to \infty$. Assume both $\left(\mbox{div}\,G_{n}\right)$ and $\left(\mbox{curl}\,H_{n}\right)$ lie in a compact subset of $H^{-1}_{loc}(\Omega)$. Then we have the following convergence in the sense of distributions as $n\to \infty$ 
 \begin{equation*}
  G_{n}.H_{n}\to \overline{G}.\overline{H}.
 \end{equation*}
}
\end{lemma}
\vspace{0.2cm}
\textbf{Proof of Theorem \ref{chap9thm3}:} A proof of Theorem \ref{chap9thm3} is available on the lines of \cite[p.518]{MR2169977}$\blacksquare$ \\
\vspace{0.2cm}\\
We now show that the sequence $\left(\frac{\partial u^{\varepsilon}}{\partial t}\right)$ lies in a compact set of $H^{-1}_{loc}(\Omega_{T})$. This is used in the extraction of {\it a.e.} convergent subsequence of the quasilinear viscous approximations $\left(u^{\varepsilon}\right)$ for $d=2$.
\begin{theorem}\label{timedervative.thm1}
 {\rm Assume Hypothesis F. Let $\left(u^{\varepsilon}\right)$ be the sequence of solutions to \eqref{regularized.IBVP}. Then 
 $\left(\frac{\partial u^{\varepsilon}}{\partial t}\right)$ is compact in $H^{-1}_{loc}(\Omega_{T})$.
}
\end{theorem}
\begin{proof}
 We prove Theorem \ref{timedervative.thm1} in two steps. In Step-1, we show that the sequence $\left(\frac{\partial u^{\varepsilon}}{\partial t}\right)$
 is bounded in $L^{1}(\Omega_{T})$ and in Step-2, we use Murat's Lemma \ref{chap9lem2} to show that $\left(\frac{\partial u^{\varepsilon}}{\partial t}\right)$
 is compact in $H^{-1}_{loc}(\Omega_{T})$.\\
 \textbf{Step-1:} Let $u_{0}\in BV_{0}\left(\Omega\right)\cap L^{\infty}\left(\Omega\right)$. Applying a result 
 from \cite{Ramesh}, we conclude the existence of a constant $C_{1}> 0$ such that for every $\varepsilon>0$, we have 
\begin{eqnarray}\label{B.BVestimate26}
 \left\|\frac{\partial u^{\varepsilon}}{\partial t}\right\|_{L^{1}(\Omega_{T})}\leq C_{1}.
\end{eqnarray}
 \textbf{Step 2:} Since $L^{1}(\Omega_{T})$ is continuously imbedded in the space of measures $M(\Omega_{T})$. Therefore we have
 \begin{eqnarray}\label{rsb21.eqn109}
  \left\|\frac{\partial u^{\varepsilon}}{\partial t}\right\|_{M(\Omega_{T})}&\leq& \left\|\frac{\partial u^{\varepsilon}}{\partial t}\right\|_{L^{1}(\Omega_{T})}
 \end{eqnarray}
In view of \eqref{B.BVestimate26}, we see that $\left(\frac{\partial u^{\varepsilon}}{\partial t}\right)$ is bounded in the space of
measures $M(\Omega_{T})$.
The sequence $\left(\frac{\partial u^{\varepsilon}}{\partial t}\right)$ is bounded in $W^{-1,\infty}(\Omega_{T})$ as 
$\|u^{\varepsilon}\|_{L^{\infty}(\Omega_{T})}\leq A$. An application of Murat's Lemma \ref{chap9lem2},
 we get that the sequence $\left(\frac{\partial u^{\varepsilon}}{\partial t}\right)$ is compact in $H^{-1}(\Omega_{T})$ $\blacksquare$
\end{proof}\\
\vspace{0.1cm}\\
We need to prove the following result for extraction of {\it a.e.} convergent subsequence of quasilinear viscous approximations
$\left(u^{\varepsilon}\right)$ to generalized viscosity problem \eqref{regularized.IBVP} for $d=2$. 
\begin{theorem}\label{compactness12.thm2}
 {\rm Assume \textbf{Hypothesis F} and let $\left(u^{\varepsilon}\right)$ be the sequence of solutions to 
 viscosity problem \eqref{regularized.IBVP}. Then there exists a subsequence $\left(u^{\varepsilon_{k}}\right)$ of 
 $\left(u^{\varepsilon}\right)$ and a function $u$ such that for {\it a.e.} $(x,t)\in\Omega_{T}$ the following convergence
 \begin{equation}\label{compactness12.eqn19}
  u^{\varepsilon_{k}}(x,t)\to u(x,t)\,\,\,\mbox{as}\,\,\,k\to\infty.
 \end{equation}
 holds.}
\end{theorem}
\begin{proof}
A proof of the extraction of almost everywhere convergent subsequence of $\left(u^{\varepsilon}\right)$ is available  in \cite{Tadmor} whenever $B\equiv 1$. We use similar arguments to prove Theorem \ref{compactness12.thm2}. 	Multiplying \eqref{regularized.IBVP.a} by $f_{1}^{'}$, $f_{2}^{'}$, we get
\begin{eqnarray}\label{compactness12.eqn20}
 \left(f_{1}^{'}(u^{\varepsilon})\right)^{2}\,\frac{\partial u^{\varepsilon}}{\partial x_{1}} + f_{1}^{'}(u^{\varepsilon})f_{2}^{'}(u^{\varepsilon})\,\frac{\partial u^{\varepsilon}}{\partial x_{2}}&=& \varepsilon\displaystyle\sum_{j=1}^{2}f_{1}^{'}(u^{\varepsilon})\frac{\partial}{\partial x_{j}}\left(B(u^{\varepsilon})\,\frac{\partial u^{\varepsilon}}{\partial x_{j}}\right)-\frac{\partial f_{1}(u^{\varepsilon})}{\partial t},\nonumber\\
 f_{1}^{'}(u^{\varepsilon})f_{2}^{'}(u^{\varepsilon})\,\frac{\partial u^{\varepsilon}}{\partial x_{1}} + \left(f_{2}^{'}(u^{\varepsilon})\right)^{2}\,\frac{\partial u^{\varepsilon}}{\partial x_{2}} &=& \varepsilon\displaystyle\sum_{j=1}^{2}f_{2}^{'}(u^{\varepsilon})\frac{\partial}{\partial x_{j}}\left(B(u^{\varepsilon})\,\frac{\partial u^{\varepsilon}}{\partial x_{j}}\right)-\frac{\partial f_{2}(u^{\varepsilon})}{\partial t}.
\end{eqnarray}
Denote
\begin{eqnarray}\label{compactness12.eqn21}
 F_{11}(\lambda) :=\int_{0}^{\lambda}\left(f_{1}^{'}(s)\right)^{2}\,ds,\nonumber\\
 F_{12}(\lambda) :=\int_{0}^{\lambda}f_{1}^{'}(s)f_{2}^{'}(s)\,ds,\nonumber\\
 F_{22}(\lambda) :=\int_{0}^{\lambda}\left(f_{2}^{'}(s)\right)^{2}\,ds.
\end{eqnarray}
Equation \eqref{compactness12.eqn20} can be rewritten as 
\begin{eqnarray}\label{compactness12.eqn22}
 \frac{\partial F_{11}(u^{\varepsilon})}{\partial x_{1}} + \frac{\partial F_{12}(u^{\varepsilon})}{\partial x_{2}} &=& \varepsilon\displaystyle\sum_{j=1}^{2}\frac{\partial}{\partial x_{j}}\left(B(u^{\varepsilon})\,\frac{\partial f_{1} (u^{\varepsilon})}{\partial x_{j}}\right)-\varepsilon\displaystyle\sum_{j=1}^{2}B(u^{\varepsilon})\left(\frac{\partial u^{\varepsilon}}{\partial x_{j}}\right)^{2}f_{1}^{''}(u^{\varepsilon})-\frac{\partial f_{1}(u^{\varepsilon})}{\partial t},\nonumber\\
 \frac{\partial F_{12}(u^{\varepsilon})}{\partial x_{1}} + \frac{\partial F_{22}(u^{\varepsilon})}{\partial x_{2}} &=& \varepsilon\displaystyle\sum_{j=1}^{2}\frac{\partial}{\partial x_{j}}\left(B(u^{\varepsilon})\,\frac{\partial f_{2}(u^{\varepsilon}) }{\partial x_{j}}\right)-\varepsilon\displaystyle\sum_{j=1}^{2}B(u^{\varepsilon})\left(\frac{\partial u^{\varepsilon}}{\partial x_{j}}\right)^{2}f_{2}^{''}(u^{\varepsilon})-\frac{\partial f_{2}(u^{\varepsilon})}{\partial t}.\nonumber\\
 &
\end{eqnarray}
We now show that RHS of two equations of \eqref{compactness12.eqn22} lie in a compact set of $H^{-1}(\Omega_{T})$. For that we show that for $i=1,2$,
\begin{enumerate}
 \item[(i).]
 \begin{equation}\label{compactness12.eqn23}
  \varepsilon\displaystyle\sum_{j=1}^{2}\frac{\partial}{\partial x_{j}}\left(B(u^{\varepsilon})\,\frac{\partial f_{i} (u^{\varepsilon})}{\partial x_{j}}\right)\to 0\,\,\mbox{in}\,\,H^{-1}(\Omega_{T}),
 \end{equation}
\item[(ii).] 
\begin{equation}\label{compactness12.eqn24}
 -\varepsilon\displaystyle\sum_{j=1}^{2}B(u^{\varepsilon})\left(\frac{\partial u^{\varepsilon}}{\partial x_{j}}\right)^{2}f_{1}^{''}(u^{\varepsilon})-\frac{\partial f_{1}(u^{\varepsilon})}{\partial t}\,\,\mbox{is bounded in the space of measure}\,\,M(\Omega_{T}).
\end{equation}
\end{enumerate}
Firstly, we prove \eqref{compactness12.eqn23}. Observe that 
\begin{eqnarray}\label{compactness12.eqn25}
 \Big\|\varepsilon\displaystyle\sum_{j=1}^{2}\frac{\partial}{\partial x_{j}}\left(B(u^{\varepsilon})\,\frac{\partial f_{i} (u^{\varepsilon})}{\partial x_{j}}\right)\Big\|_{H^{-1}(\Omega_{T})} &=& \sup\Big\{\Big|\int_{0}^{T}\int_{\Omega}\left(\varepsilon\displaystyle\sum_{j=1}^{2}\frac{\partial}{\partial x_{j}}\left(B(u^{\varepsilon})\,\frac{\partial f_{i} (u^{\varepsilon})}{\partial x_{j}}\right)\right)\,\nonumber\\ &&\phi(x,t)\,dx\,dt\Big|\,; \|\phi\|_{H^{1}_{0}(\Omega_{T})}\leq 1\Big\},\nonumber\\
 &=& \sup\Big\{\Big|-\int_{0}^{T}\int_{\Omega}\left(\varepsilon\displaystyle\sum_{j=1}^{2}\left(B(u^{\varepsilon})\,f^{'}_{i} (u^{\varepsilon})\frac{\partial u^{\varepsilon}}{\partial x_{j}}\right)\right)\,\nonumber\\ &&\frac{\partial\phi}{\partial x_{j}}(x,t)\,dx\,dt\Big|\,; \|\phi\|_{H^{1}_{0}(\Omega_{T})}\leq 1\Big\},\nonumber\\
 &\leq& \varepsilon \|B\|_{L^{\infty}(I)}\,\left(\displaystyle\max_{1\leq i\leq 2}\|f_{i}^{'}\|_{L^{\infty}(I)}\right)\,\|\nabla u^{\varepsilon}\|_{\left(L^{2}(\Omega_{T})\right)^{2}}.\nonumber\\
 &
\end{eqnarray}
Since $\sqrt{\varepsilon}\|\nabla u^{\varepsilon}\|_{\left(L^{2}(\Omega_{T})\right)^{2}}\leq C,$ which is independent of $\varepsilon$, therefore we have \eqref{compactness12.eqn23} and $\varepsilon\displaystyle\sum_{j=1}^{2}\frac{\partial}{\partial x_{j}}\left(B(u^{\varepsilon})\,\frac{\partial f_{i} (u^{\varepsilon})}{\partial x_{j}}\right)$ lie in a compact set of $H^{-1}(\Omega_{T})$.

Secondly, we show \eqref{compactness12.eqn24}. We know that $L^{1}(\Omega_{T})$ is continuously imbeeded in $\left(L^{\infty}(\Omega_{T})\right)^{\ast}$. Therefore we have 
\begin{eqnarray}\label{compactness12.eqn26}
 \Big\| -\varepsilon\displaystyle\sum_{j=1}^{2}B(u^{\varepsilon})\left(\frac{\partial u^{\varepsilon}}{\partial x_{j}}\right)^{2}f_{1}^{''}(u^{\varepsilon})-\frac{\partial f_{1}(u^{\varepsilon})}{\partial t}\Big\|_{M(\Omega_{T})} \leq 
 \int_{0}^{T}\int_{\Omega}\Big|\varepsilon\displaystyle\sum_{j=1}^{2}B(u^{\varepsilon})\left(\frac{\partial u^{\varepsilon}}{\partial x_{j}}\right)^{2}f_{1}^{''}(u^{\varepsilon})\nonumber\\-f_{1}^{'}(u^{\varepsilon})\frac{\partial u^{\varepsilon}}{\partial t}\Big|\,dx\,dt,\nonumber\\
 \leq \|B\|_{L^{\infty}(I)}\,\displaystyle\max_{1\leq i\leq 2}\left(\displaystyle\sup_{y\in I}\Big|f_{i}^{''}(y)\Big|\right)\left(\sqrt{\varepsilon}\,\|\nabla u^{\varepsilon}\|_{\left(L^{2}(\Omega_{T})\right)^{2}}\right)^{2} 
 + \displaystyle\max_{1\leq i\leq 2}\left(\displaystyle\sup_{y\in I}\Big|f_{i}^{''}(y)\Big|\right)\,\Big\|\frac{\partial u^{\varepsilon}}{\partial t}\Big\|_{L^{1}(\Omega_{T})}.
\end{eqnarray}
Applying Theorem \ref{timedervative.thm1} and $\sqrt{\varepsilon}\|\nabla u^{\varepsilon}\|_{\left(L^{2}(\Omega_{T})\right)^{2}}\leq C,$ which is independent of $\varepsilon$, we get \eqref{compactness12.eqn24}.\\
We want to use Theorem of Compensated Compactness \cite[p.31]{Dacorogna} to conclude the almost everywhere convergence of 
$\left(u^{\varepsilon}\right)$ to a function $u$ in $L^{\infty}(\Omega_{T})$.
Observe that 
\begin{eqnarray}\label{Comp.Theorem.eqn1}
 F_{11}(u^{\varepsilon})&\rightharpoonup& \overline{F}_{11}\,\,\mbox{in}\,\,\, L^{2}(\Omega_{T})\,\,\mbox{as}\,\,\varepsilon\to 0,\nonumber\\
 F_{12}(u^{\varepsilon})&\rightharpoonup& \overline{F}_{12}\,\,\mbox{in}\,\,\, L^{2}(\Omega_{T})\mbox{as}\,\,\varepsilon\to 0\,\,,\nonumber\\
 F_{22}(u^{\varepsilon})&\rightharpoonup& \overline{F}_{22}\,\,\mbox{in}\,\,\, L^{2}(\Omega_{T})\mbox{as}\,\,\varepsilon\to 0\,\,.
\end{eqnarray}
Therefore, we obtain
$$\left(F_{11}(u^{\varepsilon}), F_{12}(u^{\varepsilon}), F_{12}(u^{\varepsilon}), F_{22}(u^{\varepsilon}) \right)\rightharpoonup \left(\overline{F}_{11},\overline{F}_{12},\overline{F}_{12}, \overline{F}_{22} \right)\,\,\mbox{as}\,\,\varepsilon\to 0.$$
The following combinations 
\begin{eqnarray}\label{comp.Theorem.eqn2}
 \frac{\partial}{\partial x_{1}}F_{11}(u^{\varepsilon}) + \frac{\partial}{\partial x_{2}}F_{12}(u^{\varepsilon})\,\, &,& \frac{\partial}{\partial x_{1}}F_{12}(u^{\varepsilon}) + \frac{\partial}{\partial x_{2}}F_{22}(u^{\varepsilon}),\nonumber\\
 \frac{\partial}{\partial t}F_{11}(u^{\varepsilon}) &,& \frac{\partial}{\partial t}F_{22}(u^{\varepsilon}),
\end{eqnarray}
are compact in $H^{-1}(\Omega_{T})$.\\
Consider the set 
$$\nu :=\left\{\left(\lambda,\xi\right)\in \mathbb{R}^{4}\times\mathbb{R}^{3}\setminus\left\{0\right\}\,\,;\,\,\lambda_{1}\xi_{1}+\lambda_{2}\xi_{2}=0,\,\,
\lambda_{1}\xi_{1}+\lambda_{2}\xi_{2}=0,\,\,\lambda_{1}\xi_{0}=0,\,\,\lambda_{3}\xi_{0}=0\right\}.$$
The quadratics $Q(F_{11}(u^{\varepsilon}), F_{12}(u^{\varepsilon}), F_{12}(u^{\varepsilon}), F_{22}(u^{\varepsilon}))$ which vanish 
on the projections,
$$\Lambda=\left\{\lambda\in\mathbb{R}^{4}\,\,:\,\,\left(\lambda,\xi\right)\in\nu\right\},$$
are 
$$\left\{\lambda\in\mathbb{R}^{4}\,\,:\,\,\lambda_{1}\lambda_{4}-\lambda_{2}\lambda_{3}=0\right\}.$$
As a result, we have 
\begin{eqnarray}\label{comp.Theorem.eqn3}
 \left(F_{11}(u^{\varepsilon}),\,F_{12}(u^{\varepsilon})\right)\cdot\left(F_{22}(u^{\varepsilon}),\,F_{12}(u^{\varepsilon})\right)\rightharpoonup \overline{F}_{11}\overline{F}_{22}
 -\overline{F}_{12}^{2}\,\,\mbox{in}\,\,L^{2}(\Omega_{T})\,\,\mbox{as}\,\,\varepsilon\to 0.
\end{eqnarray}
If we write \eqref{comp.Theorem.eqn3} in term of youngs measures $\nu_{x,t}(\cdot)$, we get
\begin{eqnarray}\label{comp.Theorem.eqn4}
 \Big\langle \nu_{x,t}, \left(F_{11}(\lambda)-\overline{F}_{11}\right)\cdot\left(F_{22}(\lambda)-\overline{F}_{22}\right)-\left(F_{12}(\lambda)-\overline{F}_{12}\right)^{2}\Big\rangle=0.
\end{eqnarray}
 A proof of extraction of a subsequence of $\left(u^{\varepsilon}\right)$ converges {\it a.e.} to a function $u$ in $L^{\infty}(\Omega_{T})$ follows from \cite[p.702-p.703]{Tadmor}. This completes the proof\,$\blacksquare$
\end{proof}
\section{Kinetic Formulation-the multidimensional case}\label{kineticformulation.generalizedviscosityproblem}
We write the proof of the extraction of {\it a.e.} convergent subsequence of the quasiliear viscous approximations $\left(u^{\varepsilon}\right)$ which are unique solutions of \eqref{regularized.IBVP}. We use Velocity averaging lemma to extract the {\it a.e.} convergent subsequence. In order to apply velocity averaging lemma, we need to derive the Kinetic formulation of \eqref{regularized.IBVP.a}. We follow the computation from the lines of \cite[p.706]{Tadmor} to derive the kinetic formulation. We now introduce the sign function and the Kruzhkov entropy-entropy flux pairs which is used in the derivation of Kinetic formulation of \eqref{regularized.IBVP.a}. Denote 
$$\mbox{sg}(s):=\begin{cases}
1\,\,\,\mbox{if}\,\,\,s>0\\
0\,\,\,\mbox{if}\,\,\,s=0\\
-1\,\,\,\mbox{if}\,\,\,s<0
\end{cases}$$

For $c\in\mathbb{R}$, the family of Kruzhkov entropy and entropy fluxes are $\eta(u;c):=\left|u-c\right|$ and $j\in\left\{1,2,\cdots,d\right\}$, $q_{j}(u;c)=\mbox{sg}\left(u-c\right)\left(f_{j}(u)-f_{j}(c)\right)$. 
\subsection{Kinetic Equations of Generalized Viscosity Problem}\label{Kineticformulationviscosityproblem}
\begin{proposition}\label{PropositionKineticEquationLions1}
The kinteic equation of the generalized viscosity problem
\eqref{regularized.IBVP.a} is given by 
\begin{equation}\label{KineticFormulation.Equation612}
\frac{\partial \chi^{\varepsilon}}{\partial t} +\displaystyle\sum_{j=1}^{d}f_{j}^{\prime}(c)\,\frac{\partial\chi^{\varepsilon}}{\partial x_{j}} = \frac{\partial}{\partial c}\left(\frac{\varepsilon}{2}\,\displaystyle\sum_{j=1}^{d}\eta^{\prime}(u^{\varepsilon};c)\frac{\partial}{\partial x_{j}}\left(B(u^{\varepsilon})\frac{\partial u^{\varepsilon}}{\partial x_{j}}\right)\right)
\end{equation}
in  $\mathcal{D}^{\prime}\left(\Omega\times\mathbb{R}\times(0,T)\right)$	
\end{proposition}
\begin{proof}
Let $G:\mathbb{R}\to\mathbb{R}$ be a $C^{\infty}-$ function such that 
$$G(x):=\left|x\right|\,\,\mbox{for}\,\,\left|x\right|\geq 1,\,\,G^{\prime\prime}\geq 0.$$
For $n\in\mathbb{N}$, denote 
\begin{equation*}
G_{n}(x):=\frac{1}{n}\,G\left(n\,\left(x-c\right)\right).$$
Then $G_{n}(x)\to \left|x-c\right|$ as $n\to\infty$. We now compute 
$$G^{\prime}_{n}(x):=\begin{cases}
\frac{d}{dx}\left(\left|x-c\right|\right),\,\,\mbox{if}\,\,\left|x-c\right|\geq\frac{1}{n},\\
G^{\prime}\left(n\left(x-c\right)\right),\,\mbox{if}\,\left|x-c\right|<\frac{1}{n}.
\end{cases}
\end{equation*}
Then we have
\begin{equation*}
G^{\prime}_{n}(x)=
\begin{cases}
1\,\,\mbox{if}\,\,x\geq c\,\,\left|x-c\right|\geq\frac{1}{n},\\
-1\,\,\mbox{if}\,\,x < c\,\,\left|x-c\right|\geq\frac{1}{n},\\
G^{\prime}\left(n\left(x-c\right)\right),\,\mbox{if}\,\left|x-c\right|<\frac{1}{n}.
\end{cases}
\end{equation*}
Therefore 
$$G_{n}^{\prime}(x)\to \mbox{sg}(x-c)\,\,\mbox{as}\,\,n\to\infty.$$
Multiplying \eqref{regularized.IBVP.a} by $G^{\prime}_{n}(x)$ to get 
\begin{equation*}
G^{\prime}_{n}\left(u^{\varepsilon}\right)\,\frac{\partial u^{\varepsilon}}{\partial t} + \displaystyle\sum_{j=1}^{d}\,G^{\prime}_{n}\left(u^{\varepsilon}\right)\,\frac{\partial}{\partial x_{j}}f_{j}\left(u^{\varepsilon}\right)=\varepsilon\,G^{\prime}_{n}\left(u^{\varepsilon}\right)\displaystyle\sum_{j=1}^{d}\frac{\partial}{\partial x_{j}}\left(B(u^{\varepsilon})\,\frac{\partial u^{\varepsilon}}{\partial x_{j}}\right)
\end{equation*}	
Applying chain rule, we get
\begin{equation}\label{KineticFormulation.Equation1}
\frac{\partial}{\partial t}\left( G_{n}\left(u^{\varepsilon}\right)\right) + \displaystyle\sum_{j=1}^{d}\,G^{\prime}_{n}\left(u^{\varepsilon}\right)\,f^{\prime}_{j}\left(u^{\varepsilon}\right)\frac{\partial u^{\varepsilon}}{\partial x_{j}}=\varepsilon\,G^{\prime}_{n}\left(u^{\varepsilon}\right)\displaystyle\sum_{j=1}^{d}\frac{\partial}{\partial x_{j}}\left(B(u^{\varepsilon})\,\frac{\partial u^{\varepsilon}}{\partial x_{j}}\right)
\end{equation}
Denote
$$q_{n}(z):=\int_{k}^{z}\,G^{\prime}_{n}\left(v\right)\,f^{\prime}(v)\,dv.$$
 Using $q_{n}$ in \eqref{KineticFormulation.Equation1}, we obtain
\begin{equation*}
\frac{\partial}{\partial t}\left( G_{n}\left(u^{\varepsilon}\right)\right) + \displaystyle\sum_{j=1}^{d}\,q_{nj}^{\prime}(u^{\varepsilon})\frac{\partial u^{\varepsilon}}{\partial x_{j}} = \varepsilon\,G^{\prime}_{n}\left(u^{\varepsilon}\right)\displaystyle\sum_{j=1}^{d}\frac{\partial}{\partial x_{j}}\left(B(u^{\varepsilon})\,\frac{\partial u^{\varepsilon}}{\partial x_{j}}\right) 
\end{equation*}
Passing to the limit as $n\to\infty$ in the sense of distribution, we obtain 
\begin{equation}\label{KineticFormulation.Equation2}
\frac{\partial}{\partial t}\eta(u^{\varepsilon};c) + \displaystyle\sum_{j=1}^{d}\,\frac{\partial}{\partial x_{j}}q_{j}(u^{\varepsilon};c) = \varepsilon\,\eta^{\prime}\left(u^{\varepsilon};c\right)\displaystyle\sum_{j=1}^{d}\frac{\partial}{\partial x_{j}}\left(B(u^{\varepsilon})\,\frac{\partial u^{\varepsilon}}{\partial x_{j}}\right)\,\,\mbox{in}\,\,\mathcal{D}^{\prime}\left(\Omega\times (0,T)\right). 
\end{equation}
From \eqref{KineticFormulation.Equation2}, we obtain
\begin{equation}\label{KineticFormulation.Equation3}
\begin{split}
\frac{\partial}{\partial t}\left[\eta(u^{\varepsilon};c)-\eta(0;c)\right] &+ \displaystyle\sum_{j=1}^{d}\,\frac{\partial}{\partial x_{j}}\left[q_{j}(u^{\varepsilon};c)-q_{j}(0;c)\right]\\ 
&=\varepsilon\,\eta^{\prime}\left(u^{\varepsilon};c\right)\displaystyle\sum_{j=1}^{d}\frac{\partial}{\partial x_{j}}\left(B(u^{\varepsilon})\,\frac{\partial u^{\varepsilon}}{\partial x_{j}}\right)\,\,\mbox{in}\,\,\mathcal{D}^{\prime}\left(\Omega\times (0,T)\right). 
\end{split}
\end{equation}
Differentiating \eqref{KineticFormulation.Equation3} with respect to $c$ we get
\begin{equation}\label{KineticFormulation.Equation3ABC1}
\begin{split}
\frac{\partial}{\partial t}\left(\frac{\partial}{\partial c}\left(\eta(u^{\varepsilon};c)-\eta(0;c)\right)\right) &+ \displaystyle\sum_{j=1}^{d}\,\frac{\partial}{\partial x_{j}}\left(\frac{\partial}{\partial c}\left(q_{j}(u^{\varepsilon};c)-q_{j}(0;c)\right)\right)\\ 
&=\frac{\partial}{\partial c}\left(\varepsilon\,\eta^{\prime}\left(u^{\varepsilon};c\right)\displaystyle\sum_{j=1}^{d}\frac{\partial}{\partial x_{j}}\left(B(u^{\varepsilon})\,\frac{\partial u^{\varepsilon}}{\partial x_{j}}\right)\right)\,\,\mbox{in}\,\,\mathcal{D}^{\prime}\left(\Omega\times (0,T)\right). 
\end{split}
\end{equation}
Note that $\eta(u^{\varepsilon};c)=\left|u^{\varepsilon}-c\right|$, $\eta(0;c)=\left|c\right|$. Denote
$$P\left(u^{\varepsilon};c\right):=\left|u^{\varepsilon}-c\right|-|c|.$$
We now compute the derivative of $P$ with respect to $c$.
\begin{eqnarray}\label{KineticFormulation.Equation4}
\frac{\partial}{\partial c}P(u^{\varepsilon};c) &=&\frac{\partial}{\partial c}\left(|u^{\varepsilon}-c|-|c|\right)\nonumber\\
&=& \frac{\partial}{\partial c}|u^{\varepsilon}-c|-\frac{\partial}{\partial c}|c|\nonumber\\
&=& \mbox{sg}(u^{\varepsilon}-c)-\mbox{sg }(c)\nonumber\\
\end{eqnarray}
For $j=1,2,\cdots,d$, denote
\begin{eqnarray}\nonumber
Q_{j}(u^{\varepsilon};c) &:=& q_{j}(u^{\varepsilon};c)-q_{j}(0;c),\nonumber\\
&=& \mbox{sg}\,(u^{\varepsilon}-c)\left(f_{j}(u^{\varepsilon})-f_{j}(c)\right) + \mbox{sg}(c)\,\left(f_{j}(0)-f_{j}(c)\right).\nonumber	
\end{eqnarray}	
Let us compute the derivative of $Q_{j}^{\varepsilon}$ with respect to $c$.
\begin{equation*}
\begin{split}
\frac{\partial}{\partial c}Q_{j}(u^{\varepsilon};c)
= 2\,\left(f_{j}(u^{\varepsilon}(x,t))-f_{j}(u^{\varepsilon}(x,t))\right)\,\delta_{c=u^{\varepsilon}(x,t)} -f_{j}^{\prime}(c)\,\mbox{sg}\,(u^{\varepsilon}-c)\\ +2\,\left(f_{j}(0)-f_{j}(0)\right)\,\delta_{c=0} -f_{j}^{\prime}(c)\,\mbox{sg}\left(c\right).	
\end{split} 
\end{equation*}
Therefore we have
\begin{equation}\label{KineticFormulation.Equation5}
\frac{\partial}{\partial c}Q_{j}(u^{\varepsilon};c)
=-f_{j}^{\prime}(c)\left(\,\,\mbox{sg}\left(u^{\varepsilon}-c\right)+\,\mbox{sg}(c)\,\right).
\end{equation}
Let $\phi\in\mathcal{D}\left(\Omega\times\mathbb{R}\times (0,T)\right)$. For $j\in\left\{1,2,\cdots,d\right\}$, we compute
\begin{eqnarray}\label{KineticFormulation.Equation3ABC2}
\left<\frac{\partial^{2}}{\partial x_{j}\partial c}Q_{j}(u^{\varepsilon};c),\phi\right> &=& -\left<\frac{\partial}{\partial c}Q_{j}(u^{\varepsilon};c),\frac{\partial\phi}{\partial x_{j}}\right>\nonumber\\
&=&  \Big<f_{j}^{\prime}(c)\left(\mbox{sg}\left(u^{\varepsilon}-c\right)+\mbox{sg}(c)\right), \frac{\partial\phi}{\partial x_{j}}\Big>\nonumber\\ 
&=& \Big<f_{j}^{\prime}(c)\frac{\partial}{\partial x_{j}}\left(\mbox{sg}\left(u^{\varepsilon}-c\right)+\mbox{sg}(c)\right),\phi\Big>\nonumber\\
&=&  \Big<f_{j}^{\prime}(c)\frac{\partial}{\partial x_{j}}\left(\mbox{sg}\left(u^{\varepsilon}-c\right)-\mbox{sg}(c)\right),\phi\Big>\nonumber\\
&=& \Big<f_{j}^{\prime}(c)\left(\mbox{sg}\left(u^{\varepsilon}-c\right)-\mbox{sg}(c)\right),\frac{\partial\phi}{\partial x_{j}}\Big>.
\end{eqnarray}
Denote
$$\chi^{\varepsilon}(x,t ;c):=\chi_{u^{\varepsilon}}(c),$$
where
$$\chi_{u^{\varepsilon}}(c):= \begin{cases}
1 ,\,\,\mbox{if}\,\,u^{\varepsilon}<c< 0,\\
-1 ,\,\,\mbox{if}\,\,0< c< u^{\varepsilon},\nonumber\\
0 ,\,\,\mbox{otherwise}\nonumber
\nonumber
\end{cases}\\$$
Also observe that $\int_{-\infty}^{\infty}\,\chi_{u^{\varepsilon}}(c)\,dc=-u^{\varepsilon}$ and $\chi^{\varepsilon}=\chi_{u^{\varepsilon}}(c)$. Dividing \eqref{KineticFormulation.Equation3ABC1} by 2 and applying the above computations, we get \eqref{KineticFormulation.Equation612}. \\
Denote 
$$p_{\mbox{min}}^{\varepsilon}:=\mbox{min}\left\{\displaystyle\min_{(x,t)\in\overline{\Omega_{T}}}\frac{\varepsilon}{2}\,\displaystyle\sum_{j=1}^{d}\,1\cdot\frac{\partial}{\partial x_{j}}\left(B(u^{\varepsilon})\frac{\partial u^{\varepsilon}}{\partial x_{j}}\right),\,\displaystyle\min_{(x,t)\in\overline{\Omega_{T}}}\frac{\varepsilon}{2}\,\displaystyle\sum_{j=1}^{d}\,\left(-1\right)\cdot\frac{\partial}{\partial x_{j}}\left(B(u^{\varepsilon})\frac{\partial u^{\varepsilon}}{\partial x_{j}}\right),\,0\right\}$$\\

Since $\mbox{range}\left(\eta^{\prime}\left(u^{\varepsilon};c\right)\right)=\mbox{sgn}\left(u^{\varepsilon}-c\right)=\left\{-1,0,1\right\}$, $p_{\mbox{min}}^{\varepsilon}$ is independent of $c$ variable and we have 
$$\frac{\varepsilon}{2}\,\displaystyle\sum_{j=1}^{d}\eta^{\prime}(u^{\varepsilon};c)\frac{\partial}{\partial x_{j}}\left(B(u^{\varepsilon})\frac{\partial u^{\varepsilon}}{\partial x_{j}}\right)\geq p_{\mbox{min}}^{\varepsilon}$$ and 
\begin{equation*}
\begin{split}
\frac{\partial}{\partial c}\left(\frac{\varepsilon}{2}\,\displaystyle\sum_{j=1}^{d}\eta^{\prime}(u^{\varepsilon};c)\frac{\partial}{\partial x_{j}}\left(B(u^{\varepsilon})\frac{\partial u^{\varepsilon}}{\partial x_{j}}\right)\right)\\
=\frac{\partial}{\partial c}\left(\frac{\varepsilon}{2}\,\displaystyle\sum_{j=1}^{d}\eta^{\prime}(u^{\varepsilon};c)\frac{\partial}{\partial x_{j}}\left(B(u^{\varepsilon})\frac{\partial u^{\varepsilon}}{\partial x_{j}}\right)-p^{\varepsilon}_{\mbox{min}}\right)
\end{split}
\end{equation*}
Therefore \eqref{KineticFormulation.Equation3} defines a kinetic formulation with $\chi^{\varepsilon}(x,t ;c):=\chi_{u^{\varepsilon}}(c)$ according to the definition of \cite[p.170]{Lions}\,$\blacksquare$
\end{proof}\\
\subsection{Boundedness of the Sequence of Measures}\label{Boundedness.sequenceofmeasures}
Denote
 $$m^{\varepsilon}:=\frac{\varepsilon}{2}\displaystyle\sum_{j=1}^{d}\,\eta^{\prime}\left(u^{\varepsilon};c\right)\frac{\partial }{\partial x_{j}}\left(B(u^{\varepsilon})\frac{\partial u^{\varepsilon}}{\partial x_{j}}\,\right).$$ 
 Firstly, we prove that for every $M>0$, the sequence of measures $\left(m^{\varepsilon}\right)$ is bounded in $\mathcal{M}\left(\Omega\times\left(-M,M\right)\times (0,T)\right)$. Secondly, since we want to apply Velocity Averaging lemma on $\mathbb{R}^{d}\times\mathbb{R}\times (0,\infty)$, we extends this sequence of measures $\left(m^{\varepsilon}\right)$ to  $\mathbb{R}^{d}\times\mathbb{R}\times (0,\infty)$ using standard arguement of Riesz representation theorem of dual spaces keeping the total variation same. This extension is denoted by $\widetilde{\widetilde{m^{\varepsilon}}}$. These results are proved in this subsection. Then we can assume a standard result from \cite{Lions} that the derivative of $\widetilde{\widetilde{m^{\varepsilon}}}$ with respect to $c$ can be expressed as the composition of inverses of Riesz potentials which is one of the requirements to apply Velocity Averaging lemma. \\
 For a real number $M>0$, denote
 $$\Omega_{MT}:=\Omega\times (-M,M)\times (0,T).$$
 $$C^{1}_{0}\left(\overline{\Omega_{MT}}\right):=\left\{v\in C^{1}\left(\overline{\Omega_{MT}}\right)\,:\,\,v(x,c,t)=0\,\,\,\mbox{on}\,\partial\Omega_{MT}\right\}.$$
 The next result shows that the sequence of measures $\left(m^{\varepsilon}\right)$ is bounded in $\mathcal{M}\left(\Omega_{MT}\right)$.
\begin{proposition}\label{MeasureBounded.1}
The sequence of measures $\left(m^{\varepsilon}\right)$ is bounded in $\mathcal{M}(\overline{\Omega_{MT}})$.
\end{proposition}
\begin{proof}
We prove Proposition \ref{MeasureBounded.1}	in two steps. In Step 1, we prove that the sequence $\left(m^{\varepsilon}\right)$ is bounded in $\left(C^{1}_{0}\left(\overline{\Omega_{MT}}\right)\right)^{\ast}$. In Step 2, We prove the measure representation of $\left(m^{\varepsilon}\right)$.\\
\vspace{0.1cm}\\
\textbf{Step 1:} For $n\in\mathbb{N}$ and $s\in\mathbb{R}$, denote $\mbox{sg}_{n}(s):=\mbox{tanh}\,(ns)$. Then $\mbox{sg}_{n}(s)\to\mbox{sg}(s)$ as $n\to\infty$ and $|s|\,\mbox{sg}_{n}^{\prime}(s)\leq 4$.	\\ 
Let $\phi\in C^{1}_{0}\left(\overline{\Omega_{MT}}\right)$. We consider
\begin{equation}\label{KineticFormulation.Equation613} 
\begin{split}
\left<m^{\varepsilon},\phi\right>=\frac{\varepsilon}{2}\,\displaystyle\sum_{j=1}^{d}\int_{\Omega_{MT}}\eta^{\prime}(u^{\varepsilon};c)\frac{\partial}{\partial x_{j}}\left(B(u^{\varepsilon})\frac{\partial u^{\varepsilon}}{\partial x_{j}}\right)\phi(x,c,t)\,dx\,dc\,dt\\ =\frac{\varepsilon}{2}\,\displaystyle\sum_{j=1}^{d}\int_{\Omega_{MT}}\mbox{sg}(u^{\varepsilon}-c)\frac{\partial}{\partial x_{j}}\left(B(u^{\varepsilon})\frac{\partial u^{\varepsilon}}{\partial x_{j}}\right)\phi(x,c,t)\,dx\,dc\,dt,
\\= \frac{\varepsilon}{2}\,\displaystyle\sum_{j=1}^{d}\displaystyle\lim_{n\to\infty}\int_{\Omega_{MT}}\mbox{sg}_{n}\left(|u^{\varepsilon}-c|\right)\,\mbox{sg}(u^{\varepsilon}-c)\left(\frac{\partial}{\partial x_{j}}\left(B(u^{\varepsilon})\frac{\partial u^{\varepsilon}}{\partial x_{j}}\right)\right)\phi(x,c,t)\,dx\,dc\,dt.
\end{split}
\end{equation}
 Since $u^{\varepsilon}\in C^{2+\beta,\,1+\frac{\beta}{2}}\left(\overline{\Omega_{T}}\right)$, then $u^{\varepsilon}-c \in H^{1}\left(\Omega_{MT}\right)$ and $ |u^{\varepsilon}-c| \in H^{1}\left(\Omega_{MT}\right)$. 
 Hence $\mbox{sg}_{n}\left( |u^{\varepsilon}-c|\right)\in H^{1}\left(\Omega_{MT}\right)$. As a consequence, we have $\mbox{sg}_{n}\left( |u^{\varepsilon}-c| \right)\,\phi\in H^{1}\left(\Omega_{MT}\right)$ and $B(u^{\varepsilon})\,\frac{\partial u^{\varepsilon}}{\partial x_{j}}\in  H^{1}\left(\Omega_{MT}\right)$. Therefore we have the following product rule ( see \cite[p.17]{MR2309679})
 \begin{eqnarray}\label{Measure.Correction.Equation1}
 \frac{\partial}{\partial x_{j}}\left(\mbox{sg}_{n}\left(|u^{\varepsilon}-c|\right)\,\phi\,B(u^{\varepsilon})\,\frac{\partial u^{\varepsilon}}{\partial x_{j}}\right) &=&\mbox{sg}_{n}\left(|u^{\varepsilon}-c|\right)\,\phi\,\frac{\partial}{\partial x_{j}}\left(B(u^{\varepsilon})\,\frac{\partial u^{\varepsilon}}{\partial x_{j}}\right)\nonumber\\
 && + B(u^{\varepsilon})\,\frac{\partial u^{\varepsilon}}{\partial x_{j}}\,\frac{\partial}{\partial x_{j}}\left(\mbox{sg}_{n}\left(|u^{\varepsilon}-c|\right)\,\phi\right).
 \end{eqnarray}
 Therefore we have
 \begin{eqnarray}\label{Measure.Correction.Equation2}
 \mbox{sg}_{n}\left(|u^{\varepsilon}-c|\right)\,\phi\,\frac{\partial}{\partial x_{j}}\left(B(u^{\varepsilon})\,\frac{\partial u^{\varepsilon}}{\partial x_{j}}\right) &=&\frac{\partial}{\partial x_{j}}\left(\mbox{sg}_{n}\left(|u^{\varepsilon}-c|\right)\,\phi\,B(u^{\varepsilon})\,\frac{\partial u^{\varepsilon}}{\partial x_{j}}\right)\nonumber\\
 &&-B(u^{\varepsilon})\,\frac{\partial u^{\varepsilon}}{\partial x_{j}}\,\,\mbox{sg}_{n}^{\prime}\left(|u^{\varepsilon}-c|\right)\,\,\frac{\partial u^{\varepsilon}}{\partial x_{j}}\,\,\mbox{sg}\left(u^{\varepsilon}-c\right)\,\phi\nonumber\\
 &&-B(u^{\varepsilon})\,\frac{\partial u^{\varepsilon}}{\partial x_{j}}\,\,\mbox{sg}_{n}\left(|u^{\varepsilon}-c|\right)\,\frac{\partial\phi}{\partial x_{j}}
 \end{eqnarray}
 \begin{equation}\label{KineticFormulation.Equation614}
 \begin{split}
 \frac{\varepsilon}{2}\,\displaystyle\sum_{j=1}^{d}\displaystyle\lim_{n\to\infty}\int_{\Omega_{MT}}\mbox{sg}_{n}\left(|u^{\varepsilon}-c|\right)\,\mbox{sg}(u^{\varepsilon}-c)\left(\frac{\partial}{\partial x_{j}}\left(B(u^{\varepsilon})\frac{\partial u^{\varepsilon}}{\partial x_{j}}\right)\right)\phi(x,c,t)\,dc\,dx\,dt\\
 =-\frac{\varepsilon}{2}\,\displaystyle\sum_{j=1}^{d}\displaystyle\lim_{n\to\infty}\Big\{\int_{\Omega_{MT}}\,\mbox{sg}(u^{\varepsilon}-c)\,\frac{\partial}{\partial x_{j}}\left(\mbox{sg}_{n}\left(|u^{\varepsilon}-c|\right)\,\phi\,B(u^{\varepsilon})\,\frac{\partial u^{\varepsilon}}{\partial x_{j}}\right)\,dc\,dx\,dt\Big\}\\
 -\frac{\varepsilon}{2}\,\displaystyle\sum_{j=1}^{d}\displaystyle\lim_{n\to\infty}\Big\{\int_{\Omega_{MT}}\mbox{sg}(u^{\varepsilon}-c)\,B(u^{\varepsilon})\,\frac{\partial u^{\varepsilon}}{\partial x_{j}}\,\,\mbox{sg}_{n}^{\prime}\left(|u^{\varepsilon}-c|\right)\,\,\frac{\partial u^{\varepsilon}}{\partial x_{j}}\,\,\mbox{sg}\left(u^{\varepsilon}-c\right)\,\phi\,dc\,dx\,dt\Big\}\\
 -\frac{\varepsilon}{2}\,\displaystyle\sum_{j=1}^{d}\displaystyle\lim_{n\to\infty}\Big\{\int_{\Omega_{MT}}\mbox{sg}_{n}\left(|u^{\varepsilon}-c|\right)\,\mbox{sg}(u^{\varepsilon}-c)\left(B(u^{\varepsilon})\frac{\partial u^{\varepsilon}}{\partial x_{j}}\right)\frac{\partial\phi}{\partial x_{j}}\,dc\,dx\,dt\Big\}
 \end{split}
 \end{equation}
 Since $u^{\varepsilon}=c$ \,gives\, $\frac{\partial u^{\varepsilon}}{\partial x_{j}}=0$, therefore we have
  $$\displaystyle\lim_{n\to\infty}\mbox{sg}_{n}^{\prime}(|u^{\varepsilon}-c|)=\displaystyle\lim_{n\to\infty}\mbox{sg}_{n}^{\prime}\left(\left|\frac{\partial u^{\varepsilon}}{\partial x_{j}}\right|\right).$$ 
  Applying the above observation in \eqref{KineticFormulation.Equation614}, we obtain
  \begin{equation}\label{Measure.Correction.Equation12}
  \begin{split}
  \frac{\varepsilon}{2}\,\displaystyle\sum_{j=1}^{d}\displaystyle\lim_{n\to\infty}\int_{\Omega_{MT}}\mbox{sg}_{n}\left(|u^{\varepsilon}-c|\right)\,\mbox{sg}(u^{\varepsilon}-c)\left(\frac{\partial}{\partial x_{j}}\left(B(u^{\varepsilon})\frac{\partial u^{\varepsilon}}{\partial x_{j}}\right)\right)\phi(x,c,t)\,dc\,dx\,dt\\
  =-\frac{\varepsilon}{2}\,\displaystyle\sum_{j=1}^{d}\displaystyle\lim_{n\to\infty}\Big\{\int_{\Omega_{MT}}\,\mbox{sg}(u^{\varepsilon}-c)\,\frac{\partial}{\partial x_{j}}\left(\mbox{sg}_{n}\left(|u^{\varepsilon}-c|\right)\,\phi\,B(u^{\varepsilon})\,\frac{\partial u^{\varepsilon}}{\partial x_{j}}\right)\,dc\,dx\,dt\Big\}\\
  -\frac{\varepsilon}{2}\,\displaystyle\sum_{j=1}^{d}\displaystyle\lim_{n\to\infty}\Big\{\int_{\Omega_{MT}}\mbox{sg}(u^{\varepsilon}-c)\,B(u^{\varepsilon})\,\frac{\partial u^{\varepsilon}}{\partial x_{j}}\,\,\mbox{sg}_{n}^{\prime}\left(\left|\frac{\partial u^{\varepsilon}}{\partial x_{j}}\right|\right)\,\,\frac{\partial u^{\varepsilon}}{\partial x_{j}}\,\,\mbox{sg}\left(u^{\varepsilon}-c\right)\,\phi\,dc\,dx\,dt\Big\}\\
  -\frac{\varepsilon}{2}\,\displaystyle\sum_{j=1}^{d}\displaystyle\lim_{n\to\infty}\Big\{\int_{\Omega_{MT}}\mbox{sg}_{n}\left(|u^{\varepsilon}-c|\right)\,\mbox{sg}(u^{\varepsilon}-c)\left(B(u^{\varepsilon})\frac{\partial u^{\varepsilon}}{\partial x_{j}}\right)\frac{\partial\phi}{\partial x_{j}}\,dc\,dx\,dt\Big\}
  \end{split}
  \end{equation}
  Consider the first term on the RHS of \eqref{Measure.Correction.Equation12}, we get 
   \begin{equation}\label{Measure.Correction.Equation13}
   \begin{split}
   \int_{\Omega_{MT}}\,\mbox{sg}(u^{\varepsilon}-c)\,\frac{\partial}{\partial x_{j}}\left(\mbox{sg}_{n}\left(|u^{\varepsilon}-c|\right)\,\phi\,B(u^{\varepsilon})\,\frac{\partial u^{\varepsilon}}{\partial x_{j}}\right)\,dc\,dx\,dt \\
   =\left<\mbox{sg}(u^{\varepsilon}-c)\, ,\,\,\frac{\partial}{\partial x_{j}}\left(\mbox{sg}_{n}\left(|u^{\varepsilon}-c|\right)\,\phi\,B(u^{\varepsilon})\,\frac{\partial u^{\varepsilon}}{\partial x_{j}}\right)\right>\\
   =-\left<\delta_{c=u^{\varepsilon}}\,,\,\,\mbox{sg}_{n}\left(|u^{\varepsilon}-c|\right)\,\phi\,B(u^{\varepsilon})\,\frac{\partial u^{\varepsilon}}{\partial x_{j}}\right>=0.
   \end{split}
   \end{equation}
   The second term converges to zero as $n\to\infty$ as an application of Dominated convergence theorem. This is because as $\left|\frac{\partial u^{\varepsilon}}{\partial x_{j}}\right|\mbox{sg}_{n}^{\prime}\left(\left|\frac{\partial u^{\varepsilon}}{\partial x_{j}}\right|\right)\leq 4$ and $\left|\frac{\partial u^{\varepsilon}}{\partial x_{j}}\right|\,\mbox{sg}_{n}^{\prime}\left(\left|\frac{\partial u^{\varepsilon}}{\partial x_{j}}\right|\right)\to 0$ as $n\to\infty$. 
Therefore using all the above computation, we arrive at
\begin{equation}\label{Measure.Correction.equation14}
\left<m^{\varepsilon},\phi\right> = -\frac{\varepsilon}{2}\,\displaystyle\sum_{j=1}^{d}\Big\{\int_{\Omega_{MT}}\,\mbox{sg}(u^{\varepsilon}-c)\left(B(u^{\varepsilon})\frac{\partial u^{\varepsilon}}{\partial x_{j}}\right)\frac{\partial\phi}{\partial x_{j}}\,dc\,dx\,dt\Big\}
\end{equation}
Therefore applying H\"{o}lder inequality we have
\begin{eqnarray}\label{Measure.Correction.Equation15}
\left|\left<m^{\varepsilon},\phi\right>\right| &=& \sqrt{\varepsilon}\|B\|_{L^{\infty}(I)}\displaystyle\sum_{j=1}^{d}\left(\sqrt{\varepsilon}\left\|\frac{\partial u^{\varepsilon}}{\partial x_{j}}\right\|_{L^{2}\left(\Omega_{MT}\right)}\,\right)\left\|\frac{\partial\phi}{\partial x_{j}}\right\|_{L^{2}\left(\Omega_{MT}\right)}\nonumber\\
&\leq& \sqrt{\varepsilon}\|B\|_{L^{\infty}(I)}\,\sqrt{2M}\,\left(\displaystyle\sum_{j=1}^{d}\left(\sqrt{\varepsilon}\left\|\frac{\partial u^{\varepsilon}}{\partial x_{j}}\right\|_{L^{2}\left(\Omega_{T}\right)}\,\right)\right)\,\mbox{Vol}\left(\Omega_{MT}\right)\,\|\phi\|_{C^{1}\left(\Omega_{MT}\right)}\nonumber\\
{}
\end{eqnarray}
Since for $j\in\left\{1,2,\cdot,d\right\}$, the sequence $\left(\sqrt{\varepsilon}\left\|\frac{\partial u^{\varepsilon}}{\partial x_{j}}\right\|_{L^{2}\left(\Omega_{T}\right)}\right)$ is bounded, $\mbox{Vol}\left(\Omega_{MT}\right)<\infty$ and $0<\varepsilon<1$, therefore the sequence $\left(m^{\varepsilon}\right)$ is bounded in $\left(C^{1}_{0}\left(\overline{\Omega_{MT}}\right)\right)^{\ast}$. \\
\vspace{0.1cm}\\
\textbf{Step 2:} We know that $C^{1}\left(\overline{\Omega_{MT}}\right)$ is a Banach space with $\|.\|_{C^{1}\left(\Omega_{MT}\right)}$ norm. Since $C^{1}_{0}\left(\overline{\Omega_{MT}}\right)$ is a closed subspace of $C^{1}\left(\overline{\Omega_{MT}}\right)$, therefore $C^{1}_{0}\left(\overline{\Omega_{MT}}\right)$ is a Banach space. We also have that $C^{1}_{0}\left(\overline{\Omega_{MT}}\right)$ is separable. Consider the map $P:C^{1}_{0}\left(\overline{\Omega_{MT}}\right)\to \left(C\left(\overline{\Omega_{MT}}\right)\right)^{d+1}$ defined by
$$P(u):=\left(u,\frac{\partial u}{\partial x_{1}},\frac{\partial u}{\partial x_{2}},\cdots,\frac{\partial u}{\partial x_{d}}\right).$$ 
Therefore $P$ is an isometric isomorphism from $C^{1}_{0}\left(\overline{\Omega_{MT}}\right)$ to $\mbox{Range}(P)$. Hence $\mbox{Range}(P)$ is a closed subspace of $\left(C\left(\overline{\Omega_{MT}}\right)\right)^{d+1}$. Since $P$ is an isometric isomorphism, for $z\in \mbox{Range}(P)$ and there exists unique $u\in C^{1}_{0}\left(\overline{\Omega_{MT}}\right)$ such that a linear functional on $\mbox{Range}(P)$ is defined as 
$$\left(m^{\varepsilon}\right)^{\ast}(z) :=\left(m^{\varepsilon}\right)^{\ast}(P(u))= m^{\varepsilon}(P^{-1}(z))= m^{\varepsilon}(u).$$ Since  the sequence $\left(m^{\varepsilon}\right)$ is bounded in $\left(C^{1}_{0}\left(\overline{\Omega_{MT}}\right)\right)^{\ast}$,  the sequence $\left(m^{\varepsilon}\right)^{\ast}$ is bounded in $\left(\mbox{Range}(P)\right)^{\ast}$. Hence by Hahn-Banach extension theorem, there exists a norm preserving extension $\widetilde{\left(m^{\varepsilon}\right)^{\ast}}$ of $\left(m^{\varepsilon}\right)^{\ast}$ to $\left(C\left(\overline{\Omega_{MT}}\right)\right)^{d+1}$. By Riesz representation theorem for measures there exists $\mu^{\varepsilon}\in\left(\mathcal{M}\left(\overline{\Omega_{MT}}\right)\right)^{d+1}$ vector measures such that for $u\in \left(C\left(\overline{\Omega_{MT}}\right)\right)^{d+1}$, we have 
$$\widetilde{\left(m^{\varepsilon}\right)^{\ast}}(u)=\displaystyle\sum_{j=0}^{d}\left<\mu_{j}^{\varepsilon},\,u_{j}\right>.$$
Therefore for $u\in C^{1}_{0}\left(\overline{\Omega_{MT}}\right)$, we get
\begin{eqnarray}\label{measure.correction.equation16}
m^{\varepsilon}(u)&=&\left(m^{\varepsilon}\right)^{\ast}(P(u))=\widetilde{\left(m^{\varepsilon}\right)^{\ast}}(P(u))\nonumber\\
&=& \left<\mu_{0}^{\varepsilon},u\right> +\displaystyle\sum_{j=1}^{d}\left<\mu_{j}^{\varepsilon},\frac{\partial u^{\varepsilon}}{\partial x_{j}}\right>
\end{eqnarray} 
\hfill{$\blacksquare$}
\end{proof}
\begin{proposition}\label{Measure.Correction.proposition1}
Let $j\in\left\{0,1,2,\cdots,d\right\}$. There exist extension $\widetilde{\mu_{j}^{\varepsilon}}$ in $\mathcal{M}\left(\mathbb{R}^{d}\times\mathbb{R}\times(0,\infty)\right)$ such that $\|\mu_{j}^{\varepsilon}\|=\|\widetilde{\mu_{j}^{\varepsilon}}\|$. 	
\end{proposition}
\begin{proof}
For $j\in\left\{0,1,2,\cdots,d\right\}$, since $\left(\mu_{j}^{\varepsilon}\right)$ is a bounded sequence of measures, $\left(\mu_{j}^{\varepsilon}\right)$ can be identified as elements $\left(T_{\mu_{j}^{\varepsilon}}\right)$ in $\left(C\left(\overline{\Omega_{MT}}\right)\right)^{\ast}$, where $C\left(\overline{\Omega_{MT}}\right)$ denotes the set of bounded continuous functions. Applying Tietze extension theorem, we extend continuous function $f$ in $C\left(\overline{\Omega_{MT}}\right)$ to a continuous function $\tilde{f}$ in $C_{b}\left(\mathbb{R}^{d}\times\mathbb{R}\times(0,\infty)\right)$ such that 
$$\mbox{sup}\left\{\left|f(x,c,t)\right|\,\,:\,\,\left(x,c,t\right)\in\overline{\Omega_{MT}}\right\}=\mbox{sup}\left\{\left|\tilde{f}(x,c,t)\right|\,\,:\,\,\left(x,c,t\right)\in\mathbb{R}^{d}\times\mathbb{R}\times (0,\infty)\right\},$$
where the set $C_{b}\left(\mathbb{R}^{d}\times\mathbb{R}\times(0,\infty)\right)$ denotes the set of bounded continuous functions. Denote
$$Y:=\left\{\tilde{f}\,\,:\,\,\tilde{f}\Big|_{\overline{\Omega_{MT}}}=f\in C\left(\overline{\Omega_{MT}}\right)\right\}.$$
Let $T\in Y^{\ast}$, then 
\begin{eqnarray}\label{Kinetic.ReizReprentation.eqn1}
\left\|T\right\|_{Y^{\ast}}&=&\mbox{sup}\left\{\left|T(\tilde{f})\right|\,\,:\,\,\|\tilde{f}\|_{\infty}\leq 1\right\},\nonumber\\
&=& \mbox{sup}\left\{\left|T(\tilde{f})\right|\,\,:\,\,\|\tilde{f}\|_{\infty}= \|f\|_{\infty}\leq 1\right\},\nonumber\\
&=& \mbox{sup}\left\{\left|T(f)\right|\,\,:\,\,\|f\|_{\infty}\leq 1\right\},\nonumber\\
&=& \left\|T\right\|_{\left(C\left(\Omega_{MT}\right)\right)^{\ast}}.
\end{eqnarray}
The map $\mu_{j}^{\varepsilon}\mapsto T_{\mu_{j}^{\varepsilon}}$ is an isometric isomorphism from $\mathcal{M}\left(\overline{\Omega_{MT}}\right)$ to $\left(C\left(\Omega_{MT}\right)\right)^{\ast}$. Therefore $\|\mu_{j}^{\varepsilon}\|=\left\|T_{\mu_{j}^{\varepsilon}}\right\|$. \\
Let $T_{Y\mu_{j}^{\varepsilon}}$ be linear maps on $Y$ such that 
$$T_{Y\mu_{j}^{\varepsilon}}\Big|_{C\left(\overline{\Omega_{MT}}\right)}= T_{\mu_{j}^{\varepsilon}}.$$
Then applying \eqref{Kinetic.ReizReprentation.eqn1}, we conclude that $\left\|T_{Y\mu_{j}^{\varepsilon}}\right\|=\|\mu_{j}^{\varepsilon}\|$. Then by Hahn-Banach theorem, we extend $T_{Ym^{\varepsilon}}$ to a continuous linear functional $\widetilde{T}_{\mu_{j}^{\varepsilon}}$ on $C_{b}\left(\mathbb{R}^{d}\times\mathbb{R}\times (0,\infty)\right)$ such that
$$\left\|\widetilde{T}_{\mu_{j}^{\varepsilon}}\right\|=\left\|T_{Y\mu_{j}^{\varepsilon}}\right\|=\|\mu_{j}^{\varepsilon}\|.$$
There exists $\left(\widetilde{\mu_{j}^{\varepsilon}}\right)$ in $\mathcal{M}\left(\mathbb{R}^{d}\times\mathbb{R}\times (0,\infty)\right)$ such that 
$$\|\tilde{\mu_{j}^{\varepsilon}}\|=\left\|\tilde{T}_{\mu_{j}^{\varepsilon}}\right\|=\|\mu_{j}^{\varepsilon}\|.$$
Since $\left(\mu_{j}^{\varepsilon}\right)$ is bounded, we get $\left(\widetilde{\mu_{j}^{\varepsilon}}\right)$ is also bounded\,\,$\blacksquare$\\	
\end{proof}	
\vspace{0.1cm}\\
Applying Proposition \ref{MeasureBounded.1} and Proposition \ref{Measure.Correction.proposition1}, we conclude the following result.
\begin{proposition}\label{Measure.Correction.Proposition3}
Let $\widetilde{\mu^{\varepsilon}}:=\left(\widetilde{\mu_{0}^{\varepsilon}},\widetilde{\mu_{1}^{\varepsilon}},\widetilde{\mu_{2}^{\varepsilon}},\cdots,\widetilde{\mu_{d}^{\varepsilon}}\right)$. Then the sequence of linear functional  $\widetilde{\widetilde{m^{\varepsilon}}}:\left(C\left(\mathbb{R}^{d}\times\mathbb{R}\times(0,\infty)\right)\right)^{d+1}\to\mathbb{R}$ defined by 
$$\widetilde{\widetilde{\left(m^{\varepsilon}\right)}}(u)=\displaystyle\sum_{j=0}^{d}\left<\widetilde{\mu_{j}^{\varepsilon}},\,u_{j}\right>$$ 
is bounded in $\mathcal{M}\left(\mathbb{R}^{d}\times\mathbb{R}\times(0,\infty)\right)$.	
\end{proposition}	
\subsection{Application of Velocity Averaging lemma}
  In this subsection, we extract an {\it a.e.} convergent subsequence of $\left(u^{\varepsilon}\right)$ by applying Velocity Averaging lemma. In order to apply velocity averaging lemma, from equation \eqref{KineticFormulation.Equation612}, we recall the following kinetic equation of \eqref{regularized.IBVP.a}.  
  \begin{equation}\label{KineticFormulation.Equation612A1}
  \frac{\partial \chi^{\varepsilon}}{\partial t} +\displaystyle\sum_{j=1}^{d}f_{j}^{\prime}(c)\,\frac{\partial\chi^{\varepsilon}}{\partial x_{j}} = \frac{\partial m^{\varepsilon}}{\partial c}\,\,\mbox{in}\,\,\mathcal{D}^{\prime}\left(\Omega\times\mathbb{R}\times (0,T)\right).
  \end{equation}
We recall the following result which is used to extract a convergent subsequence of $\left(u^{\varepsilon}\right)$.
\begin{theorem}\cite[p.178]{Lions}\label{VelocityAveragingLemma.12}
Let $1<p\leq 2$, let $h$ be bounded in $L^{p}\left(\mathbb{R}^{d}\times\mathbb{R}_{c}\times\left(0,\infty\right)\right)$, let $g$ belong to a compact set of $L^{p}\left(\mathbb{R}^{d}\times\mathbb{R}_{c}\times\left(0,\infty\right)\right)$ and $r\geq 0$. We assume that $h$ satisfies 
\begin{equation}\label{Velocity.AveragingEquation1.A1}
\frac{\partial h}{\partial t} + \displaystyle\sum_{i=1}^{d}\,f^{\prime}_{i}(c)\frac{\partial h}{\partial x_{i}} =\left(-\Delta_{x,t} +1\right)^{\frac{1}{2}}\left(-\Delta_{c} +1\right)^{\frac{r}{2}}g\,\,\mbox{in}\,\,\mathcal{D}^{\prime}\left(\mathbb{R}^{d}_{x}\times\mathbb{R}_{c}\times (0,\infty)\right),
\end{equation}	
where for $i=1,2,\cdots,d$, each $f_{i}^{\prime}\in C^{l,\alpha}_{loc}$ with $l=r$, $\alpha=1$ if $r$ is an integer, $l=[r]$, $\alpha=r-l$, if $r$ is not an integer. Let $\psi\in L^{p^{\prime}}\left(\mathbb{R}_{c}\right)$ having essential compact support.
Let the following 
\begin{equation*}
\begin{split}
\mbox{meas}\Big\{c\in\mbox{supp}\,\psi,\,\,\tau+\,\left(f_{1}^{\prime}(c),f_{2}^{\prime}(c),\cdots,f_{d}^{\prime}(c)\right)\cdot\xi=0\Big\}=0\,\, \\
\mbox{for all}\,\left(\tau,\xi\right)\in\mathbb{R}\times\mathbb{R}^{d}\,\, \mbox{with}\,\,\tau^{2} +\left|\xi\right|^{2}=1.
\end{split}
\end{equation*}
holds. Then $\int_{R}\,h\psi\,dc$ belongs to a compact set of $L^{p}_{loc}\left(\mathbb{R}^{d}\times (0,\infty)\right)$.	
\end{theorem}	
 We want to apply Theorem \ref{VelocityAveragingLemma.12} with $h=\chi^{\varepsilon}(c)=\chi_{u^{\varepsilon}(x,t)}(c)$. Observe that $h$ is defined on $\mathbb{R}^{d}\times\mathbb{R}\times (0,\infty)$ but in our case $\chi^{\varepsilon}$ is defined on $\Omega\times\mathbb{R}\times\left(0,T\right)$. In order to get values of $\chi^{\varepsilon}$ on $\mathbb{R}^{d}\times\mathbb{R}\times (0,\infty)$,  we need to extend $u^{\varepsilon}$ outside $\Omega_{T}$. Therefore we define the following extension of $u^{\varepsilon}$. Denote
\begin{equation}\nonumber\\
\widetilde{u^{\varepsilon}}(x,t):=\begin{cases}
u^{\varepsilon}(x,t)\,\,\mbox{if}\,\,(x,t)\in \left(\Omega_{T}\cup\left(\Omega\times\left\{0\right\}\right)\cup\left(\partial\Omega\times(0,T)\right)\right)\nonumber\\
0\,\,\,\,\mbox{if}\,\,\,(x,t)\in\left(\mathbb{R}^{d}\times (0,\infty)\right)\setminus \left(\Omega_{T}\cup\left(\Omega\times\left\{0\right\}\right)\cup\left(\partial\Omega\times(0,T)\right)\right).\nonumber\\
\end{cases}
\end{equation}
\vspace{0.2cm}
Denote $\widetilde{\chi^{\varepsilon}}(x,c,t):=\chi_{\widetilde{u^{\varepsilon}(x,t)}}(c)$.
Observe that $\widetilde{\chi^{\varepsilon}}(x,c;t)=\chi_{\widetilde{u^{\varepsilon}}(x,t)}(c)\equiv 0$ if $(x,c,t)\notin\Omega\times\left[-\|u_{0}\|_{L^{\infty}\left(\Omega\right)},\|u_{0}\|_{L^{\infty}\left(\Omega\right)}\right]\times (0,T)$ as $\|u^{\varepsilon}\|_{L^{\infty}\left(\Omega_{T}\right)}\leq \|u_0\|_{L^{\infty}\left(\Omega\right)}$.\\
\vspace{0.1cm}\\
We now assume the following result from \cite[p.178]{Lions} as the definitions of the inverse of the Bessel potential given in \cite{Stein} and \cite{Samko} are same.
	\begin{lemma}\label{Lion.Parthame.Tadmor.p.178}
		{\em For $1<p\leq 2$ and $r> 1+ \frac{d+2}{p}$, there exists a sequence $\left(g^{\varepsilon}\right)$ in a compact set of $L^{p}(\mathbb{R}^{d}\times\mathbb{R}\times(0,\infty))$ such that\, $\frac{\partial \widetilde{\widetilde{{m}^{\varepsilon}}}}{\partial c}=\left(-\Delta_{x,t} + I\right)^{\frac{1}{2}}\left(-\Delta_{c}+I\right)^{\frac{r}{2}}g^{\varepsilon}$ in $\mathcal{D}^{\prime}\left(\mathbb{R}^{d}\times\mathbb{R}\times (0,\infty)\right)$.}
\end{lemma}
Since we are interested in solving IBVP \eqref{regularized.IBVP} in the domain $\Omega_{T}$, we denote
\begin{equation}\nonumber
\widetilde{g^{\varepsilon}}(x,c,t):=\begin{cases}
g^{\varepsilon}(x,c,t)\,\,\mbox{if}\,\,(x,c,t)\in\Omega\times\left(-\|u_{0}\|_{L^{\infty}\left(\Omega\right)},\|u_{0}\|_{L^{\infty}\left(\Omega\right)}\right)\times(0,T),\nonumber\\
0\,\,\,\,\mbox{Otherwise}.\nonumber
\end{cases}
\end{equation}
as $\Omega\times\left\{-\|u_0\|_{L^{\infty}\left(\Omega\right)}\right\}\times (0,T)$, $\Omega\times\left\{\|u_0\|_{L^{\infty}\left(\Omega\right)}\right\}\times (0,T)$ are measure zero sets in $\mathbb{R}^{d+2}$.\\
 \vspace{0.1cm}
 \begin{remark}\label{RMKfractionalDerivative1}
 	For a meaning of $\left(I-\Delta_{x,t}\right)^{\frac{1}{2}}\left(I-\Delta_{v}\right)^{\frac{r}{2}}g^{\varepsilon}$ in Lemma \ref{Lion.Parthame.Tadmor.p.178}, we refer the reader to the discussions of \cite[p.524-p.525]{Samko} and \cite[p.552-p.553]{Samko}.
 \end{remark}
In order to understand Remark \ref{RMKfractionalDerivative1}, the following a few lines are important. For $\alpha> 0$, let $G^{\alpha}:=\left(I-\Delta\right)^{-\frac{\alpha}{2}}$ be the Bessel Potential. For $1\leq p\leq\infty$, let $f\in L^{p}\left(\mathbb{R}^{d}\right)$. Then the  Bessel Potential is defined as the convolution of $f$ with the Bessel Kernel as given in \cite[p.134]{Stein} and \cite[p.540]{Samko}. Let $\Psi^{\alpha}:=\left(I-\Delta\right)^{\frac{\alpha}{2}}$ be the inverse of the Bessel Potential $G^{\alpha}$. A result on the lines of \cite[p.548]{Samko} shows that $\Psi^{\alpha}$ is the inverse of $G^{\alpha}$. For a rigorous discussion on $G^{\alpha}$ and $\Psi^{\alpha}$, we refer the reader to \cite[p.538-p.554]{Samko} and \cite[p.130-p.150]{Stein}.\\
\vspace{0.1cm}\\
The next result is about the range of the inverse of the Bessel potential under a few assumptions.
\begin{lemma}\label{Meaningoffractional.operator}
Let $1<r<\infty$, $1<p<\infty$ and $\alpha>0$ such that $\frac{1}{p}-\frac{\alpha}{d}\leq \frac{1}{r}\leq \frac{1}{p}$. Let $\Psi^{\alpha}:=\left(I-\Delta\right)^{\frac{\alpha}{2}}$. Then for $f\in L^{r}\left(\mathbb{R}^{d}\right)\cap L^{p}\left(\mathbb{R}^{d}\right)$, $\Psi^{\alpha}f\in L^{p}\left(\mathbb{R}^{d}\right)$.	
\end{lemma}	
\begin{proof}
For $\alpha>0$, $1<r<\infty$ and $1<p<\infty$, let $f\in L^{r}\left(\mathbb{R}^{d}\right)\cap L^{p}\left(\mathbb{R}^{d}\right)$. A definition of $\Psi^{\alpha}$ is given (See \cite[p.552]{Samko}) by $\Psi^{\alpha}f:=h_{\alpha}\ast f + D^{\alpha}_{P}f$, where $h_{\alpha}(y)\in L^{1}\left(\mathbb{R}^{d}\right)$ and $D^{\alpha}_{P}f$ is the hypersingular integral given by
$$D^{\alpha}_{P}f=\int_{\mathbb{R}^{d}}\,\frac{\left(\Delta^{l}_{y}f\right)(x)}{|y|^{d+\alpha}}\,P(y)\,dy,$$
where the characteristic $P(y)$ is a polynomial of degree less than $\alpha$. Observe that since $h_{\alpha}(y)\in L^{1}\left(\mathbb{R}^{d}\right)$ and $f\in L^{p}\left(\mathbb{R}^{d}\right)$, we have $h_{\alpha}\ast f\in L^{p}\left(\mathbb{R}^{d}\right)$. In order to prove that $\Psi^{\alpha}\in L^{p}\left(\mathbb{R}^{d}\right)$, it is enough to show that $D^{\alpha}_{P}f\in L^{p}\left(\mathbb{R}^{d}\right)$.\\
\vspace{0.1cm}\\
 We prove $D^{\alpha}_{P}f\in L^{p}\left(\mathbb{R}^{d}\right)$ in six steps. In Step 1, we give defintion of finite difference of a function. In Step 2, we define Riesz potential and state a result which shows that the Riesz potential is a bounded linear map under a few assumptions. In Step 3, we define Riesz differentiation and we also state a result which asserts that the Riesz differentiation is the left inverse of Riesz potential. In Step 4, we define Riesz differentiation with homogenious characteristic and the truncated hypersingular integral with homogenious characteristic. In Step 5, we state a few results of convergence of truncated hypersingular integrals to the Riesz differentiations. In Step 6, we show that $D^{\alpha}_{P}f\in L^{p}\left(\mathbb{R}^{d}\right)$ and as a consequence we complete the proof of Lemma \ref{Meaningoffractional.operator}.\\
\textbf{Step 1: [Finite Differences]}
We now give definitions of finte differences $\left(\Delta^{l}_{y}f\right)(x)$ from \cite[p.499]{Samko} of functions of several variables $f$ with a vector step size $h$. For $x,h\in\mathbb{R}^{d}$, the translation operator $\tau_{h}$ is defined by 
$$\left(\tau_{h}f\right)(x):=f(x-h).$$
The non-centered finite difference is defined by 
\begin{eqnarray}\label{Centered.difference.eqn2}
\left(\Delta_{h}^{l}f\right)(x) &:=&\left(I-\tau_{h}\right)^{l}f\nonumber\\
&=&\displaystyle\sum_{k=0}^{l}\,\left(-1\right)^{k}\,\binom{l}{k}\,f(x-kh).\nonumber
\end{eqnarray} \\
\textbf{Step 2: [Riesz Potential]} For $\alpha>0$, $\alpha\notin\left\{d,d+2,d+4,\cdots\right\}$ and $\phi$ in a suitable function space, the Riesz potential (\cite[p.483]{Samko}) is defined by 
$$I^{\alpha}\phi:=\frac{1}{\gamma_{d}(\alpha)}\int_{\mathbb{R}^{d}}\frac{\phi(y)}{|x-y|^{d-\alpha}},$$
where the normalizing constants (\cite[p.490]{Samko}) are defined by
$$\gamma_{d}=2^{\alpha}\,\pi^{\frac{d}{2}}\,\Gamma\left(\frac{\alpha}{2}\right)/\Gamma\left(\frac{d-\alpha}{2}\right).$$
The next result shows that the Riesz potential is a bounded linear map from $L^{p}\left(\mathbb{R}^{d}\right)$ to $L^{q}\left(\mathbb{R}^{d}\right)$.
\begin{theorem}\cite[p.494]{Samko}\label{RieszPotential.Result1}
Let $1\leq p\leq\infty$, $1\leq q\leq\infty$ and $\alpha>0$. The operator $I^{\alpha}$ is bounded from $L^{p}\left(\mathbb{R}^{d}\right)$ to $L^{q}\left(\mathbb{R}^{d}\right)$ if and only if
$$0<\alpha<d\,,\,\,1<p<\frac{d}{\alpha}\,,\,\,\frac{1}{q}=\frac{1}{p}-\frac{\alpha}{d}.$$	
\end{theorem}
\textbf{Step 3:} 
\begin{definition}\cite[p.499]{Samko}$\left(\mbox{\textbf{Riesz Differentiation}}\right)$\label{RieszDerivation1}
For $\alpha>0$, the realization of the operator $\left(-\Delta\right)^{\frac{\alpha}{2}}$ is the Riesz differention and is given in the form of hypersingular integral	
$$D^{\alpha}f =\frac{1}{e_{d,l}(\alpha)}\int_{\mathbb{R}^{d}}\,\frac{\left(\Delta^{l}_{y}f\right)(x)}{|y|^{d+\alpha}}\,dy,$$
where the normalizing constants $e_{d,l}(\alpha)$ is choosen so that $D^{\alpha}f$ does not depend on $l$ only if $l>\alpha$.
\end{definition}
\begin{definition}\cite[p.500]{Samko}\label{TruncatedHyperSingularintegral.2}
For $\delta>0$, the truncated hypersingular integral is defined by
$$D^{\alpha}_{\delta}f =\frac{1}{e_{d,l}(\alpha)}\int_{|y|\geq\delta}\,\frac{\left(\Delta^{l}_{y}f\right)(x)}{|y|^{d+\alpha}}\,dy,$$	
\end{definition}	
The next result is about the fact that the Riesz dervative is the left inverse of the Riesz potential.
\begin{theorem}\cite[p.517]{Samko}\label{RieszPotential.Inverse}
The operator $D^{\alpha}f=\displaystyle\lim_{\delta\to 0}\,D^{\alpha}_{\delta}f\,\,[L^{p}-\mbox{limit}]$ is the left inverse to the Riesz potential with in the frames of the spaces $L^{p}\left(\mathbb{R}^{d}\right)$:
$$ D^{\alpha}I^{\alpha}\phi\equiv\phi,\,\,\phi\in L^{p}\left(\mathbb{R}^{d}\right),\,\,1\leq p<\frac{d}{\alpha}.$$	
\end{theorem}
\textbf{Step 4:} 
\begin{definition}\cite[p.518]{Samko}$(\mbox{\textbf{Riesz Differentiation with Homogenious Characteristic}})$\label{def.RieszDerivative.3}
Let the characteristic function $P$ be homogenious in $y$, {\it i.e.,} for $y\in\mathbb{R}^{d}$, denote $y^{\prime}:=\frac{y}{\|y\|}$, then $P: S_{d-1}\to\mathbb{R}$ defined by $P=P(y^{\prime})$. For $l>\alpha$, the Riesz differentiation with with homogenious characteristic $P$ is defined by 
$$D^{\alpha}_{P}f =\frac{1}{e_{d,l}(\alpha)}\int_{\mathbb{R}^{d}}\,\frac{\left(\Delta^{l}_{y}f\right)(x)}{|y|^{d+\alpha}}\,P(y)\,dy,$$ 
\end{definition}
\begin{definition}\cite[p.524]{Samko}\label{Truncated.RieszDerivative.Characteristic}
For $\delta>0$, the truncated hypersingular integral with homogenious characteristic $P$ is defined by
$$D^{\alpha}_{P,\delta}f =\frac{1}{e_{d,l}(\alpha)}\int_{|y|\geq\delta}\,\frac{\left(\Delta^{l}_{y}f\right)(x)}{|y|^{d+\alpha}}\,P(y)\,dy,$$ 
\end{definition}		
\textbf{Step 5:} 
\begin{enumerate}
	\item[(a)]  For arbitary bounded characteristic $P$, the convergence of $\left(D^{\alpha}_{\delta}f\right)$ to $D^{\alpha}f$ in $L^{p}\left(\mathbb{R}^{d}\right)$ as $\delta\to 0$ implies the convergence of $ \left(D^{\alpha}_{P,\delta}f\right)$ to $D^{\alpha}_{P}f$ in $L^{p}\left(\mathbb{R}^{d}\right)$ as $\delta\to 0$. That is, for $1<r<\infty$, let $f\in L^{r}\left(\mathbb{R}^{d}\right)$ such that $D^{\alpha}f$ exists in $ L^{p}\left(\mathbb{R}^{d}\right)$, then $D^{\alpha}_{P}f$ exists in $ L^{p}\left(\mathbb{R}^{d}\right)$. This result is mentioned in \cite[p.523]{Samko}.
	\item[(b)] Let $P(y)$ be a characteristics such that $P$ satisfies the ellipticity condition 
	$$\int_{S_{d-1}}\,\left(-ix.\sigma\right)^{\alpha}\,P(\sigma)\,d\sigma\,\neq 0,\,\,\,|x|=1,$$
\end{enumerate}
 then the convergence of $ \left(D^{\alpha}_{P,\delta}f\right)$ to $D^{\alpha}_{P}f$ in $L^{p}\left(\mathbb{R}^{d}\right)$ as $\delta\to 0$ implies the convergence of $\left(D^{\alpha}_{\delta}f\right)$ to $D^{\alpha}f$ in $L^{p}\left(\mathbb{R}^{d}\right)$ as $\delta\to 0$ for $f\in L^{p}\left(\mathbb{R}^{d}\right)$ \cite[p.524-p.525]{Samko}.\\
 \vspace{0.1cm}\\
  The details of the resuls Step 5 (a) and Step 5 (b) can be found in Nogin and Samko's article as mentioned in \cite[p.523-p.525]{Samko}.\\
 \vspace{0.1cm}\\
 \textbf{Step 6:} For $1<p<\infty$, $f\in L^{p}\left(\mathbb{R}^{d}\right)$, we have $h_{\alpha}\ast f\in L^{p}\left(\mathbb{R}^{d}\right)$ and since $P$ is polynomial of degree less than $\alpha$, therefore $P$ is bounded on $S_{d-1}$. Applying Step 5 (a), we conclude that $D^{\alpha}_{P}f\in L^{p}\left(\mathbb{R}^{d}\right)$. Applying Theorem \ref{RieszPotential.Result1} with $q=r$, we get $\frac{1}{r}\leq\frac{1}{p}$. Therefore we get $p\leq r$. Now if $f\in L^{r}\left(\mathbb{R}^{d}\right)\cap L^{p}\left(\mathbb{R}^{d}\right)$, then $\Psi^{\alpha}f\in L^{p}\left(\mathbb{R}^{d}\right)$. \\
 \vspace{0.1cm}\\
 \textbf{Alternative Way:} If $P$ satisfies the ellipticity condition defined in Step 5 (b), by applying Step 5 (b), for $1<p<\infty$, $1<r<\infty$ with $\frac{1}{p}-\frac{\alpha}{d}\leq \frac{1}{r}\leq \frac{1}{p}$, the convergence of $ \left(D^{\alpha}_{P,\delta}f\right)$ to $D^{\alpha}_{P}f$ in $L^{p}\left(\mathbb{R}^{d}\right)$ as $\delta\to 0$ implies the convergence of $\left(D^{\alpha}_{\delta}f\right)$ to $D^{\alpha}f$ in $L^{p}\left(\mathbb{R}^{d}\right)$ as $\delta\to 0$ for $f\in L^{p}\left(\mathbb{R}^{d}\right)$ \,(See \cite[p.523-p.525]{Samko}). Hence we have $\Psi^{\alpha}f\in L^{p}\left(\mathbb{R}^{d}\right)$ for $f\in L^{r}\left(\mathbb{R}^{d}\right)$. Taking $P=1$, we conclude the desired convergence\,\,$\blacksquare$
\end{proof}	\\
\begin{proposition}\label{Fractional.Correction.inverseBessel.proposition}
Let $\alpha>0$ and $\Psi^{\alpha}$ be the inverse of Bessel potential. Then, for $1\leq p\leq 2$,  $\Psi^{\alpha}\widetilde{g^{\varepsilon}}\in L^{p}\left(\mathbb{R}^{d}\times\mathbb{R}\times (0,\infty)\right)$.	
\end{proposition}	
\begin{proof}
For $1\leq r\leq p\leq 2$, observe that $\widetilde{g^{\varepsilon}}\in L^{r}\left(\mathbb{R}^{d}\times\mathbb{R}\times (0,\infty)\right)$  as\\ $\mbox{measure}\left(\Omega\times\left(-\|u_0\|_{L^{\infty}(\Omega)},\|u_0\|_{L^{\infty}(\Omega)}\right)\times (0,T)\right)<\infty$. Therefore, for $1\leq p\leq 2$, an application of Lemma \ref{Meaningoffractional.operator} gives that $\Psi^{\alpha}\widetilde{g^{\varepsilon}}\in L^{p}\left(\mathbb{R}^{d}\times\mathbb{R}\times(0,\infty)\right)$ $\blacksquare$
\end{proof}\\
\begin{proposition}\label{KineticFormulationProposition12A}
	Let $1<p\leq 2$ and  $\left(\widetilde{\widetilde{m^{\varepsilon}}}\right)$ be a bounded sequence of positive measures in $\mathcal{M}\left(\mathbb{R}^{d}\times\mathbb{R}\times (0,\infty)\right)$. Then 
	\begin{enumerate}
	\item[(a)] for $1\leq p\leq \frac{d+2}{d+1}$\,\,\,and\,\,\,$(d+2)\,\frac{(p-1)}{p}<s<1$, $\left(\widetilde{\widetilde{m^{\varepsilon}}}\right)$ is bounded in $W^{-s,p}\left(\mathbb{R}^{d}\times\mathbb{R}\times(0,\infty)\right)$. Moreover $s\to 0_{+}$ as $p\to 1_{+}$.	
	\item[(b)] there exists a sequence $\left(g^{\varepsilon}\right)$ in compact subset of  $L^{p}\left(\mathbb{R}_{x}^{d}\times\mathbb{R}_{c}\times (0,\infty)\right)$ such that 
	\begin{equation}\label{KineticFormulation.Equation612A1B1}
	\frac{\partial \chi^{\varepsilon}}{\partial t} +\displaystyle\sum_{j=1}^{d}f_{j}^{\prime}(c)\,\frac{\partial\chi^{\varepsilon}}{\partial x_{j}} =\left(-\Delta_{x,t} +1\right)^{\frac{1}{2}}\left(-\Delta_{c} +1\right)^{\frac{r}{2}}g^{\varepsilon} \,\,\mbox{in}\,\,\mathcal{D}^{\prime}\left(\Omega\times\mathbb{R}\times (0,T)\right).
	\end{equation}
	\end{enumerate}
\end{proposition}
\begin{proof}
	We prove Proposition \ref{KineticFormulationProposition12A} in two steps. In Step 1, we prove part (a) and in Step 2, we prove part (b) of Proposition \ref{KineticFormulationProposition12A}.\\
	\vspace{0.1cm}\\
	\textbf{Step 1:} For $0<s<1$ and $s\,p^{\prime}> d+2$, applying Theorem 8.2 of \cite[p.562]{Eleonora}, we obtain that $W^{s,p^{\prime}}\left(\mathbb{R}^{d}\times\mathbb{R}\times (0,\infty)\right)$ is continuously imbedded in $C^{\alpha}\left(\mathbb{R}^{d}\times\mathbb{R}\times(0,\infty)\right)$ with $\alpha :=\left(s\,p^{\prime}-(d+2)\right)/p^{\prime}$ with standard H\"{o}lder norm in $C^{\alpha}\left(\mathbb{R}^{d}\times\mathbb{R}\times(0,\infty)\right)$. Hence $W^{s,p^{\prime}}\left(\mathbb{R}^{d}\times\mathbb{R}\times(0,\infty)\right)$ is continuously imbedded in $C\left(\mathbb{R}^{d}\times\mathbb{R}\times(0,\infty)\right)$ as   \\ $C^{\alpha}\left(\mathbb{R}^{d}\times\mathbb{R}\times(0,\infty)\right)$ is continuously imbedded $C\left(\mathbb{R}^{d}\times\mathbb{R}\times(0,\infty)\right)$ with supremum norm in $C\left(\mathbb{R}^{d}\times\mathbb{R}\times(0,\infty)\right)$. Applying duality arguement, we get that $\mathcal{M}\left(\mathbb{R}^{d}\times\mathbb{R}\times(0,\infty)\right)$ is continuously imbedded in $W^{-s,p}\left(\mathbb{R}^{d}\times\mathbb{R}\times(0,\infty)\right)$ for $s\,p^{\prime}> d+2$ and $p^{\prime},\,p$ satisfies $\frac{1}{p^{\prime}} + \frac{1}{p}=1$. Hence $p^{\prime}=\frac{p}{p-1}$. Therefore $p$ satisfies 
	\begin{eqnarray}
	\frac{s\,p}{(p-1)} &>& (d+2),\nonumber\\
	(d+2)\,p &<& (d+2) + s\,p,\nonumber\\
	&<& (d+2) + p\,\,\,\,\mbox{as}\,\,\,\,0<s<1.\nonumber 
	\end{eqnarray}
	The above computation gives $p<\frac{d+2}{d+1}$. Therefore for $1\leq p\leq \frac{d+2}{d+1}$, for $j\in\left\{0,1,2,\cdots,d\right\}$, $\left(\widetilde{\mu_{j}^{\varepsilon}}\right)$ is bounded in $W^{-s,p}\left(\mathbb{R}^{d}\times\mathbb{R}\times(0,\infty)\right)$. Hence for $1\leq p\leq \frac{d+2}{d+1}$, $\left(\widetilde{\widetilde{m^{\varepsilon}}}\right)$ is bounded in $W^{-s,p}\left(\mathbb{R}^{d}\times\mathbb{R}\times(0,\infty)\right)$. Since $p^{\prime}$ satisfies $\frac{d+2}{p^{\prime}}<s<1$, therefore $p$ satisfies $s>(d+2)\,\frac{p-1}{p}$. This implies $\frac{1}{p-1}>\frac{(d+2)}{p\,s}$. Hence we conclude that $s\to 0_{+}$ as $p\to 1_{+}$. \\
	\vspace{0.1cm}\\
	\textbf{Step 2:} Applying Lemma \ref{Lion.Parthame.Tadmor.p.178}, for all $\phi\in\mathcal{D}\left(\mathbb{R}^{d}\times\mathbb{R}\times (0,\infty)\right)$, we have  $$\left<\frac{\partial\widetilde{\widetilde{m^{\varepsilon}}}}{\partial c},\phi\right>= \left<\left(-\Delta_{x,t} + I\right)^{\frac{1}{2}}\left(-\Delta_{c}+I\right)^{\frac{r}{2}}g^{\varepsilon},\phi\right>.$$
	In particular, for  $\phi\in\mathcal{D}\left(\Omega\times\mathbb{R}\times (0,T)\right)$, we have $\left<\frac{\partial\widetilde{\widetilde{m^{\varepsilon}}}}{\partial c},\phi\right>= \left<\left(-\Delta_{x,t} + I\right)^{\frac{1}{2}}\left(-\Delta_{c}+I\right)^{\frac{r}{2}}g^{\varepsilon},\phi\right>$.
	Therefore, we conclude \eqref{KineticFormulation.Equation612A1B1} with $\widetilde{\widetilde{m^{\varepsilon}}}\Big|_{\Omega\times\mathbb{R}\times (0,T)}=m^{\varepsilon}$ $\blacksquare$
\end{proof}\\
\vspace{0.1cm}\\
We now give a standard result from measure theory.
\begin{theorem}\label{Gdelta.Fsigma}
Let $A\subset\mathbb{R}^{d}$ be Lebesgue measurable set. Then there exist a $G_{\delta}-$ set $G $ and a $ F_{\sigma}-$ set $F$ such that $F\subset A\subset G$ and 
$$\mu\left(G\setminus A\right)=\mu\left(A\setminus F\right)=0.$$	
\end{theorem}	
We now use Theorem \ref{VelocityAveragingLemma.12} to conclude the next result.
\begin{theorem}\label{Kinetic.Compactness.Result}
Assume \textbf{Hypothesis D}. Then the sequence of quasilinear viscous approximations $\left(u^{\varepsilon}\right)$ which are solutions of IBVP \eqref{regularized.IBVP} lies in a compact subset of  $L^{p}_{loc}\left(\Omega_{T}\right)$.	
\end{theorem} 
\begin{proof}
Let $\mathcal{L}_{1}$, $\mathcal{L}_{2}$ denote the sigma algebras of Lebesgue measurable sets of\\ $\Omega\times\left(-\|u_0\|_{L^{\infty}(\Omega)},\|u_0\|_{L^{\infty}(\Omega)}\right)\times (0,T)$	and $\mathbb{R}^{d}\times\mathbb{R}\times\left(0,\infty\right)$ respectively. Let $\mathcal{B}_{1}$, $\mathcal{B}_{2}$ denote the sigma algebras of Borel measurable sets of $\Omega\times\left(-\|u_0\|_{L^{\infty}(\Omega)},\|u_0\|_{L^{\infty}(\Omega)}\right)\times (0,T)$	and $\mathbb{R}^{d}\times\mathbb{R}\times\left(0,\infty\right)$ respectively. 
Since open sets of $\Omega\times\mathbb{R}\times\left(0,T\right)$ are open sets of  $\mathbb{R}^{d}\times\mathbb{R}\times\left(0,\infty\right)$, therefore $\mathcal{B}_{1}\subset \mathcal{B}_{2}.$
We know that measure zero sets of $\Omega\times\left(-\|u_0\|_{L^{\infty}(\Omega)},\|u_0\|_{L^{\infty}(\Omega)}\right)\times (0,T)$ are also measure zero sets of $\mathbb{R}^{d}\times\mathbb{R}\times\left(0,\infty\right)$. Hence, the measure zero sets are Lebesgue mesureable sets of $\mathbb{R}^{d}\times\mathbb{R}\times\left(0,\infty\right)$. Applying Theorem \ref{Gdelta.Fsigma}, we conclude that $\mathcal{L}_{1}\subset\mathcal{L}_{2}$. Let $\widetilde{\widetilde{m^{\varepsilon}}}^{\ast}$ be the Lebesgue measure on $\mathcal{L}_{2}$ introduced by the regular Borel measure $\widetilde{\widetilde{m^{\varepsilon}}}$ (See \cite[p.22-p.23]{Athreya}) on $\mathcal{B}_{2}$ such that $\widetilde{\widetilde{m^{\varepsilon}}}^{\ast}\Big|_{\mathcal{B}_{2}}=\widetilde{\widetilde{m^{\varepsilon}}}$. Define $\nu^{\varepsilon}:\mathcal{L}_{2}\to [0,\infty]$ by 
$$\nu^{\varepsilon}(A):=\widetilde{\widetilde{m^{\varepsilon}}}^{\ast}\left(A\cap(\Omega\times\left(-\|u_0\|_{L^{\infty}(\Omega)},\|u_0\|_{L^{\infty}(\Omega)}\right)\times (0,T))\right).$$
We prove that $\left(\nu^{\varepsilon}\right)$ is a sequence of measures on $\mathcal{L}_{2}$ in two steps. In Step (a), we show that for every $\varepsilon >0$, $\nu^{\varepsilon}\left(\emptyset\right)=0$ and in Step (b), we show that for every $\varepsilon >0$, $\nu^{\varepsilon}$ is countabily additive on $\mathcal{L}_{2}$.\\
\textbf{Step (a):} For $\varepsilon>0$, since $\widetilde{\widetilde{m^{\varepsilon}}}$ are measures, we have $\nu^{\varepsilon}\left(\emptyset\right)=\widetilde{\widetilde{m^{\varepsilon}}}^{\ast}\left(\emptyset\right)=\widetilde{\widetilde{m^{\varepsilon}}}\left(\emptyset\right)=0$.\\
\textbf{Step (b):} For $n\in\mathbb{N}$, let $A_{n}\in\mathcal{L}_{2}$ such that $A_{i}\cap A_{j}=\emptyset$ for $i\neq j$. Observe that 
\begin{equation}\label{Correctionnuepsilonmeasure.eqn1}
\begin{split}
\displaystyle\cup_{j=1}^{\infty}A_{j}=\left(\displaystyle\cup_{j=1}^{\infty}\left(A_{j}\cap\left(\Omega\times\left(-\|u_0\|_{L^{\infty}(\Omega)},\|u_0\|_{L^{\infty}(\Omega)}\right)\times (0,T)\right)\right)\right)\nonumber\\ \cup\left(\displaystyle\cup_{j=1}^{\infty}\left(A_{j}\cap\left(\Omega\times\left(-\|u_0\|_{L^{\infty}(\Omega)},\|u_0\|_{L^{\infty}(\Omega)}\right)\times (0,T)\right)^{c}\right)\right).
\end{split}
\end{equation}
In view of \eqref{Correctionnuepsilonmeasure.eqn1}, we have
\begin{eqnarray}\label{Correctionnuepsilonmeasure.eqn2}
\nu^{\varepsilon}\left(\displaystyle\cup_{j=1}^{\infty}A_{j}\right)&=&\widetilde{\widetilde{m^{\varepsilon}}}^{\ast}\left(\left(\displaystyle\cup_{j=1}^{\infty}A_{j}\right)\cap\left(\Omega\times\left(-\|u_0\|_{L^{\infty}(\Omega)},\|u_0\|_{L^{\infty}(\Omega)}\right)\times (0,T)\right)\right)\nonumber\\ &=&\widetilde{\widetilde{m^{\varepsilon}}}^{\ast}\left(\displaystyle\cup_{j=1}^{\infty}\left(A_{j}\cap\left(\Omega\times\left(-\|u_0\|_{L^{\infty}(\Omega)},\|u_0\|_{L^{\infty}(\Omega)}\right)\times (0,T)\right)\right)\right)\nonumber\\
&=&\displaystyle\sum_{j=1}^{\infty}\widetilde{\widetilde{m^{\varepsilon}}}^{\ast}\left(\left(A_{j}\cap\left(\Omega\times\left(-\|u_0\|_{L^{\infty}(\Omega)},\|u_0\|_{L^{\infty}(\Omega)}\right)\times (0,T)\right)\right)\right)\nonumber\\ 
&=& \displaystyle\sum_{j=1}^{\infty}\nu^{\varepsilon}\left(A_{j}\right).
\end{eqnarray}
Therefore $\left(\nu^{\varepsilon}\right)$ defines a sequence of measures.\\
\vspace{0.2cm}\\
We now verify \eqref{Velocity.AveragingEquation1.A1} in the following three steps with $h=\widetilde{\chi^{\varepsilon}}$.\\
\textbf{Step 1:} Let $\phi\in\mathcal{D}\left(\mathbb{R}^{d}\times\mathbb{R}\times (0,\infty)\right)$ such that $$\mbox{supp}\left(\phi\right)\subset\Omega\times\left(-\|u_0\|_{L^{\infty}(\Omega)},\|u_0\|_{L^{\infty}(\Omega)}\right)\times (0,T).$$
In view of \eqref{KineticFormulation.Equation612A1}, we have
$$\frac{\partial \widetilde{\chi^{\varepsilon}}}{\partial t} +\displaystyle\sum_{j=1}^{d}f_{j}^{\prime}(c)\,\frac{\partial\widetilde{\chi^{\varepsilon}}}{\partial x_{j}} =\frac{\partial \chi^{\varepsilon}}{\partial t} +\displaystyle\sum_{j=1}^{d}f_{j}^{\prime}(c)\,\frac{\partial\chi^{\varepsilon}}{\partial x_{j}} =\frac{\partial m^{\varepsilon}}{\partial c}= \frac{\partial \widetilde{\widetilde{m^{\varepsilon}}}^{\ast}}{\partial c}=\frac{\partial \nu^{\varepsilon}}{\partial c}$$
in $\mathcal{D}^{\prime}\left(\Omega\times\left(-\|u_0\|_{L^{\infty}(\Omega)},\|u_0\|_{L^{\infty}(\Omega)}\right)\times (0,T)\right)$. Applying Proposition \ref{KineticFormulationProposition12A}, we have \eqref{Velocity.AveragingEquation1.A1}.\\
\vspace{0.2cm}\\
\textbf{Step 2:} Let $\phi\in\mathcal{D}\left(\mathbb{R}^{d}\times\mathbb{R}\times (0,\infty)\right)$ such that $$\mbox{supp}\left(\phi\right)\subset\mathbb{R}^{d}\times\mathbb{R}\times (0,\infty)\setminus\Omega\times\left(-\|u_0\|_{L^{\infty}(\Omega)},\|u_0\|_{L^{\infty}(\Omega)}\right)\times (0,T),$$ then from the equation, we get
\begin{eqnarray}\label{Measure.Averaging.LHS.RHS}
\int_{\mathbb{R}^{d}}\int_{-\|u_0\|_{L^{\infty}(\Omega)}}^{\|u_0\|_{L^{\infty}(\Omega)}}\int_{0}^{\infty}\left\{\widetilde{\chi^{\varepsilon}}(c)\frac{\partial\phi}{\partial t}\,+\,\displaystyle\sum_{j=0}^{d}\,f_{j}^{\prime}(c)\widetilde{\chi^{\varepsilon}}\,\frac{\partial\phi}{\partial x_{j}}\right\}\,dx\,dc\,dt &=&-\left<\frac{\partial\nu^{\varepsilon}}{\partial c},\,\phi\right>\nonumber\\
&=& \left<\nu^{\varepsilon},\,\frac{\partial\phi}{\partial c}\right>\nonumber\\
&=&\int_{\mathbb{R}^{d}\times\mathbb{R}\times(0,\infty)}\,\frac{\partial\phi}{\partial c}\,d\nu^{\varepsilon}\nonumber\\
{}
\end{eqnarray}
Since $\mbox{Supp}\left(\nu^{\varepsilon}\right)\subset\overline{\Omega\times\left(-\|u_0\|_{L^{\infty}(\Omega)},\|u_0\|_{L^{\infty}(\Omega)}\right)\times (0,T)}$, hence the LHS and RHS both are zero. \\
\vspace{0.1cm}\\
\textbf{Step 3:} Let $\phi\in\mathcal{D}\left(\mathbb{R}^{d}\times\mathbb{R}\times(0,\infty)\right)$ such that $$\mbox{measure}\left(\mbox{Supp}\left(\phi\right)\cap\left(\Omega\times\left(-\|u_0\|_{L^{\infty}(\Omega)},\|u_0\|_{L^{\infty}(\Omega)}\right)\times (0,T)\right)\right)>0.$$
 For $j\in\left\{1,2,\cdots,d\right\}$, we know that $\chi^{\varepsilon}=\mbox{sg}\left(u^{\varepsilon}-c\right)-\mbox{sg}\left(c\right)$ satisfies \eqref{KineticFormulation.Equation612A1} in $\mathcal{D}^{\prime}\left(\Omega\times\mathbb{R}\times(0,T)\right)$. Observe that
\begin{eqnarray}\label{Measure.newline.intersection}
\frac{\partial}{\partial t}\left(\mbox{sg}\left(u^{\varepsilon}-c\right)-\mbox{sg}\left(c\right)\right)&=&\frac{\partial}{\partial t}\left(\mbox{sg}\left(u^{\varepsilon}+c\right)+\mbox{sg}\left(c\right)\right)\nonumber\\ \frac{\partial}{\partial x_{j}}\left(\mbox{sg}\left(u^{\varepsilon}-c\right)-\mbox{sg}\left(u^{\varepsilon}-c\right)\right)&=&\frac{\partial}{\partial x_{j}}\left(\mbox{sg}\left(u^{\varepsilon}+c\right)+\mbox{sg}\left(c\right)\right)\nonumber
\end{eqnarray}
Since $\widetilde{u^{\varepsilon}}\equiv 0$ on $\partial\Omega\times (0,T)$, therefore $\widetilde{u^{\varepsilon}}$  lies in $H^{1}\left(\mathbb{R}^{d}\times (0,T)\right)$. Observe that the set $\mathbb{R}^{d}\times\mathbb{R}\times\left\{T\right\}$ is a measure zero set in $\mathbb{R}^{d}\times\mathbb{R}\times (0,\infty)$.  Let $\phi\in\mathcal{D}\left(\mathbb{R}^{d}\times\mathbb{R}\times(0,\infty)\right)$ such that $\mbox{supp}\left(\phi\right)\subset\mathcal{D}\left(\mathbb{R}^{d}\times\mathbb{R}\times (0,T)\right)$ or $\mbox{supp}\left(\phi\right)\subset\mathcal{D}\left(\mathbb{R}^{d}\times\mathbb{R}\times(T,\infty)\right)$, then the derivation of Kinteic equations in Subsection \ref{Kineticformulationviscosityproblem} is valid in the domains $\mathbb{R}^{d}\times\mathbb{R}\times(0,T)$ and $\mathbb{R}^{d}\times\mathbb{R}\times(T,\infty)$. Therefore, we obtain
\begin{equation}\label{KineticFormulation.Equation3ABC123}
\begin{split}
\left(\frac{\partial}{\partial t}\left(\eta(\widetilde{u^{\varepsilon}};c)-\eta(0;c)\right)\right) &+ \displaystyle\sum_{j=1}^{d}\,\left(\frac{\partial}{\partial x_{j}}\left(q_{j}(\widetilde{u^{\varepsilon}};c)-q_{j}(0;c)\right)\right)\\ 
&=\left(\varepsilon\,\eta^{\prime}\left(\widetilde{u^{\varepsilon}};c\right)\displaystyle\sum_{j=1}^{d}\frac{\partial}{\partial x_{j}}\left(B(\widetilde{u^{\varepsilon}})\,\frac{\partial \widetilde{u^{\varepsilon}}}{\partial x_{j}}\right)\right)\,\,\mbox{on}\,\,\mathbb{R}^{d}\times\mathbb{R}\times\left((0,\infty)\setminus\left\{T\right\}\right). 
\end{split}
\end{equation}
Since $\left(\eta(\widetilde{u^{\varepsilon}};c)-\eta(0;c)\right),\,\left(q_{j}(\widetilde{u^{\varepsilon}};c)-q_{j}(0;c)\right)\in L^{1}_{loc}\left(\mathbb{R}^{d}\times\mathbb{R}\times (0,\infty)\right)$ and  $\mathbb{R}^{d}\times\mathbb{R}\times\left\{T\right\}$ is a set of measure zero, we have the following in distributions.
\begin{equation}\label{KineticFormulation.Equation3ABC12}
\begin{split}
\frac{\partial}{\partial c}\left(\frac{\partial}{\partial t}\left(\eta(\widetilde{u^{\varepsilon}};c)-\eta(0;c)\right)\right) &+ \displaystyle\sum_{j=1}^{d}\,\frac{\partial}{\partial c}\left(\frac{\partial}{\partial x_{j}}\left(q_{j}(\widetilde{u^{\varepsilon}};c)-q_{j}(0;c)\right)\right)\\ 
&=\frac{\partial}{\partial c}\left(\varepsilon\,\eta^{\prime}\left(\widetilde{u^{\varepsilon}};c\right)\displaystyle\sum_{j=1}^{d}\frac{\partial}{\partial x_{j}}\left(B(\widetilde{u^{\varepsilon}})\,\frac{\partial \widetilde{u^{\varepsilon}}}{\partial x_{j}}\right)\right)\,\,\mbox{in}\,\,\mathcal{D}^{\prime}\left(\mathbb{R}^{d}\times\mathbb{R}\times(0,\infty)\right). 
\end{split}
\end{equation}
\begin{equation}\label{KineticFormulation.Equation3ABC125}
\begin{split}
\frac{\partial}{\partial t}\left(\frac{\partial}{\partial c}\left(\eta(\widetilde{u^{\varepsilon}};c)-\eta(0;c)\right)\right) &+ \displaystyle\sum_{j=1}^{d}\,\frac{\partial}{\partial x_{j}}\left(\frac{\partial}{\partial c}\left(q_{j}(\widetilde{u^{\varepsilon}};c)-q_{j}(0;c)\right)\right)\\ 
&=\frac{\partial}{\partial c}\left(\varepsilon\,\eta^{\prime}\left(\widetilde{u^{\varepsilon}};c\right)\displaystyle\sum_{j=1}^{d}\frac{\partial}{\partial x_{j}}\left(B(\widetilde{u^{\varepsilon}})\,\frac{\partial \widetilde{u^{\varepsilon}}}{\partial x_{j}}\right)\right)\,\,\mbox{in}\,\,\mathcal{D}^{\prime}\left(\mathbb{R}^{d}\times\mathbb{R}\times(0,\infty)\right). 
\end{split}
\end{equation}
From \eqref{KineticFormulation.Equation3ABC125}, we arrive at
\begin{equation}\label{KineticFormulation.Equation3ABC1256}
\begin{split}
\frac{\partial}{\partial t}\left(\mbox{sg}\left(\widetilde{u^{\varepsilon}}-c\right)-\mbox{sg}(c)\right) &+ \displaystyle\sum_{j=1}^{d}\,f^{\prime}_{j}(c)\frac{\partial}{\partial x_{j}}\left(\mbox{sg}\left(\widetilde{u^{\varepsilon}}-c\right)-\mbox{sg}(c)\right)\\ 
&=\frac{\partial}{\partial c}\left(\varepsilon\,\eta^{\prime}\left(\widetilde{u^{\varepsilon}};c\right)\displaystyle\sum_{j=1}^{d}\frac{\partial}{\partial x_{j}}\left(B(\widetilde{u^{\varepsilon}})\,\frac{\partial \widetilde{u^{\varepsilon}}}{\partial x_{j}}\right)\right)\,\,\mbox{in}\,\,\mathcal{D}^{\prime}\left(\mathbb{R}^{d}\times\mathbb{R}\times(0,\infty)\right). 
\end{split}
\end{equation}
Observe that
\begin{equation}\label{KineticFormulation.Equation3ABC12567}
\begin{split}
\frac{\partial}{\partial t}\left(\mbox{sg}\left(\widetilde{u^{\varepsilon}}-c\right)+\mbox{sg}(c)\right) &+ \displaystyle\sum_{j=1}^{d}\,f^{\prime}_{j}(c)\frac{\partial}{\partial x_{j}}\left(\mbox{sg}\left(\widetilde{u^{\varepsilon}}-c\right)+\mbox{sg}(c)\right)\\ 
&=\frac{\partial}{\partial c}\left(\varepsilon\,\eta^{\prime}\left(\widetilde{u^{\varepsilon}};c\right)\displaystyle\sum_{j=1}^{d}\frac{\partial}{\partial x_{j}}\left(B(\widetilde{u^{\varepsilon}})\,\frac{\partial \widetilde{u^{\varepsilon}}}{\partial x_{j}}\right)\right)\,\,\mbox{in}\,\,\mathcal{D}^{\prime}\left(\mathbb{R}^{d}\times\mathbb{R}\times(0,\infty)\right). 
\end{split}
\end{equation}
Let $\phi\in\mathcal{D}\left(\mathbb{R}^{d}\times\mathbb{R}\times(0,\infty)\right)$ such that $\mbox{supp}\left(\phi\right)\cap\mathbb{R}^{d}\times\mathbb{R}\times \left\{T\right\}\neq\O $. Then, we consider the following 
\begin{equation}\label{Measure.intersection.equation.solution2}
\begin{split}
\int_{\mathbb{R}^{d}\times\mathbb{R}\times(0,\infty)}\,\left\{\frac{\partial\widetilde{\chi^{\varepsilon}}(c)}{\partial t} +\displaystyle\sum_{j=1}^{d}f_{j}^{\prime}(c)\frac{\partial\widetilde{\chi^{\varepsilon}}}{\partial x_{j}}\right\}\,\phi\,dx\,dc\,dt\hspace{9cm}\\ = \int_{\mathbb{R}^{d}\times\mathbb{R}\times(0,\infty)}\,\Big\{\frac{\partial}{\partial t}\left(\mbox{sg}\left(\widetilde{u^{\varepsilon}}-c\right)+\mbox{sg}\left(c\right)\right)+\displaystyle\sum_{j=1}^{d}f_{j}^{\prime}(c)\frac{\partial}{\partial x_{j}}\left(\mbox{sg}\left(\widetilde{u^{\varepsilon}}-c\right)+\mbox{sg}\left(c\right)\right)\Big\}\phi\,dx\,dc\,dt\hspace{2cm}\\
=-\displaystyle\lim_{N\to\infty}\int_{-N}^{N}\int_{\Omega\times(0,T)}\,\Big\{\left(\mbox{sg}\left(u^{\varepsilon}-c\right)+\mbox{sg}\left(c\right)\right)\frac{\partial\phi}{\partial t} +\displaystyle\sum_{j=1}^{d}f_{j}^{\prime}(c)\left(\mbox{sg}\left(u^{\varepsilon}-c\right)+\mbox{sg}\left(c\right)\right)\frac{\partial\phi}{\partial x_{j}}\Big\}\,dx\,dt\,dc\hspace{1cm}\\ +\displaystyle\lim_{N\to\infty}\int_{-N}^{N}\Big(\int_{\Omega}\left(\mbox{sg}\left(\widetilde{u^{\varepsilon}}(x,T)-c\right)+\mbox{sg}\left(c\right)\right)\phi(x,c,T)-\left(\mbox{sg}\left(u^{\varepsilon}(x,0)-c\right)+\mbox{sg}\left(c\right)\right)\phi(x,c,0)\Big)dt\,dx\,dc\\
+\displaystyle\sum_{j=1}^{d}\displaystyle\lim_{N\to\infty}\int_{-N}^{N}\int_{0}^{T}\int_{\partial\Omega}\left(\mbox{sg}\left(u^{\varepsilon}(x,t)-c\right)+\mbox{sg}\left(c\right)\right)\phi(x,c,t)\sigma_{j}\,d\sigma\,dt\,dc.\hspace{4cm}
\end{split}
\end{equation}
Since $u^{\varepsilon}=0\,\,\mbox{on}\,\partial\Omega\times (0,T)$, therefore $\left(\mbox{sg}\left(u^{\varepsilon}-c\right)+\mbox{sg}\left(c\right)\right)\,\phi(x,c,t)=0$ on the parabolic boundary of $\Omega_{T}$. Observe that the LHS of \eqref{Measure.intersection.equation.solution2} is zero in\\ $\mathcal{D}^{\prime}\left(\mathbb{R}^{d}\times\mathbb{R}\times (0,\infty)\setminus \left(\Omega\times\left(-\|u_0\|_{L^{\infty}(\Omega)},\|u_0\|_{L^{\infty}(\Omega)}\right)\times (0,T)\right)\right)$, therefore, we consider the RHS of \eqref{KineticFormulation.Equation3ABC1256} only on $\Omega\times\left(-\|u_0\|_{L^{\infty}(\Omega)},\|u_0\|_{L^{\infty}(\Omega)}\right)\times (0,T)$. Hence, we obtain the following
$$\varepsilon\,\eta^{\prime}\left(\widetilde{u^{\varepsilon}};c\right)\displaystyle\sum_{j=1}^{d}\frac{\partial}{\partial x_{j}}\left(B(\widetilde{u^{\varepsilon}})\,\frac{\partial \widetilde{u^{\varepsilon}}}{\partial x_{j}}\right)=\varepsilon\,\eta^{\prime}\left(u^{\varepsilon};c\right)\displaystyle\sum_{j=1}^{d}\frac{\partial}{\partial x_{j}}\left(B(u^{\varepsilon})\,\frac{\partial u^{\varepsilon}}{\partial x_{j}}\right)=m^{\varepsilon}$$
in $\mathcal{D}^{\prime}\left(\Omega\times\left(-\|u_0\|_{L^{\infty}(\Omega)},\|u_0\|_{L^{\infty}(\Omega)}\right)\times (0,T)\right)$.

Using \eqref{Measure.intersection.equation.solution2} in  \eqref{KineticFormulation.Equation3ABC1256}, we arrive at
\begin{eqnarray}\label{Measure.calculation.equation.51a}
\int_{\mathbb{R}^{d}\times\mathbb{R}\times(0,\infty)}\,\left\{\frac{\partial\widetilde{\chi^{\varepsilon}}(c)}{\partial t} +\displaystyle\sum_{j=1}^{d}f_{j}^{\prime}(c)\frac{\partial\widetilde{\chi^{\varepsilon}}}{\partial x_{j}}\right\}\,\phi\,dx\,dc\,dt=\int_{\mathbb{R}^{d}\times\mathbb{R}\times(0,\infty)}\,\frac{\partial\phi}{\partial c}\,d\nu^{\varepsilon}.
\end{eqnarray}
Applying Proposition \ref{KineticFormulationProposition12A} in \eqref{Measure.calculation.equation.51a}, we conclude \eqref{Velocity.AveragingEquation1.A1}. An application of Theorem \ref{VelocityAveragingLemma.12} with $h=\chi_{\widetilde{u^{\varepsilon}}}(c)$, $g=\widetilde{g^{\varepsilon}}$ and $\psi(c)=\chi_{[-\|u_{0}\|_{\infty},\|u_{0}\|_{\infty}]}(c)$, we conclude the proof of Theorem \ref{Kinetic.Compactness.Result} $\blacksquare$  	
\end{proof}	
\section{Unique entorpy solution to scalar conservation law}\label{Section.F.Otto.entropysolution}
In this section, we want to prove that the {\it a.e.} limit of quasilinear viscous approximations is the unique entropy solution
for scalar conservation laws in the sense of Otto \cite{MR1387428}. We now give some basic definitions before introducing the concepts of entropy solution. 
\begin{definition}[\textbf{Entropy}]\label{Otto.entropy}
 {\rm Let $\eta$ be a convex function. Then $\eta$ is said to be an entropy if for $j\in\left\{1,2,\cdots,d\right\}$, there exist function 
 $q_{j}$ such that for all $z\in\mathbb{R}$, the following equality 
 $$\eta^{\prime}(z)\,f_{j}^{\prime}(z)= q_{j}^{\prime}(z)$$
 hold.}
\end{definition}
Denote $Q:=\left(Q_{1},Q_{2},\cdots,Q_{d}\right).$
\begin{definition}[\textbf{Boundary-Entropy}]\cite[p.103]{Necas}\label{Otto.Bentropy}
{\rm Let $H, Q\in C^{2}(\mathbb{R}^{2})$. A pair $\left(H, Q\right)$ is called a boundary entropy-entropy flux pair if for 
$w\in\mathbb{R}$, $\left(H(.,w), Q(.,w)\right)$ is an entropy-entropy flux pair in the sense of Definition \ref{Otto.entropy}
and $H, Q$ satisfy
$$ H(w,w)=0,\,\, Q(w,w)=0,\,\,\partial_{1}\,H(w,w)=0,$$
where $\partial_{1}H$ denotes the partial derivative with respect to the first variable.}
\end{definition}
\begin{remark}\cite[p.104]{Necas}\label{Boundary.Entropy.F.Otto.rmk1}
Let $k\in\mathbb{R}$ be arbitary but fixed and let $l\in\mathbb{N}$. Define the entropy-entropy flux pair $\left(\eta_{l},q_{l}\right)$ by
\begin{eqnarray}\nonumber\\
\eta_{l}(z)&:=&\left(\left(z-k\right)^{2}+\left(\frac{1}{l}\right)^2\right)^{\frac{1}{2}}-\frac{1}{l},\nonumber\\
q_{l}(z)&:=& \int_{k}^{z}\,\eta_{l}^{\prime}(r)\,f^{\prime}(r)\,dr.\nonumber
\end{eqnarray}
Then $\left(\eta_{l},q_{l}\right)$ converges uniformly to a non-smooth entropy-entropy flux pair $\left(|z-k|,\,F(z,k)\right)$ as $l\to\infty$, where
$$F(z,k)=\mbox{sg}\left(z-k\right)\,\,\left(f(z)-f(k)\right).$$
Denote the closed interval between real numbers $a,b$ by 
$$\mathcal{I}[a,b]:=\left[\mbox{min}(a,b),\mbox{max}(a,b)\right].$$
Define a boundary entropy-entropy flux pair $\left(H_{l},Q_{l}\right)$ by
\begin{eqnarray}\nonumber
H_{l}(z,w)&:=&\left(\left(\mbox{dist}\left(z,\mathcal{I}(w,k)\right)\right)^{2}+\left(\frac{1}{l}\right)^2\right)^{\frac{1}{2}}-\frac{1}{l},\nonumber\\
Q_{l}(z,w)&:=& \int_{w}^{z}\,\partial_{1}H_{l}(r,w)\,f^{\prime}(r)\,dr.\nonumber
\end{eqnarray}
The sequence $\left(H_{l},Q_{l}\right)$ converges uniformly to $\left(\mbox{dist}(z,\mathcal{I}[w,k]),\mathcal{F}(z,w,k)\right)$ as $l\to\infty$,
where $\mathcal{F}\in \left(C(\mathbb{R}^{3})\right)^{d}$ is given by 
\[
\mathcal{F}(z,w,k):=
\begin{cases}
f(w)-f(k) & \text{if $z\leq w\leq k$}\\
0         & \text{if $w\leq z\leq k$}\\
f(z)-f(k) & \text{if $w\leq k\leq z$}\\
\vspace{0.1cm}\\
f(k)-f(z) & \text{if $z\leq k\leq w$}\\
0         & \text{if $k\leq z\leq w$}\\
f(z)-f(w) & \text{if $k\leq w\leq z$}.\\
\end{cases}
\]
\end{remark}
\begin{definition}[\textbf{F. Otto}]\cite[p.103]{Necas}\label{F.Otto}
 Let $u_{0}\in L^{\infty}(\Omega)$ and $f\in\left(C^{1}(\mathbb{R})\right)^{d}$. We say that $u$ is an entropy solution for IBVP
 \eqref{ivp.cl} if and only if $u\in L^{\infty}(\Omega_{T})$ and satisfies:
 \begin{enumerate}
  \item  the conservation laws and the entropy condition in the sense
  \begin{eqnarray}\label{definition.otto.eqn1}
   \int_{\Omega_{T}}\left(\eta(u)\,\frac{\partial\phi}{\partial t}\,+\,\displaystyle\sum_{j=1}^{d}q_{i}(u)\,\frac{\partial\phi}
  {\partial x_{j}}\right)\,dx\,dt\geq 0
  \end{eqnarray}
for all $\phi\in\mathcal{D}(\Omega_{T}),\,\,\,\phi\geq 0$ and all entropy-entropy flux pairs $\left(\eta, q\right)$;
  \item the boundary condition in the sense 
  \begin{eqnarray}\label{definition.otto.eqn2}
   \mbox{}ess\displaystyle\lim_{h\to 0}\int_{0}^{T}\int_{\partial\Omega}\,Q\left(u(r+h\nu(r)),0\right)\cdot\nu(r)\,\beta(r)
  \,\,\,dr\geq 0,
  \end{eqnarray}
  for all $\beta\in L^{1}(\partial\Omega\times(0,T)),\,\,\,\beta\geq 0$ {\it a.e.} and all boundary entropy fluxes $Q$;
  \item the initial condition $u_{0}\in L^{\infty}(\Omega)$ in the sense
  \begin{eqnarray}\label{definition.otto.eqn3}
   \mbox{ess}\displaystyle\lim_{t\to 0}\int_{\Omega}\left|u(x,t)-u_{0}(x)\right|\,\,dx=0.
  \end{eqnarray}
 \end{enumerate}
\end{definition}
\vspace{0.2cm}
\begin{proof}[\textbf{Proof of Theorem\ref{paper2.compensatedcompactness.theorem1}}] We prove Theorem\ref{paper2.compensatedcompactness.theorem1} in six steps and for any dimension $d\in\mathbb{N}$. In Step 1, we show that the {\it a.e.} limit of a  subsequence of solutions to generalized viscosity problem \eqref{regularized.IBVP} is a weak solution of scalar conservation law \eqref{ivp.cl}. In Step 2, a weak solution satifies the inequality \eqref{definition.otto.eqn1}. In Step 3, we develope theory similar to \cite{Necas} to prove \eqref{definition.otto.eqn2}. In Step 4, we show that $u$ satisfies \eqref{definition.otto.eqn2}. In Step 5, we show that $u$ satisfies initial condition \eqref{definition.otto.eqn3}. In Step 6, we conclude uniqueness of weak solutions of IBVP for conservation laws \eqref{ivp.cl}.\\
\textbf{Step 1:} Let $\phi \in \mathcal{D}(\Omega\times [0, T))$. Multiplying the equation \eqref{regularized.IBVP.a} by $\phi$, integrating over $\Omega_{T}$ and using integrating by parts formula, we get
	\begin{equation}\label{eqnchap502}
	\begin{split}
	\int_{0}^{T}\,\int_{\Omega}u^{\varepsilon}\,\frac{\partial \phi}{\partial t}\,dx\,dt - \varepsilon\,\displaystyle\sum_{j=1}^{d}\int_{0}^{T}\,\int_{\Omega}\,B(u^{\varepsilon})\,\frac{\partial u^{\varepsilon}}{\partial x_{j}}\,\frac{\partial \phi}{\partial x_{j}}\,dx\,dt +\displaystyle\sum_{j=1}^{d}\int_{0}^{T}\,\int_{\Omega}\,f_{j}(u^{\varepsilon})\,\frac{\partial \phi}{\partial x_{j}}\,dx\,dt\\ =\int_{\Omega}u^{\varepsilon}(x,0)\,\phi(x,0)\,dx\,dt.
	\end{split}
	\end{equation}
	We pass to the limit in \eqref{eqnchap502} in three steps. Firstly, we show that 
	\begin{equation}\label{weaksolution.eqn1}
	\displaystyle\lim_{\varepsilon\to 0}\int_{0}^{T}\,\int_{\Omega}u^{\varepsilon}\,\frac{\partial \phi}{\partial t}\,dx\,dt =\int_{0}^{T}\,\int_{\Omega}u\,\frac{\partial \phi}{\partial t}\,dx\,dt.
	\end{equation}
	In view of maximum principle \eqref{regularized.eqnchap303}, the integrand on LHS in \eqref{weaksolution.eqn1} is {\it a.e.} pointwise bounded  by $\|u_{0}\|_{L^{\infty}(\Omega)}\Big|\frac{\partial \phi}{\partial t}\Big|$. Since $u^{\varepsilon}\to u$ as $\varepsilon\to 0$,  we get \eqref{weaksolution.eqn1} by applying bounded convergence theorem.\\
	Next, for $j=1,2,\cdots,d$,  proofs of 
	\begin{equation}\label{weaksolution.eqn2}
	\displaystyle\lim_{\varepsilon\to 0}\int_{0}^{T}\,\int_{\Omega}\,f_{j}(u^{\varepsilon})\,\frac{\partial \phi}{\partial x_{j}}\,dx\,dt = \int_{0}^{T}\,\int_{\Omega}\, f_{j}(u)\,\frac{\partial\phi}{\partial x_{j}}\,dx\,dt
	\end{equation}
	and 
	\begin{equation}\label{weaksolution.eqn4}
	\int_{\Omega}u^{\varepsilon}(x,0)\,\phi(x,0)\,dx=\int_{\Omega}u_{0\varepsilon}(x)\,\phi(x,0)\,dx\to \int_{\Omega} u(x,0)\,\phi(x,0)\,dx
	\end{equation}
	follow similarly.
	Finally, for $j=1,2,\cdots,d$, we show that
	\begin{equation}\label{eqnchap504}
	\displaystyle\lim_{\varepsilon\to 0}\varepsilon\,\int_{0}^{T}\,\int_{\Omega}\,B(u^{\varepsilon})\,\frac{\partial u^{\varepsilon}}{\partial x_{j}}\,\frac{\partial \phi}{\partial x_{j}}\,dx\,dt = 0.
	\end{equation}
	Note that
	\begin{eqnarray}\label{chap5eqn13}
	\Big|\varepsilon\,\int_{0}^{T}\,\int_{\Omega}\,B(u^{\varepsilon})\,\frac{\partial u^{\varepsilon}}{\partial x_{j}}\,\frac{\partial \phi}{\partial x_{j}}\,dx\,dt\Big|&\leq& \|B\|_{L^{\infty}(I)}\varepsilon\,\int_{0}^{T}\int_{\Omega}\Big|\frac{\partial u^{\varepsilon}}{\partial x_{j}}\Big|\,\Big|\frac{\partial\phi}{\partial x_{j}}\Big|\,dx\,dt,\nonumber\\
	&\leq& \|B\|_{L^{\infty}(I)}\sqrt{\varepsilon}\left(\sqrt{\varepsilon}\Big\|\frac{\partial u^{\varepsilon}}{\partial x_{j}}\Big\|_{L^{2}(\Omega_{T})}\right)\Big\|\frac{\partial\phi}{\partial x_{j}}\Big\|_{L^{2}(\Omega_{T})}.
	\end{eqnarray}
	Using boundedness of $\left\{\sqrt{\varepsilon}\Big\|\frac{\partial u^{\varepsilon}}{\partial x_{j}}\Big\|_{L^{2}(\Omega_{T})}\right\}_{\varepsilon\geq 0}$ in \eqref{chap5eqn13}, we get \eqref{eqnchap504}.\\
	Equations \eqref{weaksolution.eqn1},\eqref{eqnchap504},\eqref{weaksolution.eqn2},\eqref{weaksolution.eqn4} together give 
	\begin{equation}\label{eqnchap512}
	\int_{0}^{T}\,\int_{\Omega}u\,\frac{\partial \phi}{\partial t}\,dx\,dt +\displaystyle\sum_{j=1}^{d}\int_{0}^{T}\,\int_{\Omega} f_{j}(u)\,\frac{\partial \phi}{\partial x_{j}}\,dx \,dt =\int_{\Omega}u_{0}(x)\,\phi(x,0)\,dx\,dt.
	\end{equation}
	Therefore $u$ is a weak solution of initial value problem for conservation law \eqref{ivp.cl}.\\
	\vspace{0.1cm}\\
	\textbf{Step 2:} We work with Kruzhkov's entropy pairs as any convex function belongs to the convex hull of the set of all affine functions of the form 
	$$u\mapsto |u-k|.$$
	For $k\in\mathbb{R}$, Kruzhkov's entropy pairs are given by
	\begin{eqnarray}\label{chap70eqn2}
	\eta(u) &=& |u-k|,\nonumber\\
	q(u) &=& \mbox{sg}(u-k)\,\,\left(f(u)-f(k)\right),
	\end{eqnarray}
	where for $j=1,2,\cdots,d$, $\eta$ and $q =\left(q_{1},q_{2},\cdots,q_{d}\right)$ satisfy the following relation
	$$\eta^{'}(u)f_{j}^{'}(u)= q_{j}^{'}(u).$$
	Multiplying \eqref{regularized.IBVP.a} by $\eta^{'}(u^{\varepsilon})$ and applying chain rule, we get
	\begin{equation}\label{chap70eqn4}
	\frac{\partial \eta(u^{\varepsilon})}{\partial t} + \displaystyle\sum_{j=1}^{d}\frac{\partial q_{j}(u^{\varepsilon})}{\partial x_{j}}  = \varepsilon\, \eta^{'}(u^{\varepsilon}) \frac{\partial }{\partial x_{j}}\left(B(u^{\varepsilon})\frac{\partial u^{\varepsilon}}{\partial x_{j}}\right). 
	\end{equation}
	Let $\phi\in \mathcal{D}(\Omega_{T})$ and $\phi\geq 0$. Multiplying \eqref{chap70eqn4} by $\phi$, integrating over $\Omega_{T}$ and using integration by parts formula, we obtain
	\begin{equation}\label{chap70eqn5}
	-\int_{\Omega_{T}}\eta(u^{\varepsilon})\frac{\partial\phi}{\partial t}\,dx dt -\displaystyle\sum_{j=1}^{d}\int_{\Omega_{T}} q_{j}(u^{\varepsilon})\frac{\partial\phi}{\partial x_{j}}\,dx\,dt = \varepsilon\,\displaystyle\sum_{j=1}^{d}\int_{\Omega_{T}}\eta^{'}(u^{\varepsilon})\phi \frac{\partial }{\partial x_{j}}\left(B(u^{\varepsilon})\frac{\partial u^{\varepsilon}}{\partial x_{j}}\right)\,dx dt.
	\end{equation}
	Taking $\liminf$ as $\varepsilon\to 0$ on both sides of \eqref{chap70eqn5}, we have
	\begin{equation}\label{chap70eqn6}
	-\int_{\Omega_{T}}\eta(u)\frac{\partial\phi}{\partial t}\,dx dt -\displaystyle\sum_{j=1}^{d}\int_{\Omega_{T}} q_{j}(u)\frac{\partial\phi}{\partial x_{j}}\,dx dt = \displaystyle\liminf_{\varepsilon\to 0}\displaystyle\sum_{j=1}^{d}\left(\varepsilon\int_{\Omega_{T}} \eta^{'}(u^{\varepsilon})\phi \frac{\partial }{\partial x_{j}}\left(B(u^{\varepsilon})\frac{\partial u^{\varepsilon}}{\partial x_{j}}\right)\,dx dt\right).
	\end{equation}
	Let $G:\mathbb{R}^{d}\to\mathbb{R}$ be a $C^{\infty}$ function such that 
	$$G(x) = |x|\,\,\mbox{for} |x|\geq 1\,\,\mbox{and}\,\,G^{''}\geq 0.$$
	For $\delta > 0$, define $G_{\delta}:\mathbb{R}\to\mathbb{R}$ by 
	$$G_{\delta}(y) = \delta G\left(\frac{(y-k)}{\delta}\right)$$
	so that $G_{\delta}(y)\to |y-k|$ as $\delta\to 0$, see\cite[p.72]{MR1304494}. The function $G_{\delta}$ is a convex function and $G^{'}_{\delta}(v)\to \mbox{sg}(v-k)$ as $\delta\to 0$. \\
	For $j=1,2,\cdots,d$, note that
	\begin{eqnarray}\label{chap70eqn7}
	\varepsilon \int_{\Omega_{T}} \eta^{\prime}(u^{\varepsilon})\phi \frac{\partial }{\partial x_{j}}\left(B(u^{\varepsilon})\frac{\partial u^{\varepsilon}}{\partial x_{j}}\right)\,dx dt &=&\varepsilon\displaystyle\lim_{\delta\to 0}\int_{\Omega_{T}} G_{\delta}^{'}(u^{\varepsilon})\phi \frac{\partial }{\partial x_{j}}\left(B(u^{\varepsilon})\frac{\partial u^{\varepsilon}}{\partial x_{j}}\right)\,dx dt.\nonumber\\
	&=& -\varepsilon \displaystyle\displaystyle\lim_{\delta\to 0}\int_{\Omega_{T}} G_{\delta}^{''}(u^{\varepsilon})\phi B(u^{\epsilon})\left( \frac{\partial u^{\varepsilon}}{\partial x_{j}}\right)^{2}\,dx dt\nonumber\\
	&-&\varepsilon \displaystyle\lim_{\delta\to 0}\int_{\Omega_{T}} G_{\delta}^{'}(u^{\varepsilon})\frac{\partial\phi}{\partial x_{j}} B(u^{\varepsilon})\frac{\partial u^{\varepsilon}}{\partial x_{j}}\, dx dt.
	\end{eqnarray}
	Since $\left\{\sqrt{\varepsilon}\|\frac{\partial u^{\varepsilon}}{\partial x_{j}}\|_{L^{2}(\Omega_{T})}\right\}_{\varepsilon\geq 0}$ is bounded, for $j=1,2,\cdots,d$, we can easily deduce that
	\begin{equation*}
	\Big|-\varepsilon \displaystyle\lim_{\delta\to 0}\int_{\Omega_{T}} G_{\delta}^{'}(u^{\varepsilon})\frac{\partial\phi}{\partial x_{j}} B(u^{\varepsilon})\frac{\partial u^{\varepsilon}}{\partial x_{j}}\, dx dt.\Big|\to 0\,\,\mbox{as}\,\,\varepsilon\to 0.
	\end{equation*}
	Therefore, we have 
	\begin{equation}\label{entropy.solution.eqn1}
	\Big|-\varepsilon\displaystyle\sum_{j=1}^{d}\displaystyle\lim_{\delta\to 0}\int_{\Omega_{T}} G_{\delta}^{'}(u^{\varepsilon})\frac{\partial\phi}{\partial x_{j}} B(u^{\varepsilon})\frac{\partial u^{\varepsilon}}{\partial x_{j}}\, dx dt.\Big|\to 0\,\,\mbox{as}\,\,\varepsilon\to 0.
	\end{equation}
	Observe that
	\begin{equation}\label{entropy.solution.eqn2}
	-\displaystyle\limsup_{\varepsilon\to 0}\left(\varepsilon\displaystyle\sum_{j=1}^{d} \displaystyle\lim_{\delta\to 0}\int_{\Omega_{T}} G_{\delta}^{''}(u^{\varepsilon})\phi B(u^{\epsilon})\left( \frac{\partial u^{\varepsilon}}{\partial x_{j}}\right)^{2}\,dx dt\right)\leq 0.
	\end{equation}
	Using \eqref{entropy.solution.eqn1}, \eqref{entropy.solution.eqn2} in \eqref{chap70eqn6}, we have required inequality \eqref{definition.otto.eqn1}.
\end{proof}\\
\vspace{0.1cm}\\
\textbf{Step 3:} The following results are useful in proving Theorem \ref{Boundaryinequality.F.Ottothm1} and Lemma \ref{Boundaryinequality.F.Ottolem3}.
\begin{lemma}\cite[p.105]{Necas}\label{Boundaryinequality.F.Ottolem1}
Let $\delta >0$ be a real number and $g\in L^{\infty}\left((0,\delta)\right)$. Then the following 
\begin{enumerate}
\item[(1)] $\displaystyle\liminf_{n\to\infty}\,\left(n\int_{0}^{\frac{1}{n}}g(t)\,dt\right)\geq\,\mbox{ess}\displaystyle\liminf_{t\to 0+}\,g(t)$,
\item[(2)] $\displaystyle\limsup_{n\to\infty}\,\left(n\int_{0}^{\frac{1}{n}}g(t)\,dt\right)\leq\,\mbox{ess}\displaystyle\limsup_{t\to 0+}\,g(t)$,
\item[(3)] $\displaystyle\lim_{n\to\infty}\,\left(n\int_{0}^{\frac{1}{n}}g(t)\,dt\right)=\,\mbox{ess}\displaystyle\lim_{t\to 0+}\,g(t)$
\end{enumerate}
hold.
\end{lemma}
\begin{lemma}\label{Boundaryinequality.F.Ottolem2}
	Let $\delta >0$ be a real number and $g\in L^{\infty}\left((-\delta,0)\right)$. Then the following 
	\begin{enumerate}
		\item[(1)] $\displaystyle\liminf_{n\to\infty}\,\left(n\,\int_{-\frac{1}{n}}^{0}g(t)\,dt\right)\geq\,\mbox{ess}\displaystyle\liminf_{t\to 0+}\,g(t)$,
		\item[(2)] $\displaystyle\limsup_{n\to\infty}\,\left(n\,\int_{-\frac{1}{n}}^{0}g(t)\,dt\right)\leq\,\mbox{ess}\displaystyle\limsup_{t\to 0+}\,g(t)$,
		\item[(3)] $\displaystyle\lim_{n\to\infty}\,\left(n\,\int_{-\frac{1}{n}}^{0}g(t)\,dt\right)=\,\mbox{ess}\displaystyle\lim_{t\to 0+}\,g(t)$
	\end{enumerate}
	hold.
\end{lemma}
The proof of Lemma\ref{Boundaryinequality.F.Ottolem2} is similar to the proof of Lemma \ref{Boundaryinequality.F.Ottolem1}. 
\begin{theorem}\label{Boundaryinequality.F.Ottothm1}
	Let $\left(H,Q\right)$ be boundary entropy-entropy flux pair. Then for all $\phi\in\mathcal{D}(\mathbb{R}^{d}\times\mathbb{R})\, ,\phi\geq 0$ and $k\in\mathbb{R}$, the following inequality
	\begin{equation}\label{Boundaryinequality.F.Otto.eqn111}
	\begin{split}
	-\left\{\int_{0}^{T}\int_{\Omega}H(u,k)\frac{\partial\phi}{\partial t}\,dx\,dt +\displaystyle\sum_{j=1}^{d}\int_{0}^{T}\int_{\Omega}Q_{j}(u,k)\,\frac{\partial\phi}{\partial x_{j}}\,dx\,dt\right\}\leq \int_{\Omega}H(u_{0},k)\,\phi(x,0)\,dx\\
	-\displaystyle\sum_{j=1}^{d}\mbox{ess}\displaystyle\liminf_{s\to 0-}\int_{0}^{T}\int_{\partial\Omega}Q_{j}(u(r+s\,\nu),k)\nu_{j}\,\phi(r)\,dr
	\end{split}
	\end{equation}
	holds.
\end{theorem}
A proof of Theorem \ref{Boundaryinequality.F.Ottothm1} can be found on the lines of \cite[p.106]{Necas}.\\
\vspace{0.1cm}\\
In order to obtain inequality \eqref{definition.otto.eqn2}, we need to prove the following result. 
\begin{lemma}\label{Boundaryinequality.F.Ottolem3}
Let $\left(H,Q\right)$ be a boundary entropy-entropy flux pair and $u\in L^{\infty}\left(\Omega_{T}\right)$. For all $\gamma\in\mathcal{D}(\Omega_{T})$, $\gamma\geq 0$ and all $k\in\mathbb{R}$, $u$ satisfy 
\begin{equation}\label{F.Ottoentropy.N.eqn12}
\int_{0}^{T}\int_{\Omega}\left\{H(u,k)\,\frac{\partial\gamma}{\partial t}\,+\,\displaystyle\sum_{0}^{T}Q_{j}(u,k)\,\frac{\partial\gamma}{\partial x_{j}}\right\}\,dx\,dt\geq 0.
\end{equation}
Then there exists a set $E$ of Lebesgue measure zero such that for all $\beta\in L^{1}(\Gamma)$, $\beta\geq 0$ and $v^{D}\in L^{\infty}(\Gamma)$
$$\displaystyle\lim_{s\to 0}\int_{\Gamma}Q(u\left(r+s\nu(r)\right),v^{D})\cdot\,\nu(r)\,\beta(r)\,dr$$
exists for $s\notin E$.\\
Moreover, let $M> 0$ be constant of Lipschitz continuity of $f$. For all $\gamma\in\mathcal{D}\left(\mathbb{R}^{d}\times (0,T)\right)$, $\gamma\geq 0$, all $k\in\mathbb{R}$ and for some $u^{D}$, $u$ satisfies 
\begin{equation}\label{F.Ottoentropy.N.eqn13}
-\int_{0}^{T}\int_{\Omega}\left\{H(u,k)\,\frac{\partial\gamma}{\partial t}\,+\,\displaystyle\sum_{0}^{T}Q_{j}(u,k)\,\frac{\partial\gamma}{\partial x_{j}}\right\}\,dx\,dt\leq M\,d\,\int_{\Gamma}\,H(u^{D},k)\,\gamma\,dr.
\end{equation}
Then for all $\beta\in L^{1}(\Gamma)$, $\beta\geq 0$, the following
\begin{equation}\label{F.Ottoentropy.N.eqn14}
\mbox{ess}\,\displaystyle\lim_{s\to 0}\int_{\Gamma}Q(u\left(r+s\nu(r)\right),v^{D})\cdot\,\nu(r)\,\beta(r)\,dr\geq 0
\end{equation} 
almost everywhere $s\in (-\infty,0)$.
\end{lemma}
The main difficulty in proving Lemma \ref{Boundaryinequality.F.Ottolem3} is to prove inequality \eqref{F.Ottoentropy.N.eqn13} which is satisfied by the boundary entropy-entropy flux pair. We overcome this difficulty by proving the following result. \begin{lemma}\label{Boundaryinequality.F.Ottolem41}
	Let $\left(H,Q\right)$ be boundary entropy-entropy flux pair and $u\in L^{\infty}(\Omega_{T})$. Then
	\begin{enumerate}
		\item[(1)] for all $\phi\in\mathcal{D}(\mathbb{R}^{d}\times\mathbb{R})$, $\phi\geq 0$ and $k\in\mathbb{R}$, the following inequality
		\begin{equation}\label{F.Ottoentropy.N.eqn29}
		\begin{split}
		-\left\{\int_{0}^{T}\int_{\Omega}\left(H(u,k)\,\frac{\partial\phi}{\partial t} + \displaystyle\sum_{j=1}^{d}Q_{i}(u,k)\,\frac{\partial\phi}{\partial x_{j}}\right)\right\}\,dx\,dt\leq \int_{\Omega}\,H(u_{0},k)\,\phi(x,0)\,dx \\+ Md\,\int_{\Gamma}H(u,k)\,\phi\,dr.
		\end{split}
		\end{equation}
		holds. 
		\item[(2)] for all $\phi\in\mathcal{D}(\mathbb{R}^{d}\times(0,T))$, $\phi\geq 0$ and $k\in\mathbb{R}$, the following inequality
		\begin{equation}\label{F.Ottoentropy.N.eqn30}
		\begin{split}
		-\left\{\int_{0}^{T}\int_{\Omega}\left(H(u,k)\,\frac{\partial\phi}{\partial t} + \displaystyle\sum_{j=1}^{d}Q_{i}(u,k)\,\frac{\partial\phi}{\partial x_{j}}\right)\right\}\,dx\,dt\leq M\,d\,\int_{\Gamma}H(u,k)\,\phi\,dr.
		\end{split}
		\end{equation}
		holds. 
	\end{enumerate}	
\end{lemma}	
\begin{proof}
	For $k\in\mathbb{R}$ and $j=1,2,\cdots,d$, denote 
	$$\eta(z):=H(z,k)\,\, ,\,\,q_{j}(z):=Q_{j}(z,k)\,\,\mbox{and}\,\,q:=\left(q_{1},q_{2},\cdots,q_{d}\right).$$
	Then for $j=1,2,\cdots,d$ and $z\in\mathbb{R}$, $\left(\eta,q\right)$ satisfies 
	$$q_{j}^{\prime}(z)=\eta^{\prime}(z)f_{j}^{\prime}(z).$$
	Since $q_{j}(k)=0$, therefore $q_{j}(z)=\int_{k}^{z}\,\eta^{\prime}(s)f_{j}^{\prime}(s)\,ds$. 
	For $j=1,2,\cdots,d$, observe that
	\[
	q_{j}(z)=
	\begin{cases}
	\int_{k}^{z}\,\eta^{\prime}(s)f_{j}^{\prime}(s)\,ds & \text{if $z\geq k$}\\
	-\int_{z}^{k}\,\eta^{\prime}(s)f_{j}^{\prime}(s)\,ds & \text{if $z < k$}
	\end{cases}
	\]	
	Since $\eta^{\prime\prime}(z)\geq 0$ for $z\in\mathbb{R}$, therefore $\eta^{\prime}:\mathbb{R}\to\mathbb{R}$ is an increasing function. Since $H$ is a boundary entropy, we have
	$\eta^{\prime}(k)=0$. Therefore for $z\geq k$, we have $\eta^{\prime}(z)\geq 0$ and for $z\leq k$, we have $\eta^{\prime}(z)\leq 0$. Therefore we obtain $-\eta^{\prime}(z)\geq 0$ for $z<k$.
	For $\beta\in L^{1}(\Gamma)$, $\beta\geq 0$ and {\it a.e.} $r\in\Gamma$, observe that
	\[
	\left|q_{j}(u(r +s\nu(r)))\,\nu_{j}(r)\,\beta(r)\right|\,\leq
	\begin{cases}
	M\,\left(\int_{k}^{u(r +s\nu(r))}\,\eta^{\prime}(s)\,ds\right)\beta(r) & \text{if $u(r +s\nu(r))\geq k$}\\
	M\left(\int_{u(r +s\nu(r))}^{k}\,-\eta^{\prime}(s)\,ds\right)\beta(r) & \text{if $u(r +s\nu(r)) < k$}
	\end{cases}
	\] 
	Therefore we have
	$$\left|q_{j}(u(r +s\nu(r)))\,\nu_{j}(r)\,\beta(r)\right|\leq M\,\eta\left(u(r+s\nu(r))\right)\,\beta(r).$$
	This shows that for $j=1,2,\cdots,d$, the following holds:
	\begin{equation}\label{F.Ottoentropy.N.eqn32}
	q_{j}(u(r +s\nu(r)))\,\nu_{j}(r)\,\beta(r) \geq - M\,\eta\left(u(r+s\nu(r))\right)\,\beta(r).
	\end{equation}
	Integrating \eqref{F.Ottoentropy.N.eqn32} over $\Gamma$, we get
	\begin{equation}\label{F.Ottoentropy.N.eqn33}
	\int_{\Gamma}\,q_{j}(u(r +s\nu(r)))\,\nu_{j}(r)\,\beta(r)\,dr \geq - M\,\int_{\Gamma}\eta\left(u(r+s\nu(r))\right)\,\beta(r)\,dr.
	\end{equation}
	Taking $\mbox{ess}\,\liminf$ as $s\to 0-$ on both sides of \eqref{F.Ottoentropy.N.eqn33}, we obtain
	\begin{equation}\label{F.Ottoentropy.N.eqn34}
	\mbox{ess}\,\displaystyle\liminf_{s\to 0-}\int_{\Gamma}\,q_{j}(u(r +s\nu(r)))\,\nu_{j}(r)\,\beta(r)\,dr \geq - \mbox{ess}\,\displaystyle\limsup_{s\to 0-}\int_{\Gamma}\,M\eta\left(u(r+s\nu(r))\right)\,\beta(r)\,dr.
	\end{equation}
	Summing \eqref{F.Ottoentropy.N.eqn34} over $j=1,2,\cdots,d$, we have
	\begin{equation}\label{F.Ottoentropy.N.eqn35}
	-\displaystyle\sum_{j=1}^{d}\mbox{ess}\,\displaystyle\liminf_{s\to 0-}\int_{\Gamma}\,q_{j}(u(r +s\nu(r)))\,\nu_{j}(r)\,\beta(r)\,dr \leq M\,d\,\int_{\Gamma}\,\eta\left(u(r)\right)\,\beta(r)\,dr.
	\end{equation}
	Using \eqref{F.Ottoentropy.N.eqn35}in \eqref{Boundaryinequality.F.Otto.eqn111}, we conclude inequality \eqref{F.Ottoentropy.N.eqn29}.\\
	\vspace{0.1cm}\\
	In particular taking $\phi\in\mathcal{D}(\mathbb{R}^{d}\times (0,T))$ in \eqref{F.Ottoentropy.N.eqn29}, we conclude \eqref{F.Ottoentropy.N.eqn30}$\blacksquare$\\
\end{proof}
\vspace{0.1cm}\\
\textbf{Proof of Lemma \ref{Boundaryinequality.F.Ottolem3}:}
Since we have inequality \eqref{F.Ottoentropy.N.eqn13} in Lemma \ref{Boundaryinequality.F.Ottolem41}, therefore a proof of Lemma \ref{Boundaryinequality.F.Ottolem3} can be obtained by following similar proof of a result from \cite[p.115]{Necas}. The only difference in the proof of Lemma \ref{Boundaryinequality.F.Ottolem3} is that we replace the Lipschitz constant $M$ in an inequality of \cite[p.115]{Necas} by another Lipschitz constant $M\,d$ on RHS of inequality \eqref{F.Ottoentropy.N.eqn13}$\blacksquare$\\
\vspace{0.1cm}\\
The following result is used to show that the limit $u$ satisfies the initial data in $L^{1}(\Omega)$ sense.
\begin{lemma}\cite[p.118]{Necas}\label{Boundaryinequality.F.Ottolem4}
Let $\left(u,u_{0}\right)\in L^{\infty}(\Omega_{T})\times L^{\infty}(\Omega)$ satisfy
\begin{equation}\label{F.Ottoentropy.N.eqn281}
-\int_{0}^{T}\int_{\Omega}\left\{\left|u-k\right|\,\frac{\partial\beta}{\partial t} +\displaystyle\sum_{j=1}^{d}F_{j}(u)\,\frac{\partial\beta}{\partial x_{j}}\right\}\,dx\,dt\leq \int_{\Omega}\left|u_{0}-k\right|\beta(0)\,dx	
\end{equation}	
for all $\beta\in\mathcal{D}((-\infty,T)\times\Omega)$ with $\beta\geq 0$ and all $k\in\mathbb{R}$. Then we have
$$\mbox{ess}\displaystyle\lim_{t\to 0+}\int_{\Omega}\left|u(x,t)-u_{0}(x)\right|\,dx=0$$
\end{lemma}
\vspace{0.1cm}
\textbf{Step 4:} Applying Lemma \ref{Boundaryinequality.F.Ottolem3} with  $u^{D}=0$, we conclude \eqref{definition.otto.eqn2}.\\
\vspace{0.1cm}\\
\textbf{Step 5:} Observe that \eqref{F.Ottoentropy.N.eqn29} holds for $\left(\eta_{l},q_{l}\right)$ introduced in Remark\ref{Boundary.Entropy.F.Otto.rmk1}. Then passing to the limit as $l\to\infty$, we conclude \eqref{F.Ottoentropy.N.eqn281}. Applying Lemma \ref{Boundaryinequality.F.Ottolem4}, we conclude 
\eqref{definition.otto.eqn3}. \\
Therefore $u$ is an entropy solution in the sense of Otto \cite{MR1387428}.\\
\vspace{0.1cm}\\
\textbf{Step 6:} The uniqueness of entropy solutions follows from \cite[p.113]{Necas}$\blacksquare$\\
\vspace{0.1cm}\\
\textbf{Proof of Theorem\ref{paper2.compensatedcompactness.theorem2}:} The proof of Theorem\ref{paper2.compensatedcompactness.theorem2} is similar to the proof of Theorem\ref{paper2.compensatedcompactness.theorem1}$\blacksquare$
\vspace{0.1cm}
\textbf{Proof of Theorem\ref{paper2.compensatedcompactness.theorem3}:} The proof of Theorem\ref{paper2.compensatedcompactness.theorem3} is similar to the proof of Theorem\ref{paper2.compensatedcompactness.theorem1}$\blacksquare$\\
\vspace{0.1cm}\\
\textbf{Acknowledgement: The Author Ramesh Mondal thanks his Ph.D supervisor S. Sivaji Ganesh of Indian Institute of Technology Bombay for fruitful discussions and suggestions throughout this work. All the discussions and suggestions made the author write this article. Ramesh Mondal also thanks CSIR, New Delhi, India as the work of section 3 was done when he was a CSIR SRF at Indian Institute of Technology Bombay, India.}

\newpage
\bibliographystyle{plain}

\end{document}